\documentclass[12pt]{amsart}
\usepackage{amsfonts}
\usepackage{amssymb}
\usepackage{amsmath}
\usepackage{amsthm}
\usepackage{multirow}
\usepackage{array}
\usepackage[font={footnotesize}]{caption}
\usepackage{subcaption}
\usepackage{float}
\usepackage{cite}
\usepackage{graphicx}
\usepackage{stmaryrd}
\usepackage{amscd}
\usepackage{color}
\usepackage[normalem]{ulem}
\newtheorem{theorem}{Theorem}
\theoremstyle{plain}

\newtheorem{definition}[theorem]{Definition}

\newtheorem{corollary}[theorem]{Corollary}
\newtheorem{lemma}[theorem]{Lemma}

\newtheorem{proposition}[theorem]{Proposition}

\numberwithin{equation}{section}

\def\eps{\varepsilon}

\def\ri{{\rm i}}

\newcommand{\R}{\mathbb{R}}
\newcommand{\C}{\mathbb{C}}

\newcommand{\N}{\mathbb{N}}

\newcommand{\Z}{\mathbb{Z}}

\renewcommand{\phi}{\varphi}

\newcommand{\bpm}{{\begin{pmatrix}}}
\renewcommand{\phi}{\varphi}

\DeclareMathOperator{\Id}{Id}

\DeclareMathOperator{\spann}{span}

\DeclareMathOperator{\sign}{sign}
\DeclareMathOperator{\dom}{dom}

\def\blem{\begin{lemma}}\def\elem{\end{lemma}}
\def\bthm{\begin{theorem}}\def\ethm{\end{theorem}}
\def\bcor{\begin{corollary}}\def\ecor{\end{corollary}}
\def\beq{\begin{equation}}\def\eeq{\end{equation}}
\def\sm{\setminus}

\newcommand{\matII}[4]{
\ensuremath{ 
\begin{pmatrix}
#1 & #2 \\
#3 & #4 \\
\end{pmatrix}}}

\newcommand{\vecII}[2]{
\ensuremath{
\begin{pmatrix}
#1 \\ #2 \\
\end{pmatrix}}}
\begin{document}

\title[]{A priori bounds and global bifurcation results for frequency combs modeled by the Lugiato-Lefever
equation}

\author{Rainer Mandel}
\address{R. Mandel \hfill\break 
Scuola Normale Superiore di Pisa, \hfill\break
I-56126 Pisa, Italy}

\curraddr{Institute for Analysis, Karlsruhe Institute of Technology (KIT), \hfill\break
D-76128 Karlsruhe, Germany}
\email{Rainer.Mandel@kit.edu}

\author{Wolfgang Reichel}
\address{W. Reichel \hfill\break 
Institute for Analysis, Karlsruhe Institute of Technology (KIT), \hfill\break
D-76128 Karlsruhe, Germany}
\email{wolfgang.reichel@kit.edu}

\date{\today}


\begin{abstract}
  In nonlinear optics $2\pi$-periodic solutions $a\in C^2([0,2\pi];\C)$ of the stationary Lugiato-Lefever
  equation $-d a''= (\ri -\zeta)a +|a|^2a-\rm i f$ serve as a model for frequency combs, which are optical
  signals consisting of a superposition of modes with equally spaced frequencies. We prove that nontrivial
  frequency combs can only be observed for special ranges of values of the forcing and
  detuning parameters $f$ and $\zeta$, as it has been previously documented in experiments and numerical
  simulations. E.g., if the detuning parameter $\zeta$ is too large then nontrivial frequency combs do not
  exist, cf. Theorem~\ref{Thm 2 nonexistence}. Additionally, we show that for large ranges of parameter
  values nontrivial frequency combs may be found on continua which bifurcate from curves of trivial frequency
  combs. Our results rely on the proof of a priori bounds for the stationary Lugiato-Lefever equation as well
  as a detailed rigorous bifurcation analysis based on the bifurcation theorems of Crandall-Rabinowitz and
  Rabinowitz. We use the software packages AUTO and MATLAB to illustrate our results by numerical
  computations of bifurcation diagrams and of selected solutions.
\end{abstract}

\maketitle


\section{Introduction} 

  In physics literature an optical signal is called a frequency comb if it consists of a superposition of
  modes with equally spaced frequencies. By a suitable choice of reference frame a frequency comb becomes
  stationary (time-independent). For $k\in\Z$ let $\hat a_k$ denote the complex amplitude of the $k$-th mode of the signal in
  the frequency domain and let $a(x)= \sum_{k\in \Z} \hat a_k  e^{ \ri kx}$ be the associated Fourier series.
  A commonly used mathematical model is given by the stationary Lugiato-Lefever equation
  \begin{align}\label{Gl freqcomb} 
    \begin{aligned}
      -d a_1'' & = -a_2-\zeta a_1 + (a_1^2+a_2^2)a_1, \\
      -d a_2'' & = a_1-\zeta a_2 +(a_1^2+a_2^2)a_2-f, \\
      & a_1, a_2 \;\; 2\pi\mbox{-periodic} 
    \end{aligned}
  \end{align}
  where the parameters satisfy $d\neq 0$ and $\zeta,f\in\R$ and $a(x,t)=a_1(x) + \ri a_2(x)$, i.e., $a$ is
  split into its real and imaginary part. Derivations of \eqref{Gl freqcomb} may be found, e.g., in
  \cite{chembo_2010,herr_2012}. More details on the physical background and the meaning of the parameters are
  given in Section~\ref{physics}. Equation \eqref{Gl freqcomb} is a nonvariational version of the stationary
  nonlinear Schr\"odinger equation with added damping and forcing. Our analysis and our results are based on the investigation of \eqref{Gl freqcomb} from the point of view of bifurcation from trivial (i.e. spatially constant) solutions.


  \subsection{Mathematical context and main results}
     
  One first notices that \eqref{Gl freqcomb} has trivial, i.e. spatially constant, solutions. They correspond to vanishing amplitudes of all modes except the $0$-mode and are
  described in detail in Lemma~\ref{trivial} below. Stable, spatially periodic patterns bifurcating from
  trivial solutions of \eqref{Gl freqcomb} were already observed in \cite{Lugiato_Lefever1987}. Recently, a
  more far-reaching bifurcation analysis appeared in \cite{Godey_et_al2014} (see also \cite{Godey_2016} for a detailed mathematical analysis) where the differential equation
  in \eqref{Gl freqcomb} is considered as a four-dimensional dynamical system in the unknowns $a_1, a_1',
  a_2, a_2'$. As the parameters $f$ and $\zeta$ change, the trivial solutions exhibit various bifurcation
  phenomena. This approach allows an extensive account of various possible types of solutions of the
  differential equations in \eqref{Gl freqcomb}. However, with this approach the $2\pi$-periodicity of the
  solutions may be lost. An emanating solution with spatial period $\tau$ has to be rescaled to the
  fixed period $2\pi$ and as a result \eqref{Gl freqcomb} is changed, e.g., $d$ becomes $(2\pi)^2 d/\tau^2$.
  A different view on bifurcation has been developed in \cite{Miyaji_Ohnishi_Tsutsumi2010}. Here spatially
  $2\pi$-periodic solutions of \eqref{Gl freqcomb} are considered via a bifurcation approach using the
  center manifold reduction. The resulting picture is very detailed in the vicinity of special parameter
  values but beyond these values nothing seems to be known -- a gap in the literature which we would like to
  fill with the present paper.

  \medskip

  Similarly to the above-mentioned papers we use bifurcation theory to prove the existence of frequency combs
  bifurcating from the set of spatially constant solutions. Let us therefore
  point out the main features which distinguish our approach from the previous ones. Unlike
  \cite{Godey_et_al2014, Godey_2016} we consider \eqref{Gl freqcomb} on certain spaces of $2\pi$-periodic functions. We obtain a very
  rich bifurcation picture which is not limited to local considerations as in
  \cite{Godey_et_al2014, Godey_2016, Miyaji_Ohnishi_Tsutsumi2010}. In Theorem~\ref{Thm 1 A priori estimates}
  and Theorem~\ref{Thm 2 nonexistence} we find a priori bounds and uniqueness results which allow us to show
  that
  \begin{itemize}
  \item[(a)] nonconstant solutions of \eqref{Gl freqcomb} only occur in the range $\sign(d)\zeta\in [\zeta_*,
  \zeta^*]$,
  \item[(b)] nonconstant solutions of \eqref{Gl freqcomb} satisfy $\|a\|_\infty+|\zeta|\leq C$.
  \end{itemize}
  Here, the values $\zeta_*, \zeta^*$ and $C$ are explicit and only depend on the parameters $f,d$.
  We begin with our results concerning pointwise a priori bounds for solutions of \eqref{Gl freqcomb} in terms
  of the parameters $\zeta,d,f$.
  
  \medskip

  \begin{theorem} \label{Thm 1 A priori estimates}
    Let $d\neq 0$, $f,\zeta\in\R$. Every solution $a\in C^2([0,2\pi],\R^2)$ of \eqref{Gl freqcomb} satisfies
    \begin{align} \label{Gl A priori est 1}
      \|a\|_{L^\infty} \leq \frac{|f|(1+12\pi^2f^2|d|^{-1})}{\max\{1,-\zeta \sign(d)-\gamma(d,f)\}},
    \end{align}
    where
    \begin{equation} \label{Gl defn gamma(d,f)}
    \gamma(d,f) := \left\{
    \begin{array}{ll}
    36\pi^2f^4|d|^{-1}, & d>0, \vspace{\jot}\\
    36\pi^2f^4|d|^{-1}+f^2(1+12\pi^2f^2|d|^{-1})^2, & d<0.
    \end{array}
    \right.
    \end{equation}
  \end{theorem}

  \medskip
  \noindent
  {\bf Remark.} Further bounds in $L^2([0,2\pi],\R^2)$ and $H^1([0,2\pi],\R^2)$ may be extracted
  from the proof of this theorem.

  \medskip

  In our second result we employ Theorem~\ref{Thm 1 A priori estimates} to show that the set of nonconstant
  solutions is bounded with respect to $\zeta$.
  
  \medskip

  \begin{theorem}\label{Thm 2 nonexistence}
    Let $d\neq 0,f\in\R,\zeta\in\R$ and let $\zeta_*,\zeta^*$ be given by
    \begin{align*}
      \zeta_* &:=  -\gamma(d,f)-\sqrt{6}|f|(1+12\pi^2f^2|d|^{-1}),\\
      \zeta^* &:= 6f^2(1+12\pi^2f^2|d|^{-1})^2.
    \end{align*}
    Then every solution of \eqref{Gl freqcomb} is constant provided $\sign(d)\zeta<\zeta_*$ or
    $\sign(d)\zeta>\zeta^*$.
  \end{theorem}
  
  \medskip

  Next we consider \eqref{Gl freqcomb} from the point of view of bifurcation theory where one of the
  two values $f$ or $\zeta$ is fixed and the other one is the bifurcation parameter. According to the two
  possible choices we identify two curves of constant solutions: $\hat\Gamma_f$ with $f$ being fixed and
  $\bar\Gamma_\zeta$ with $\zeta$ fixed, see Lemma~\ref{trivial} for explicit parametrizations of these
  curves. We investigate branches of nontrivial solutions that bifurcate from the trivial branches
  $\hat\Gamma_f$ and $\bar\Gamma_\zeta$ and obtain information about their global shape.
%
%
 %
   In our approach we consider the following special class of solutions of \eqref{Gl freqcomb}. We call them
   synchronized solutions in order to emphasize that they have a particular shape.
   
   \medskip
   
  \begin{definition} 
    A $2\pi$-periodic solution $a\in C^2([0,2\pi];\C)$ of \eqref{Gl freqcomb} is called synchronized if
    $a'(0)=a'(\pi)=0$.
  \end{definition}
  
  \medskip
  
  \noindent{\bf Remarks.} (a)\; Synchronized solutions are even around $x=0$ and $x=\pi$. The
  advantage of considering synchronized solutions is that the translation invariance of the original equation
  \eqref{Gl freqcomb} is no longer present in this Neumann boundary value problem which makes the bifurcation
  analysis much easier, see also the remark after Proposition~\ref{Prop simple kernels}.
  
  \smallskip
  
  (b)\; It would be interesting to find out whether \eqref{Gl freqcomb} admits solutions which are
  not synchronized. Note that in the case of scalar periodic boundary value problems the restriction to
  homogeneous Neumann boundary conditions on an interval of half the period is natural since (up to a shift)
  all solutions satisfy this condition. However, in the system case this is not clear at all and we have to leave it as an open problem.
  
  \medskip

  Before we state our results let us recall some common notions in bifurcation theory.
  In the context of bifurcation from $\hat\Gamma_f$ (Theorem~\ref{Thm 3 bifurcation
  zeta}) a pair $(a,\zeta)$ is called a trivial solution if it is spatially constant, i.e., if $(a,\zeta)$
  lies on $\hat\Gamma_f$. A trivial solution $(a,\zeta)$ is called a bifurcation point if a sequence
  $(a_k,\zeta_k)_{k\in \N}$ of non-trivial solutions of the periodic system~\eqref{Gl freqcomb} converges to
  $(a,\zeta)$. Similarly trivial solutions and bifurcation points are defined when bifurcation from
  $\bar\Gamma_\zeta$ is investigated, see Theorem~\ref{Thm 4 bifurcation f}.
%
%
  Since our analysis of synchronized solutions is based on the bifurcation theorem of Crandall-Rabinowitz
  \cite{CrRab_bifurcation} such bifurcating non-trivial solutions lie on local curves around the bifurcation
  points. Furthermore, we will use the global bifurcation theorem of Rabinowitz \cite{Rab_some_global} to show
  that these curves are part of a connected set that is unbounded or returns to the curve of trivial solutions
  at some other point. A continuum satisfying one of these two properties will be a called global continuum.
  If, additionally, the nontrivial  (i.e. nonconstant) solutions from this continuum are confined in a bounded
  subset of $L^\infty([0,2\pi],\R^2)\times\R$ then the continuum will be called bounded.

  \medskip

  For fixed $f\in\R$ we find that at most finitely many global continua bifurcate from $\hat\Gamma_f$. These
  continua are bounded and intersect $\hat\Gamma_f$ at another trivial solution. For the reader's convenience
  this is illustrated in the plots of Figures~\ref{list_branches_3}--\ref{solutions_branches_4}. The result
  reads as follows.

  \medskip

  \begin{theorem}\label{Thm 3 bifurcation zeta}
    Let $d\neq 0,f\in\R$. If $|f|<1$ then the curve $\hat\Gamma_f$ (see  Lemma~\ref{trivial}
    (a)) does not contain any bifurcation point for \eqref{Gl freqcomb}. In case $|f|\geq 1$ the following holds:
    \begin{itemize}
      \item[(i)] All bifurcation points are among the points
      $(\hat a_1(t),\hat a_2(t),\hat\zeta(t))$ where the number
      $t\in [-\sqrt{1-|f|^{-2}},+\sqrt{1-|f|^{-2}}]$ satisfies
      \begin{equation}\label{Gl bifpoints Thm3}
	    dk^2
	    = f^2(1-t^2) - \frac{t}{\sqrt{1-t^2}} -\sigma \sqrt{f^4(1-t^2)^2-1}
      \end{equation}
      for some $k\in\N$ and some $\sigma\in\{-1, 1\}$.
    \item[(ii)] The curve $\hat\Gamma_f$ contains at most $\hat k(f)$ bifurcation points for \eqref{Gl
      freqcomb} where
      $$
        \hat k(f):= 2 \big( |d|^{-1}\big(f^2+\sqrt{f^2-1}+\sqrt{f^4-1}\big)\big)^{1/2}.
      $$
      \item[(iii)] If in addition to \eqref{Gl bifpoints Thm3} one has
      \begin{itemize}
        \item[(S)]$-k^2+2d^{-1}(f^2(1-t^2)- t(1-t^2)^{-1/2}) \neq j^2$ for all $j\in
        \N_0\setminus\{k\}$,
        \item[(T)] $4 f^6 t^3(1-t^2)^2+f^4(1-t^2)^{1/2}-2tf^2-(1-t^2)^{-3/2}$ \\
                   $-\sigma\sqrt{f^4(1-t^2)^2-1}\Bigl(4f^4 t^3(1-t^2)+f^2(2t^2-1)(1-t^2)^{-1/2}\Bigr)\neq 0$
      \end{itemize}
      then a global continuum containing nontrivial synchronized solutions emanates from $(\hat a_1(t),\hat
      a_2(t),\hat\zeta(t))$. This continuum is bounded and it returns to $\hat\Gamma_f$ at some other point.
      In a neighbourhood of $(\hat a_1(t),\hat
      a_2(t),\hat\zeta(t))$ the continuum is a curve consisting of $\tfrac{2\pi}{k}$-periodic
      nontrivial synchronized solutions.
    \end{itemize}
  \end{theorem}

    \medskip
  \noindent
  {\bf Remarks.} The numbers $k$ and $t$ in \eqref{Gl bifpoints Thm3} have the following meaning: if \eqref{Gl freqcomb} is linearized at $(\hat a_1(t),\hat a_2(t),\hat\zeta(t))$ then $\alpha^k\cos(kx)$ with $\alpha^k\in \R^2\setminus\{(0,0)\}$ solves this linearized equation. From this fact we can draw several conclusion that are worth to be mentioned. The following remarks (a)-(d) are equally valid in the context of Theorem~\ref{Thm 4 bifurcation f}.

  \smallskip
  
  (a)\; Whenever $t_0$ is a turning point of the curve of trivial solutions, i.e. $\hat\zeta'(t_0)=0$, then \eqref{Gl bifpoints Thm3} is satisfied for $k=0$. Since  in this case typically (S) and (T) also hold a bifurcation from turning points $\hat\Gamma_f$ is predicted which, however, is not visible in any bifucation diagram. The fact that this is not a contradiction to the Crandall-Rabinowitz Theorem is explained in detail at the beginnig of Section~\ref{subsec Det bifpoints}.

  \smallskip

  (b)\; The integer $k$ also represents the number of maxima of the bifurcating solutions close to the bifurcation point. In accordance with our numerical experiments this leads to the rule of thumb that $k$-solitons may be found on branches indexed with $k$. As mentioned in (a), at turning points of the curve of trivial solutions \eqref{Gl bifpoints Thm3} is satisfied for $k=0$ and some $t_0$. Therefore, if $|d|$ is small enough, then \eqref{Gl bifpoints Thm3} typically admits a solution for $k=1$ and some $t_1$ near $t_0$. Therefore, bifurcating branches with 1-solitons typically emanate near turning points of the curve of trivial solutions. Our numerical experiments in Section~\ref{illustrations} confirm this.
  
  \smallskip
    
  (c)\;  Suppose we have a bifurcation point corresponding to some numbers $k, t$. If we consider \eqref{Gl freqcomb} as an equation in the space of $2\pi/k$-periodic synchronized solutions then the Crandall-Rabinowitz Theorem shows the existence of a bifurcating branch consisting entirely of $2\pi/k$-periodic synchronized solutions. This branch returns to the trivial curve at bifurcation points associated to multiples of $k$. A typical example can be found in the left bifurcation diagram of Figure~\ref{bif_diagram2_zeta=10} in Section~\ref{zeta=10,d=-0.2}. Within the broader space of synchronized solutions the branches consisting of $2\pi/k$-periodic solutions persist but additional connections to the trivial branches can exist as the right bifurcation diagram in Figure~\ref{bif_diagram2_zeta=10} shows. 
  
  Another consequence of this observation can be seen in Figure~\ref{bif_diagram_f=1.6} in Section~\ref{f=1.6,d=0.1} for the parameter values $f=1.6, d=0.1$. The theorem predicts bifurcation points precisely for $k\in\{1,2,3,4,5,6,7\}$, see Figure~\ref{list_branches_3}. For $k\in\{4,5,6,7\}$ any integer multiple is larger than $7$. Therefore continua emanating from the bifurcation points associated to $k\in\{4,5,6,7\}$ have to return to the trivial branch at bifurcation points associated with the same $k$. This corresponds to the green, red, blue and pink branches in Figure~\ref{bif_diagram_f=1.6}.

  \smallskip
   
  (d)\; The bifurcation point divides the bifurcating curve from  (iii) into two pieces.
  Each solution on one piece correponds to a solution on the other via a phase shift by $\pi/k$. In
  other words two such solutions $a,\tilde a$ are related by $\tilde a(x)=a(x+\pi/k)$. Since 
  their $L^2$-norms coincide the AUTO plots in Section~\ref{illustrations}
  only show one curve near every bifurcation point. Accordingly, each point on this curve represents two
  solutions with the same $L^2$-norm. 
  
  \medskip

  Now let us state the corresponding result for the bifurcation analysis associated to the family of trivial
  solutions $\bar\Gamma_\zeta$ for given $\zeta\in\R$. In contrast to the above results we find that
  infinitely many global continua emanate from $\bar\Gamma_\zeta$ in the case $d>0$ whereas for $d<0$ again
  only finitely many such continua can exist. The numerical plots from
  Figures~\ref{list_branches_1}--\ref{solitons_f} illustrate our results. Although we do not have a proof for
  the boundedness of the continua in this case, it seems nevertheless plausible in view of the numerical
  plots. 

  \medskip

  \begin{theorem}\label{Thm 4 bifurcation f}
    Let $d\neq 0,\zeta\in\R$.
    \begin{itemize}
    \item[(i)] All bifurcation points on the curve $\bar\Gamma_\zeta$ (see Lemma~\ref{trivial} (b)) are among the points $(\bar a_1(s),\bar a_2(s),\bar f(s))$  where 
    \begin{equation} \label{Gl bifpoints Thm4}
       |s| = \Big( \frac{2}{3}(\zeta+dk^2) -
          \frac{\sigma}{3}\sqrt{(\zeta+dk^2)^2-3} \Big)^{1/2}
    \end{equation}
    for some $k\in\N,\sigma\in\{-1,1\}$ and provided $\zeta+dk^2\geq \sqrt{3}$.
    \item[(ii)]
      If $d<0$ then the curve $\bar\Gamma_\zeta$ contains at most $\bar k(\zeta)$
      bifurcation points for \eqref{Gl freqcomb} where
      $$
        \bar k(\zeta):= 4\big(|d|^{-1}(\zeta-\sqrt{3})_+\big)^{1/2}.
      $$
    \item[(iii)] If in addition to \eqref{Gl bifpoints Thm4} one has
      \begin{itemize}
        \item[(S)] $-k^2+\frac{2}{3}d^{-1}(\zeta+4dk^2-2\sigma\sqrt{(\zeta+dk^2)^2-3})\neq
        j^2$\; for all $j\in \N_0\setminus\{k\}$,
        \item[(T)] $\zeta+dk^2\neq \sqrt{3}$\; and\; $4\zeta+dk^2-2\sigma\sqrt{(\zeta+dk^2)^2-3}\neq 0$\;
        and \\$2\zeta+ 5dk^2-4\sigma\sqrt{(\zeta+dk^2)^2-3}\neq 0$
      \end{itemize}
      then a global continuum containing synchronized and
      $\tfrac{2\pi}{k}$-periodic nontrivial solutions bifurcates from $(\bar
      a_1(s),\bar a_2(s),\bar f(s))$. 
      In a neighbourhood of this point the continuum is a curve consisting of
      $\tfrac{2\pi}{k}$-periodic nontrivial synchronized solutions.
    \end{itemize}
  \end{theorem}
%

  \medskip

  \noindent
  {\bf Remarks.} (a)\; We may restrict our attention to the bifurcations ocurring for positive $f$. Indeed,
  by Theorem \ref{Thm 1 A priori estimates} we have that for $f=0$ only the trivial solution $a=(0,0)$ exists. In
  particular solution continua cannot cross the value $f=0$ so that it is sufficient to analyze the
  bifurcations for positive $f$. The bifurcations for negative $f$ can be found via the discrete symmetry of
  \eqref{Gl freqcomb} given by $(a_1,a_2,f)\mapsto (-a_1,-a_2,-f)$.
  
  \smallskip
  
  \;(b)\; In principle the exceptional points where the conditions (S) and (T) are not satisfied
  could be analyzed using different bifurcation theorems. Suitable candidates for a bifurcation
  theorem in the presence of two-dimensional kernels are the theorems of Healey,
  Kielh\"ofer, Kr\"omer \cite{Kroe_Hea_Kie_bifurcation} and Westreich \cite{We_bifurcation_at}. Bifurcation results without
  transversality condition can be found in \cite{Liu_Shi_Wang_bifurcation}. However, the required amount of
  calculations are far too high to justify the use of these theorems in our situation. For
  the same reason we did not include a detailed analysis of the initial directions of the bifurcation branches which,
  without any theoretical difficulty, may be calculated using the formulas from section~I.6 in Kielh\"ofer's
  book~\cite{Kielh_bifurcation_theory}.



 \subsection{Frequency combs in physics and engineering} \label{physics}

  Currently frequency combs are gaining interest as optical sources for high-speed data transmission where
  the individual comb lines are used as carriers. A high power per combline with the same spectral power distribution is
  important. An experimental set-up for such frequency combs is given by a microresonator which is coupled to an
  optical waveguide under the influence of a single, strong, external laser source that is tuned to a
  resonance wavelength of the device. Inside the resonator the optical intensity is strongly enhanced and
  modes start to interact in a nonlinear way. As a consequence, the primarily excited mode couples with a
  multitude of neighboring modes. This leads to a cascaded transfer of power from the pump to the comb lines.
  Under suitable choice of parameters, a stationary cascade of excited modes can be obtained and results in
  a stable frequency comb with equidistant spectral lines.
  %
  %
  If $\hat a_k(t)$ denotes the dimensionless complex-valued amplitude of the $k$-th mode in the
  microresonator at time $t$ then, following \cite{chembo_2010,herr_2012}, it satisfies
  the following set of coupled differential equations
  \begin{equation}
     \ri\partial_t \hat a_k(t) = (- \ri+\zeta)\hat a_k(t) +d k^2 \hat a_k(t)
    -\sum_{k',k''\in\Z} \hat a_{k'}(t) \hat a_{k''}(t)\overline{\hat a}_{k'+k''-k}(t)+ \ri\delta_{0k}f
    \label{basic.F}
  \end{equation}
  for each $k\in\Z$.
  In this equation, the parameters $\zeta, d, f$ are real and $t$ is normalized time. The term
  $\ri\delta_{0k}f$ corresponds to forcing by the external pump, $\zeta$ represents the normalized frequency
  detuning between the source and the principal resonance of the microresonator, and $d$ quantifies the
  dispersion in the system. The case $d<0$ corresponds to normal dispersion whereas $d>0$ is called the
  anomalous regime, cf.~\cite{Godey_et_al2014, Godey_2016}. The loss of power due to radiation and wave\-guide coupling is
  modeled by the damping term $- \ri \hat a_k(t)$. In the literature a stationary
  solution of \eqref{basic.F} is called a frequency comb.

  \medskip

  Via the Fourier series $a(x,t)= \sum_{k\in \Z} \hat a_k(t)  e^{ \ri kx}$ frequency combs can equally be
  defined as stationary solutions of the Lugiato-Lefever equation
  \begin{equation}
     \ri\partial_t a(x,t) = (- \ri+\zeta)a(x,t)-d \partial_x^2 a(x,t) -|a(x,t)|^2 a(x,t)+ \ri f,
    \quad x\in \R/2\pi\mathbb{Z}, \; t \in \R.
	\label{basic}
  \end{equation}
  It was originally proposed in \cite{Lugiato_Lefever1987} as a model for the envelope of a field transmitted
  through a nonlinearly responding optical cavity. It resembles a nonlinear Schr\"odinger equation with added
  damping and forcing. Stationary solutions of \eqref{basic} of the form $a=a_1+\ri a_2$ correspond to
  solutions $(a_1, a_2)$ of \eqref{Gl freqcomb}.

  \medskip

  The experimental generation of frequency combs in microresonators has been demonstrated many times, cf. the
  review paper \cite{kippenberg_2011}. One of the first demonstrations \cite{haye_2007} used a toroidal
  fused-silica microresonator. In \cite{herr_2012} the dynamics of the Kerr comb formation process is
  experimentally explored and found to be independent of the resonator material system and geometry. One of
  the first theoretical papers \cite{chembo_2010} marks the starting point for a series of subsequent
  investigations and publications. In this paper a numerical simulation of Kerr frequency combs is given
  based on \eqref{basic.F}, i.e., the modal expansion of the fields. A considerable computational effort is
  needed to handle the multitude of coupled differential equations. Nevertheless, a detailed account of the
  temporal dynamics is supplied and analytical expressions (approximations) of the distance of primary comb
  lines in terms of resonator and pump parameters are derived.

  \medskip

  As one can see in Section~\ref{illustrations} there are many different shapes of frequency combs.
  Of particular interest are so-called soliton combs. These are stationary solutions of \eqref{basic} which
  are highly localized in space. Accordingly their frequency spectrum shows many densely spaced comb lines;
  cf. Figures~\ref{solitons_f}, \ref{many_solitons}, \ref{many_solitons2} and \ref{solitons_zeta}.
  Moreover, the power of the $k$-th excited frequency in a soliton comb is much higher than the $k$-th
  excited frequency in a comb with sparse frequency spectrum, cf. Figure~\ref{turing_f_2}. Since these properties of soliton combs are very desirable for high-speed data
  transmission, they received attention in recent literature. In \cite{coen2013a} a numerical study of pump
  and resonator parameters and their effect on the bandwidth of Kerr combs was performed, and first
  indications appeared that soliton combs can only be achieved by a special tuning of the pump parameters.
  According to the simulations presented in \cite{Erkintalo2014} these solitons show a high coherence along
  with a high number of comb lines with flat power distribution. The first experimental proof of soliton
  combs was done in \cite{Herr2013}. The effect of higher order dispersion terms is the topic of
  \cite{Parra-Rivas2014}. It is shown that incorporating third-order dispersion terms into the model enlarges
  the parameter ranges where stable soliton combs exist. In \cite{Wang2014} the effect of higher order
  dispersion on the comb shape is discussed.

  \subsection{Further mathematical results} \label{further}

  A rigorous study of the time-dependent problem \eqref{basic} both from the analytical and from the numerical
  point of view was recently given in \cite{jami:14}. Applying Theorem~2.1 of \cite{jami:14} to the function
  $a(x,t)e^{\ri\zeta t}$ one obtains that for $d=1$ and initial data lying in $H^4_{\rm per}([0,2\pi];\C)$ the
  initial value problem associated to \eqref{basic} admits a unique solution
  \begin{equation*}
    a \in C(\R_+;H^4_{\rm per}([0,2\pi];\C))\cap C^1(\R_+;H^2_{\rm per}([0,2\pi];\C))\cap
    C^2(\R_+;L^2_{\rm per}([0,2\pi];\C))
  \end{equation*}
  satisfying the additional bounds $\|a(t)\|_2\leq C, \|a(t)\|_{H^1}\leq C\sqrt{1+t}$ for some positive
  number $C>0$ which is independent of $t$. Furthermore, the paper provides a detailed analysis of the
  Strang splitting associated to \eqref{basic} including error bounds in
  $L^2_{\rm per}([0,2\pi];\C),H^1_{\rm per}([0,2\pi];\C)$ as well as estimates related to the stability properties of
  the numerical scheme. Further numerical and analytical results related to periodically forced and damped NLS
  may be found in \cite{bishop:90,haller:99}.

  \subsection{Structure of the paper}

  The paper is organized as follows. In Section~\ref{setup} we provide the functional analytical
  framework for our analysis. This includes an appropriate choice of the function spaces
  and corresponding solution concepts. In Section~\ref{proof_of_theorem1and2} the proofs of Theorem~\ref{Thm 1
  A priori estimates} (a~priori bounds) and Theorem~\ref{Thm 2 nonexistence} (uniqueness)
  are given. The proofs of the bifurcation results from Theorem~\ref{Thm 3 bifurcation zeta} and \ref{Thm 4
  bifurcation f} can be found in Section~\ref{proof_of_theorem3and4}. Section~\ref{illustrations}
  contains illustrations with tables of bifurcation points, bifurcation diagrams and plots of approximate solutions. In the final Section~\ref{conclusions} we draw conclusions from our results and formulate some open questions.

 \subsection{On the generation of the numerical plots} The illustrations in Sections~\ref{zeta=0,d=0.1}--\ref{f=2,d=-0.1} were
 created with the software package AUTO. It is a free software which determines bifurcation points,
 approximations of solutions and generates bifurcation diagrams. It can be downloaded from
 \texttt{indy.cs.concordia.ca/auto/}. We postprocessed the outupt of AUTO by a MATLAB program to improve the
 quality of the approximated solutions of \eqref{Gl freqcomb} via several Newton iterations and to compute
 the Fourier coefficients of the improved approximated solutions. The MATLAB program also produces .pdf files
 of plotted solutions and their Fourier coefficients. A .zip file containing a README-description and driver
 files for the code can be downloaded freely from
 \texttt{www.waves.kit.edu/downloads/CRC1173\_Preprint\_2016-7\_supplement.zip}. \linebreak By running the
 driver files AUTO and MATLAB will be invoked and generate all plots of Sections~\ref{zeta=0,d=0.1}--\ref{f=2,d=-0.1}.

  \section{Mathematical setup} \label{setup}

  First we describe the spaces of solutions in which our analysis works. A weak solution $a\in
  H^1([0,2\pi];\R^2)$ of \eqref{Gl freqcomb} will be called a solution for the sake of simplicity.
  Notice that every such solution coincides almost everywhere with a smooth classical solution of the equation
  so that regularity issues will not play a role in the sequel and all solution concepts in fact coincide.
  In the context of Theorems~\ref{Thm 1 A priori estimates} and \ref{Thm 2 nonexistence} it is convenient to
  consider $2\pi$-periodic classical solutions. For the proof of Theorem~\ref{Thm 2 nonexistence} the space
  $H^2_{\rm per}([0,2\pi];\R^2)$ will be useful. In the context of the bifurcation results of
  Theorems~\ref{Thm 3 bifurcation zeta} and \ref{Thm 4 bifurcation f} we consider synchronized solutions of
  \eqref{Gl freqcomb}, i.e. solutions that satisfy additionally $a'(0)=a'(\pi)=0$. If we set
  $H:=H^1([0,\pi],\R^2)$ then $a=(a_1,a_2)$ is a synchronized solution of \eqref{Gl freqcomb} if and only if
  $a=(a_1,a_2)\in H$ satisfies
  \begin{align*}
    \int_0^\pi d a_1'\phi_1'\,dx &= \int_0^\pi \big(-a_2-\zeta a_1 + (a_1^2+a_2^2)a_1 \big)\phi_1 \,dx
    &&\hspace{-1cm}\mbox{ for all } \phi_1 \in H,\\
    \int_0^\pi d a_2'\phi_2'\,dx &= \int_0^\pi \big(a_1-\zeta a_2 +(a_1^2+a_2^2)a_2-f \big)\phi_2 \,dx
    &&\hspace{-1cm}\mbox{ for all }\phi_2 \in H.
  \end{align*}
  A synchronized solution can be extended evenly around $x=\pi$ and thus produce a $2\pi$-periodic function.
  This weak setting in the Hilbert space $H$ will be convenient for the proof of the bifurcation results.

  \medskip

  Next we describe the trivial (i.e. spatially constant) solutions of \eqref{Gl freqcomb}. In order to obtain
  a global parameterization of the solution curves some new auxiliary parameters $t$ resp. $s$ will be used
  instead of $\zeta$ resp. $f$. The totality of constant solutions is given next.
  
  \medskip

  \begin{lemma} \label{trivial}
  Let $d\neq 0$ be fixed.
  \begin{itemize}
  \item[(a)] Let $f\in \R$ be given. Then the set of constant solutions $(a_1,a_2,\zeta)$ of \eqref{Gl
  freqcomb} is given by $\hat\Gamma_f = \{ (\hat a_1(t),\hat a_2(t),\hat\zeta(t)) :|t|<1\}$ where
  $$
    \hat a_1(t) = f(1-t^2),\qquad
    \hat a_2(t) = -ft\sqrt{1-t^2},\qquad
    \hat \zeta(t) = f^2(1-t^2)+\frac{t}{\sqrt{1-t^2}}.
  $$
  \item[(b)] Let $\zeta\in \R$ be given. Then the set of constant solutions $(a_1,a_2,f)$ of \eqref{Gl
  freqcomb} is given by $\bar\Gamma_\zeta= \{(\bar a_1(s),\bar a_2(s),\bar f(s)): s\in\R\}$ where
  $$
    \bar a_1(s) = \frac{s}{\sqrt{1+(s^2-\zeta)^2}},\;\;
    \bar a_2(s) = \frac{s(s^2-\zeta)}{\sqrt{1+(s^2-\zeta)^2}},\;\;
    \bar f(s) = s\sqrt{1+(s^2-\zeta)^2}.
  $$
  \end{itemize}
  \end{lemma}

  \begin{proof}
  Let us first show that constant solutions $a=(a_1,a_2)$ of \eqref{Gl freqcomb} satisfy
  \begin{equation}  \label{das_simple}
    f^2 = |a|^2(1+(|a|^2-\zeta)^2).
  \end{equation}
  Indeed, for constant solutions \eqref{Gl freqcomb} can be written as
  $$
    \begin{pmatrix} |a|^2-\zeta & -1 \\ 1 & |a|^2-\zeta \end{pmatrix} \begin{pmatrix} a_1 \\ a_2\end{pmatrix}
     = \begin{pmatrix} 0 \\ f \end{pmatrix}
  $$
  and hence by inverting the matrix
  \begin{equation}  \label{das_simpleII}
    \begin{pmatrix} a_1 \\ a_2\end{pmatrix} = \frac{f}{1+(|a|^2-\zeta)^2}\begin{pmatrix} 1 \\ |a|^2-\zeta
  \end{pmatrix}.
  \end{equation}
  Taking the Euclidean norm on both sides of the equation gives \eqref{das_simple}.
  \medskip

  Now let us prove (a), so let $f\in\R$ be given and define $t\in (-1,1)$ via $t(1-t^2)^{-1/2} =
  \zeta-|a|^2$. Then \eqref{das_simple} implies
  $$
    \zeta-t(1-t^2)^{-1/2} = |a|^2=\frac{f^2}{1+t^2(1-t^2)^{-1}} = f^2(1-t^2)
  $$
   and hence
  $$
    \zeta= f^2(1-t^2)+\frac{t}{\sqrt{1-t^2}}.
  $$
  From the linear system \eqref{das_simpleII} and the definition of $t$ we obtain the desired
  formulas for $a_1, a_2$. In order to prove (b) let $\zeta\in\R$ and set $s:=\sign(f)|a|$. Then we have
  $$
    f^2=s^2(1+(s^2-\zeta)^2),\quad\text{hence }
    f=s\sqrt{1+(s^2-\zeta)^2}.
  $$
  From the linear system \eqref{das_simpleII} and this formula for $s$ we obtain the result.
  \end{proof}

  \section{Proof of Theorem~\ref{Thm 1 A priori estimates} and Theorem~\ref{Thm 2 nonexistence}} \label{proof_of_theorem1and2}

  We always assume $d\not =0$ and $f,\zeta\in\R$. We write $\|\cdot\|_p$ for the standard norm on
  $L^p([0,2\pi];\R^2)$ for $p\in [1,\infty]$.

  \medskip

  \noindent
  {\em Proof of Theorem~\ref{Thm 1 A priori estimates}:} We divide the proof into several steps.

  \medskip

  \noindent
  {\em Step 1:} Here we prove the $L^2$-estimate $\|a\|_2 \leq \sqrt{2\pi}|f|$. To this end we define the
  $2\pi$-periodic function $g:[0,2\pi]\to\R$ by
  $$
    g := d (a_2 a_1 ' - a_1 a_2')'.
  $$
  By using \eqref{Gl freqcomb} one finds
  \begin{equation}\label{Gl g}
    g = d a_2 a_1'' - d a_1 a_2'' = a_2(a_2+\zeta a_1-|a|^2a_1) + a_1(a_1-\zeta a_2+|a|^2a_2-f) = |a|^2 - f
    a_1.
  \end{equation}
  Since $a_2 a_1'- a_1 a_2'$ is $2\pi$-periodic we obtain
  \begin{align*}
    0
    = \int_0^{2\pi}  g\,dx
    = \int_0^{2\pi} \big(\, |a|^2 - fa_1\,\big)\,dx
    \geq  \|a\|_2^2 - \sqrt{2\pi}|f| \|a\|_2
  \end{align*}
  which implies the desired $L^2$-bound
 \begin{equation}
 \|a\|_2 \leq \sqrt{2\pi}|f|.
 \label{l2_bound}
 \end{equation}

\medskip

\noindent{\em Step 2:} Next we prove $|d|\|a'\|_2
    \leq 6\pi f^2 \|a\|_2$ and thus $|d|\|a'\|_2 \leq 6\sqrt{2}\pi^{3/2}|f|^3$ due to \eqref{l2_bound}. Using
    the differential equation \eqref{Gl freqcomb} we get
  \begin{align}
    |d| \|a'\|_2^2
    &= |d| \int_0^{2\pi} \, a_1' \big(-d a_2''+\zeta a_2-|a|^2a_2+f\big)'\,dx
      + a_2' \big(d a_1''-\zeta a_1+|a|^2a_1\big)'\,dx  \nonumber\\
    &= |d|d \int_0^{2\pi} \big(\, a_2' a_1''' - a_1' a_2'''\,\big)\,dx
      + |d|\int_0^{2\pi}  a_1' (-|a|^2a_2)' + a_2' (|a|^2a_1)'\,dx    \nonumber \\
    &=  |d|d \int_0^{2\pi} (a_2' a_1'' - a_1' a_2'')'\,dx
      + |d|\int_0^{2\pi} (-a_2 a_1' +a_1 a_2' ) (|a|^2)'\,dx    \nonumber  \\
    &=  0+|d|\int_0^{2\pi} (a_2 a_1' -a_1 a_2' )'|a|^2\,dx     \label{previous_estimate} \\
    &\leq  \int_0^{2\pi} |g| |a|^2\,dx     \nonumber\\
    &\leq \|g\|_\infty \|a\|_2^2  \nonumber \\
    &\stackrel{\eqref{l2_bound}}{\leq} \sqrt{2\pi} |f|\|g\|_\infty \|a\|_2. \nonumber
  \end{align}
  Note that $g$ is the derivative of a $2\pi$-periodic function and therefore satisfies $\int_0^{2\pi} g\,dx=0$. Hence
  there exists $x_0\in [0,2\pi]$ such that $g(x_0)=0$ and the supremum norm of $g$ can be estimated as follows
  \begin{align*}
    \|g\|_\infty
    &\leq \sup_{x\in [0,2\pi]} |g(x)-g(x_0)| \\
    &\leq \int_0^{2\pi} |g'|\,dx \\
    &\stackrel{\eqref{Gl g}}{\leq} \int_0^{2\pi} 2|a'||a| + |f||a_1'|\,dx  \\
    &\leq(2\|a\|_2+\sqrt{2\pi}|f|) \|a'\|_2  \\
    &\stackrel{\eqref{l2_bound}}{\leq}  3\sqrt{2\pi}|f| \|a'\|_2.
  \end{align*}
  By our previous estimates \eqref{previous_estimate} and \eqref{l2_bound} this gives
  \begin{equation} \label{Gl estimate H1}
    |d|\|a'\|_2
    \leq \sqrt{2\pi}|f|\|a\|_2 \cdot 3\sqrt{2\pi}|f|
    = 6\pi f^2 \|a\|_2
    \leq 6\sqrt{2}\pi^{3/2}|f|^3
  \end{equation}
  which finishes step 2.

  \medskip

  \noindent
  {\em Step 3:} Now we show the first of two $L^\infty$-bounds: $\|a\|_\infty \leq
  |f|(1+12\pi^2f^2|d|^{-1})$. From the $L^2-$estimate \eqref{l2_bound} we infer that there is an $x_1\in I$ satisfying $|a(x_1)|\leq |f|$.
  Hence our first $L^\infty$-estimate follows from
  \begin{align}
    \|a\|_\infty
    &\leq |a(x_1)| + \|a-a(x_1)\|_\infty \nonumber \\
    &\leq |f| + \|a'\|_1 \nonumber \\
    &\leq |f| + \sqrt{2\pi}\|a'\|_2 \label{first_l_infinity}\\
    &\stackrel{\eqref{Gl estimate H1}}{\leq} |f| + \sqrt{2\pi}\cdot 6\sqrt{2}\pi^{3/2}|f|^3|d|^{-1} \nonumber
    \\
    &= |f|(1+12\pi^2f^2|d|^{-1}).\nonumber
  \end{align}

  \medskip

  \noindent
  {\em Step 4:} Next we show
  \begin{equation} \label{Gl ineq Step4}
    \|a\|_2 \leq (-\zeta\sign(d)-\gamma(d,f))^{-1} \sqrt{2\pi}|f|.
  \end{equation}
  whenever $-\zeta \sign(d)-\gamma(d,f)>0$ for $\gamma(d,f)$ from \eqref{Gl defn gamma(d,f)}.
  Testing \eqref{Gl freqcomb} with $(a_1,a_2)$ and adding up the resulting equations yields
  \begin{equation} \label{Gl identity}
    d\|a'\|_2^2 = - \zeta \|a\|_2^2 + \|a\|_4^4 - f\int_0^{2\pi} a_2 \,dx.
  \end{equation}
  This can be used in the following way.
  \begin{align}
    36\pi^2f^4|d|^{-1}\|a\|_2^2
    &= |d|^{-1} (6\pi f^2\|a\|_2)^2  \nonumber \\
    &\stackrel{\eqref{Gl estimate H1}}{\geq} |d|^{-1} (|d|\|a'\|_2)^2 \label{second_l_infinity} \\
    &= |d|\|a'\|_2^2 \nonumber \\
    &\stackrel{\eqref{Gl identity}}{=} - \zeta\sign(d) \|a\|_2^2 + \sign(d)\|a\|_4^4 - \sign(d)f\int_0^{2\pi} a_2\,dx. \nonumber
  \end{align}
  In order to prove \eqref{Gl ineq Step4} we first suppose $d>0$. Then \eqref{second_l_infinity} implies
  $$
  36\pi^2f^4|d|^{-1}\|a\|_2^2 \geq -\zeta \|a\|_2^2- \sqrt{2\pi}|f|\|a\|_2
  $$
  from which we infer the desired bound
  $$
  \sqrt{2\pi}|f| \geq \left(-\zeta-36\pi^2f^4|d|^{-1}\right) \|a\|_2=(-\zeta- \gamma(d,f))\|a\|_2.
  $$
 Supposing now $d<0$ we find that \eqref{second_l_infinity} implies
 \begin{align*}
 36\pi^2f^4|d|^{-1}\|a\|_2^2 & \geq (\zeta-\|a\|_\infty^2)  \|a\|_2^2-\sqrt{2\pi}|f|\|a\|_2 \\
 & \stackrel{\eqref{first_l_infinity}}{\geq} (\zeta-f^2(1+12\pi^2f^2|d|^{-1})^2)\|a\|_2^2 - \sqrt{2\pi}|f|\|a\|_2
 \end{align*}
 from which we obtain
  $$
    \sqrt{2\pi}|f|
    \geq \left(\zeta-36\pi^2f^4|d|^{-1}-f^2(1+12\pi^2f^2|d|^{-1})^2\right \|a\|_2 = (\zeta-\gamma(d,f))\|a\|_2.
  $$
  so that \eqref{Gl ineq Step4} is proved.

\medskip

\noindent
{\em Step 5:} Finally, we show
$$
  \|a\|_\infty \leq  (-\zeta\sign(d)-\gamma(d,f))^{-1}|f|(1+12\pi^2f^2|d|^{-1}).
$$
$-\zeta \sign(d)-\gamma(d,f)>0$. This completes the proof of Theorem~\ref{Thm 1 A priori estimates}. The
$L^2$-estimate from Step~4 entails, just as in Step~3, that there exists some $x_1\in [0,2\pi]$ such that
$|a(x_1)|\leq (-\zeta \sign(d)-\gamma(d,f))^{-1}|f|$ and the claim follows from
  \begin{align*}
    \|a\|_\infty & \leq |a(x_1)|+\sqrt{2\pi}\|a'\|_2 \\
    &\stackrel{\eqref{Gl estimate H1}}{\leq }|a(x_1)|+ 6\pi\sqrt{2\pi}f^2|d|^{-1}\|a\|_2 \\
    &\leq  \frac{|f|(1+12\pi^2f^2|d|^{-1})}{-\zeta\sign(d)-\gamma(d,f)} \qquad \mbox{ by Step 4}.
  \end{align*}

%

 \medskip

  Now we come to the proof of the uniqueness result from Theorem~\ref{Thm 2 nonexistence}. Let us first
  outline our strategy to prove the result. We use the fact that a solution
  $a=(a_1,a_2):[0,2\pi]\to\R\times\R$ of \eqref{Gl freqcomb} is constant if and only if the function
  $A=(A_1,A_2):=(a_1',a_2')$ is trivial. Since $(a_1,a_2)$ solves \eqref{Gl freqcomb} the functions $A_1,
  A_2$ satisfy the boundary value problem
  \begin{align} \label{eq_A}
    \begin{aligned}
    -d A_1'' &= -A_2 - \zeta A_1+(3a_1^2+a_2^2)A_1+2a_1a_2A_2, \\
    -d A_2'' &= A_1 - \zeta A_2 +2a_1a_2A_1+ (a_1^2+3a_2^2)A_2, \\
    & A_1, A_2 \;\; 2\pi\mbox{-periodic.}
    \end{aligned}
  \end{align}
  In view of \eqref{eq_A} it is natural to study the operator
  $$
  L_{d,\zeta}: \left\{ \begin{array}{rcl}
  H^2_{\rm per}([0,2\pi];\R^2) & \to & L^2([0,2\pi];\R^2), \vspace{\jot}\\
    (B_1,B_2) & \mapsto & (-d B_1''+\zeta B_1+B_2, -dB_2''+\zeta B_2-B_1).
  \end{array}
  \right.
  $$
  Using the fact that the embedding $\Id: H^2_{\rm per}([0,2\pi];\R^2)\to L^2([0,2\pi];\R^2)$ is compact we
  obtain the following result.
  
  \medskip

  \begin{lemma} \label{lem operatornorm estimate}
    The operator $L_{d,\zeta}$ has a bounded inverse $L_{d,\zeta}^{-1}:
    L^2([0,2\pi];\R^2) \to H^2_{\rm per}([0,2\pi];\R^2)$ with the property that $\Id
    \circ L_{d,\zeta}^{-1}: L^2([0,2\pi];\R^2)\to L^2([0,2\pi];\R^2)$ is compact and $\|\Id \circ L_{d,\zeta}^{-1}\| \leq
    \min\{1, (\sign(d)\zeta)_+^{-1}\}$.
  \end{lemma}

  \begin{proof}
    Let $(B_1,B_2)\in H^2_{\rm per}([0,2\pi];\R^2)$ and $(g_1,g_2)\in L^2([0,2\pi];\R^2)$ satisfy the equation
    $L_{d,\zeta}(B_1,B_2) = (g_1,g_2)$, i.e., $-d B_1'' +\zeta B_2 + B_2 = g_1$ and $-d B_2'' +\zeta B_2 - B_1
    = g_2$.
    Testing these differential equations with $(B_1,B_2)$ respectively $(B_2,-B_1)$ and adding up the
    resulting equations yields
    \begin{align}
      d(\|B_1'\|_2^2+\|B_2'\|_2^2) + \zeta (\|B_1\|_2^2+\|B_2\|_2^2)
      &=  \int_0^{2\pi} \big(\,g_1B_1+g_2B_2\,\big)\,dx, \label{first}\\
      \|B_1\|_2^2+\|B_2\|_2^2
      &=  \int_0^{2\pi}\big(\,g_1B_2-g_2B_1\,\big)\,dx. \label{second}
    \end{align}
    The second line \eqref{second} shows that $L_{d,\zeta}$ is injective. Moreover, using $
    \|B_1'\|_2^2+\|B_2'\|_2^2\geq 0$ as well as H\"older's inequality we obtain from \eqref{first} and
    \eqref{second}
    $$
      \max\{1,\sign(d)\zeta\} \|(B_1,B_2)\|_2^2 \leq \|(g_1,g_2)\|_2 \|(B_1,B_2)\|_2
    $$
    and thus
     $$
       \|(B_1,B_2)\|_2 \leq \min\{1,(\sign(d)\zeta)_+^{-1}\} \|(g_1,g_2)\|_2.
     $$
    From this and \eqref{first} we get the estimate
    $$
      |d|\|(B_1', B_2')\|_2^2  \leq (1+|\zeta|) \|(g_1,g_2)\|_2^2
    $$
    and using the differential equation and the $L^2$-estimate on $(B_1,B_2)$ we find
    $$
      |d|\|(B_1'',B_2'')\|_2 \leq  (2+|\zeta|)\|(g_1,g_2)\|_2.
    $$
    This proves the bounded invertibility of $L_{d,\zeta}$ as well as the norm estimate for $\Id\circ
    L_{d,\zeta}^{-1}$.
 \end{proof}

 \medskip
 
  \noindent
  {\em Proof of Theorem~\ref{Thm 2 nonexistence}:} If $(A_1,A_2)$ satisfies \eqref{eq_A} then we have
  $(A_1,A_2)=K_a (A_1,A_2)$ where $K_a:L^2([0,2\pi];\R^2)\to L^2([0,2\pi];\R^2)$ is given by
  \begin{equation} \label{Gl Def Ka}
    K_a \vecII{A_1}{A_2} := \Id\circ L_{d,\zeta}^{-1}\Big( M_a \vecII{A_1}{A_2} \Big) \mbox{ with }
     M_a= \matII{3a_1^2+a_2^2}{2a_1a_2}{2a_1a_2}{a_1^2+3a_2^2}.
  \end{equation}
  Hence, it suffices to prove that the operator norm $\|K_a\|$ is less than 1 whenever
  $\sign(d)\zeta>\zeta^*$ or $\sign(d)\zeta<\zeta_*$. Consider the matrix $M_a$ as a map from $L^2([0,2\pi];\R^2)$
  into itself. Then its operator norm is bounded as follows
  $$
      \|M_a\|
      \leq \Big\|\matII{4\|a\|_\infty^2}{2\|a\|_\infty^2}{2\|a\|_\infty^2}{4\|a\|_\infty^2}\Big\|
      =  \|a\|_\infty^2\cdot\Big\|\matII{4}{2}{2}{4}\Big\|
      =  6\|a\|_\infty^2,
  $$
  since the largest eigenvalue of $\matII{4}{2}{2}{4}$ is $6$. Combining this inequality with
  the estimate from Lemma~\ref{lem operatornorm estimate} we get
  $$
    \|K_a\| \leq 6\min\{1,(\sign(d)\zeta)_+^{-1}\} \|a\|_\infty^2.
  $$
  In the first case, where $\sign(d)\zeta>\zeta^*>0$, Theorem~\ref{Thm 1 A priori estimates} gives, by choice
  of $\zeta^*$,
  $$
    \|K_a\|
    \leq  6(\sign(d)\zeta)^{-1}\|a\|_\infty^2
    <  6(\zeta^*)^{-1} f^2(1+12\pi^2f^2|d|^{-1})^2=1.
  $$
  In the second case, where $\sign(d)\zeta<\zeta_*<0$ and in particular $-\zeta\sign(d)-\gamma(d,f)>0$, we
  get from Theorem~\ref{Thm 1 A priori estimates} and again by the choice of $\zeta_*$
  $$
    \|K_a\|
    \leq 6\|a\|_\infty^2
    \leq \left( \frac{\sqrt{6}|f|(1+12\pi^2f^4|d|^{-1})}{-\zeta\sign(d)-\gamma(d,f)}\right)^2
    < \left( \frac{\sqrt{6}|f|(1+12\pi^2f^4|d|^{-1})}{-\zeta_*-\gamma(d,f)}\right)^2 =1
  $$
   which is all we had to show. \qed

%
%
%
  \section{Proof of Theorem \ref{Thm 3 bifurcation zeta} and Theorem \ref{Thm 4 bifurcation f}} \label{proof_of_theorem3and4}

  In this section we prove the bifurcation results for the Lugiato-Lefever equation~\eqref{Gl freqcomb}. We
  will always assume that $d\not = 0$ is fixed. As explained earlier our sufficient conditions for
  bifurcation from constant solutions of~\eqref{Gl freqcomb} will be established in the context of the
  Neumann boundary value problem
  \begin{align}\label{Gl freqcomb_neumann}
  \begin{aligned}
     -d a_1'' &= -a_2-\zeta a_1 + (a_1^2+a_2^2)a_1, &\quad a_1'(0)=a_1'(\pi)=0,\\
     -d a_2''&= a_1-\zeta a_2 +(a_1^2+a_2^2)a_2-f, &\quad a_2'(0)=a_2'(\pi)=0.
  \end{aligned}
  \end{align}
  In this way the shift-invariance of the general $2\pi$-periodic system is circumvented. Using the notation
  introduced in the introduction we will find the existence of synchronized solution branches bifurcating
  from the curves of constant solutions.

  Let us now shortly outline how our bifurcation analysis is organized. In
  Section~\ref{framework} we first provide a functional analytical framework for solutions of \eqref{Gl
  freqcomb_neumann}. The construction of solutions is done with the help of the bifurcation theorem due to
  Crandall and Rabinowitz (Theorem~\ref{Thm Crandall-Rabinowitz}).
  In Theorem~\ref{Thm 3 bifurcation zeta}, the family of trivial solutions will be $\hat\Gamma_f$ for fixed
  $f\in\R$ and in Theorem~\ref{Thm 4 bifurcation f} it will be $\bar \Gamma_\zeta$ for fixed $\zeta\in\R$.
  The proof of these theorems is accomplished in four steps. In Section~\ref{subsec Det bifpoints} we first
  determine the candidates for the bifurcation points of \eqref{Gl freqcomb} with respect to $\hat\Gamma_f,
  \bar\Gamma_\zeta$ proving Theorem~\ref{Thm 3 bifurcation zeta}(i), Theorem~\ref{Thm 4 bifurcation f}(i).
  This result will be used in Section~\ref{subsec number bifpoints} to establish the upper bounds for the
  number of bifurcation points claimed in Theorem~\ref{Thm 3 bifurcation zeta}(ii) and Theorem~\ref{Thm 4
  bifurcation f}(ii). In order to prove the existence of bifurcating branches it remains to check the
  hypotheses of the Crandall-Rabinowitz theorem in the context of the Neumann boundary value problem
  \eqref{Gl freqcomb_neumann}. In Section~\ref{simple_kernels} we show that the kernels at the possible
  bifurcation points (calculated in Section~\ref{subsec Det bifpoints}) are simple if the conditions (S) from
  the respective theorem holds. In the same way the transversality condition will be verified in
  Section~\ref{transversality} supposing that condition (T) holds. Hence, a direct application of
  Theorem~\ref{Thm Crandall-Rabinowitz} establishes the existence of local curves containing nontrivial
  solutions of \eqref{Gl freqcomb_neumann} that emanate from $\hat\Gamma_f,\bar\Gamma_\zeta$ respectively. 
  Considering \eqref{Gl freqcomb_neumann} as an equation in the space of
   $2\pi/k$-periodic functions one even finds that the uniquely determined bifurcating branch
   from bifurcation points associated to $k$ via \eqref{Gl bifpoints Thm3} resp. \eqref{Gl
   bifpoints Thm4} locally consists of $2\pi/k$-periodic solutions. Notice that this is possible due to the
   fact that the functions in the kernel of the linearized operator are $2\pi/k$-periodic, see
   \eqref{Gl kernel Ga} in Proposition \ref{Prop kernels}. Moreover, the a priori bounds for $a$ from Theorem~\ref{Thm 1 A priori estimates} and the uniqueness
  result from Theorem~\ref{Thm 2 nonexistence} tell us that for any given $f\in\R$ the continua
  emanating from $\hat\Gamma_f$ must be bounded with respect to both variables $a,\zeta$ so that Rabinowitz'
  global bifurcation theorem \cite{Rab_some_global} yields that each continuum returns to $\hat\Gamma_f$
  at another point. 
  Hence, part (iii) of Theorem~\ref{Thm 3 bifurcation zeta} and Theorem~\ref{Thm 4
  bifurcation f} is shown and the proof is complete.

  \medskip

  \subsection{Functional analytical framework and preliminaries} ~ \label{framework}

  \medskip

  We look for solutions of \eqref{Gl freqcomb_neumann} in the function space $H:=H^1([0,\pi];\R^2)$. It is a Hilbert space with the inner product
  $\langle\cdot,\cdot\rangle_H$ given by
  \begin{equation*}
    \langle \phi,\psi \rangle_H := \int_0^\pi
    |d| (\phi_1'\psi_1'+ \phi_2'\psi_2')+\phi_1 \psi_1+\phi_2\psi_2\,dx
    \qquad\text{for } \phi=\begin{pmatrix} \phi_1 \\ \phi_2\end{pmatrix}, \psi=\begin{pmatrix} \psi_1 \\ \psi_2\end{pmatrix}\in H.
  \end{equation*}
  If $D:\dom(D)\to\R$ denotes the selfadjoint realization of
  the differential operator $\phi \mapsto -|d|\phi'' + \phi$ with homogeneous
  Neumann boundary values at $0, \pi$ then $\langle \phi,\psi\rangle_H = \langle D\phi,\psi\rangle_{L^2}$ for all $\phi\in
  \dom(D)$. Even though it will not be used in the sequel let us state
  without proof
  \begin{equation*}
    \dom(D)= \Big\{ \phi \in H^2([0,\pi];\R^2) :
    \phi'(0)=\phi'(\pi)= \begin{pmatrix} 0 \\ 0 \end{pmatrix}  \Big\}.
  \end{equation*}
  The operator $D$ has a compact inverse $D^{-1}: H\to \dom(D)\subset H$ so that \eqref{Gl freqcomb_neumann}
  may be rewritten as $G(a,\zeta,f)=0$ where the function $G:H\times \R\times\R \to H$ is given by
  \begin{equation}
	G(a,\zeta,f) :=  \sign(d)a-D^{-1}\left(-\zeta a +\sign(d) a+ |a|^2a + \begin{pmatrix} -a_2 \\ a_1
	\end{pmatrix} + \begin{pmatrix} 0 \\ -f \end{pmatrix} \right).
	\label{crab}
  \end{equation}
  In order to prove bifurcation results from the family of constant solutions of \eqref{Gl freqcomb_neumann} let us recall
  the Crandall-Rabinowitz bifurcation theorem.
  
  \medskip

  \begin{theorem}[Crandall-Rabinowitz \cite{CrRab_bifurcation}] \label{Thm Crandall-Rabinowitz}
    Let $I\subset\R$ be an open interval and let $F:H\times I\to H$ be twice continuously differentiable such
    that $F(0,\lambda)=0$ for all $\lambda\in I$ and such that $F_x(0,\lambda_0)$ is an index-zero Fredholm
    operator for $\lambda_0\in I$. Moreover assume:
    \begin{itemize}
      \item[(H1)] there is $\phi\in H,\phi\neq 0$ such that $\ker(F_x(0,\lambda_0))=\spann\{\phi\}$,
      \item[(H2)] $\langle F_{x\lambda}(0,\lambda_0)[\phi], \phi^*\rangle_H \neq 0$ where
    $\ker(F_x(0,\lambda_0)^*)=\spann\{\phi^*\}$.
    \end{itemize}
    Then there exists $\epsilon>0$ and a continuously differentiable curve $(x,\lambda): (-\epsilon,\epsilon)\to H\times \R$ with $\lambda(0)=\lambda_0$, $x(0)=0$, $x'(0)=\phi$ and $x(t)\neq 0$ for $0<|t|<\epsilon$ and $F(x(t),\lambda(t))=0$ for all $t\in (-\epsilon,\epsilon)$. Moreover, there exists a neighbourhood $U\times J\subset H\times I$ of $(0,\lambda_0)$ such that all nontrivial solutions in $U\times J$ of $F(x,\lambda)=0$ lie on the curve.
  \end{theorem}

  \medskip

  In the proof of Theorem~\ref{Thm 3 bifurcation zeta} and Theorem~\ref{Thm 4 bifurcation f} we will apply
  Theorem~\ref{Thm Crandall-Rabinowitz} to the functions $\hat F:H\times (-1,1)\to H$ and $\bar
  F:H\times\R\to H$ given by
  \begin{align} \label{Gl defn hatF, barF}
    \begin{aligned}
    \hat F(b,t) &:=G(b+\hat a(t),\hat\zeta(t),f), \qquad b\in H, t\in (-1,1),\\
    \bar F(b,s) &:=G(b+\bar a(s),\zeta,\bar f(s)), \qquad b\in H, s\in \R
    \end{aligned}
  \end{align}
  where the trivial solution curves $(\hat a(t),\hat\zeta(t))$ respectively $(\bar a(s),\zeta(s))$ are taken from
  Lemma~\ref{trivial}. Checking the assumptions of Theorem~\ref{Thm Crandall-Rabinowitz} requires the
  calculation of the derivatives of $\hat F,\bar F$ at the trivial solutions. The necessary preparations are
  made in the following proposition.
  
  \medskip

  \begin{proposition}  \label{prop derivatives}
    Let $a \in H$, $\zeta,f\in\R$ and set
    \begin{gather*} 
      N(a,\zeta) = \begin{pmatrix}
		-\zeta+3a_1^2+a_2^2 & -1+2a_1a_2 \\
		1+2a_1a_2 & -\zeta+a_1^2+3a_2^2
	\end{pmatrix},\\
	M_1(a) = \begin{pmatrix}
		6a_1& 2a_2 \\
		2a_2 & 2a_1
	\end{pmatrix},\qquad
	M_2(a) = \begin{pmatrix}
		2a_2& 2a_1 \\
		2a_1 & 6a_2
	\end{pmatrix}.
    \end{gather*}
    Then we have for all $\phi,\psi\in H$
    \begin{align*}
      G_a(a,\zeta,f)[\phi]
      &= \sign(d)\phi - D^{-1}\big( N(a,\zeta)\phi+\sign(d)\phi \big), \\
      G_{aa}(a,\zeta,f)[\phi,\psi]
      &= - D^{-1} \vecII{\phi^TM_1(a)\psi}{\phi^TM_2(a)\psi}
    \end{align*}  
    as well as
    \begin{align*}
      \begin{aligned}
        G_\zeta(a,\zeta,f) = D^{-1}a, \qquad
        G_f(a,f) = D^{-1} (0,1)^T, \\ 
        G_{a\zeta}(a,\zeta,f)[\phi] = D^{-1}\phi,\qquad
        G_{af}(a,\zeta,f)[\phi]  = 0.
       \end{aligned}
    \end{align*}
  \end{proposition}

  The proof of Proposition~\ref{prop derivatives} is mere calculation and will therefore be dropped. In the
  next proposition we characterize $\ker(G_a(a,\zeta,f))$ and $\ker(G_a(a,\zeta,f)^*)$ at a constant solution
  $a\in \R^2\subset H$ of \eqref{Gl freqcomb_neumann}.
  
  \medskip

\begin{proposition} \label{Prop kernels}
  Let $\zeta,f\in\R$. Then for every constant solution $a=(a_1,a_2)\in \R^2\subset H$ of
  \eqref{Gl freqcomb_neumann} we have
  \begin{align}
    \ker(G_a(a,\zeta,f)) &=\spann\{\phi_k(a) : k\in\N_0 \text{ satisfies }\eqref{Gl bifurcation condition}\}
    , \label{Gl kernel Ga}\\
    \ker(G_a(a,\zeta,f)^*) &=\spann\{\phi_k^*(a) : k\in\N_0 \text{ satisfies }\eqref{Gl
    bifurcation condition}\}, \label{Gl kernel Gastar}
  \end{align}
  where
  \begin{equation} \label{Gl bifurcation condition}
    (\zeta+dk^2)^2-4|a|^2(\zeta+dk^2)+1+3|a|^4=0
  \end{equation}
  and
  \begin{align}
    \phi_k&(a)(x) = \alpha^k \cos(kx),\qquad
    \phi_k^*(a)(x) = \beta^k \cos(kx)
    \qquad\text{with }\alpha^k,\beta^k\in\R^2 \text{ given by} \notag \\
	\alpha^k
	&=  \begin{cases}
	\begin{pmatrix}
	1-2a_1a_2 \\ 3 a_1^2+a_2^2-\zeta - d k^2
	\end{pmatrix} \quad\text{if }a_1a_2 \neq \frac{1}{2}
	 \quad\text{or}\quad 3 a_1^2+a_2^2\neq \zeta + d k^2, \\
	 \begin{pmatrix}
	a_1^2+3a_2^2-\zeta-dk^2 \\ -1-2a_1a_2
    \end{pmatrix} \quad\text{if }a_1a_2 = \frac{1}{2} \quad\text{and}\quad
    3 a_1^2+a_2^2= \zeta + d k^2,
    \end{cases} \label{Gl formula phik}  \\
	\beta^k
	&= \begin{cases}
	\begin{pmatrix}
	-1-2a_1a_2 \\ 3a_1^2+a_2^2-\zeta - d k^2
    \end{pmatrix}
     \quad\text{if }a_1a_2 \neq -\frac{1}{2} \quad\text{or}\quad 3a_1^2+a_2^2\neq \zeta+dk^2, \\
     \begin{pmatrix}
	a_1^2+3a_2^2-\zeta-dk^2 \\ 1-2a_1a_2
    \end{pmatrix}  \quad\text{if }a_1a_2 = -\frac{1}{2}  \quad\text{and}\quad 3a_1^2+a_2^2= \zeta+dk^2.
     \end{cases}  \label{Gl formula phikstar}
  \end{align}
\end{proposition}
\begin{proof}
  By Proposition \ref{prop derivatives} every function $\phi\in \ker(G_a(a,\zeta,f))$ satisfies $-d\phi'' =
  N(a,\zeta)\phi$ in~$(0,\pi)$ and $\phi'(0)=\phi'(\pi)=0$. From the Fourier series
  expansion
  \begin{equation} \label{fourier}
    \phi(x) = \sum_{k\in \N_0} \alpha^k\cos(kx) \qquad\text{for }x\in [0,\pi]
  \end{equation}
  we obtain the equation
  \begin{equation*}
    \sum_{k\in \N_0} (d k^2 \Id -N(a,\zeta)) \alpha^k \cos(kx) = 0 \qquad\text{for }x\in [0,\pi].
  \end{equation*}
  Hence, for all $k\in \N_0$ the vector $\alpha^k$ lies in the kernel of
  the matrix
    \begin{equation*}
	dk^2\Id-N(a,\zeta) = \begin{pmatrix}
	\zeta+dk^2-3a_1^2-a_2^2 & 1-2a_1a_2 \\
	-1-2a_1a_2 & \zeta+dk^2-a_1^2-3a_2^2
	\end{pmatrix}.
	\end{equation*}
  This implies that $\ker(G_a(a,\zeta,f))$ is nontrivial if and only if the determinant of one of these
  matrices vanishes. Calculating $\det(dk^2\Id-N(a,\zeta))$ for all $k\in\N_0$ we obtain that
  $\ker(G_a(a,\zeta,f))$ is nontrivial if and only if there is a solution $k\in\N_0$ of \eqref{Gl bifurcation
  condition}. In that case the kernel of $dk^2\Id-N(a,\zeta)$ is spanned by the
  vector $\phi_k(a)$ given by \eqref{Gl formula phik} which proves the formula for $\ker(G_a(a,\zeta,f))$ from
  \eqref{Gl kernel Ga}. A similar calculation shows that $\phi^*\in\ker(G_a(a,\zeta,f)^\ast)$ satisfies
  $-d{\phi^*}'' = N(a,\zeta)^T\phi^*$ in $(0,\pi)$ and ${\phi^*}'(0)={\phi^*}'(\pi)=0$. From this the formula
  \eqref{Gl kernel Gastar} for $\ker(G_a(a,\zeta,f)^\ast)$ follows as above.
\end{proof}

\medskip

   Since for every given $\zeta, f\in \R$ equation \eqref{Gl bifurcation condition} has at most two
   different solutions $k_1,k_2\in\N_0$ we know that the spaces $\ker(G_a(a,\zeta,f))$ are at most two-dimensional. In the following
   proposition we single out those parameters for which we have one-dimensional kernels.
   
   \medskip

  \begin{proposition} \label{Prop simple kernels}
    Let $\zeta, f, a$ be chosen as in Proposition~\ref{Prop kernels} such that \eqref{Gl bifurcation
    condition} holds for some $k\in\N_0$. Then $\ker(G_a(a,\zeta,f))$ and
    $\ker(G_a(a,\zeta,f)^*)$ are one-dimensional if and only if
    $$
      -k^2+d^{-1}(4|a|^2-2\zeta)\neq j^2 \quad \text{for all } j\in \N_0\setminus\{k\}.
    $$
  \end{proposition}
  \begin{proof}
    Let $\ker(G_a(a,\zeta,f))$ contain two linearly independent nontrivial vectors. Proposition~\ref{Prop
    kernels} then implies that equation \eqref{Gl bifurcation condition} has a second solution
    $j\in\N_0$ which gives
    $$
     (\zeta+dk^2)^2 - 4(\zeta+dk^2)|a|^2 = (\zeta+dj^2)^2 - 4(\zeta+dj^2)|a|^2\quad\text{and}\quad
     \zeta+dk^2\neq \zeta+dj^2.
    $$
    From this we infer $2\zeta+dk^2+dj^2=4|a|^2$ or equivalently
    $$
      -k^2+ d^{-1}(4|a|^2-2\zeta) = j^2  \quad\text{for some } j\in \N_0\setminus\{k\}.
    $$
    Vice versa, by \eqref{Gl kernel Ga}, this condition implies that $\ker(G_a(a,\zeta,f))$ is
    two-dimensional and the result follows.
  \end{proof}

 \medskip

 \noindent
 {\bf Remark.} \label{rem periodic problem} The applicability of the Crandall-Rabinowitz Theorem
   relies on the simplicity of the kernel of the linearized equation, which we will check using
   Proposition~\ref{Prop simple kernels}. In the setting of $2\pi$-periodic functions
   simplicity of the kernel of the linearized equation never holds. This can be seen as follows:
   First notice that \eqref{Gl bifurcation condition}, which is a necessary condition for
   bifurcation for the Neumann problem \eqref{Gl freqcomb_neumann}, is also a necessary condition for
   bifurcation for the $2\pi$-periodic problem \eqref{Gl freqcomb}. The proof from above only needs small
   changes: the operator $D\phi := (-|d|\phi_1'' + \phi_1,-|d|\phi_2''+\phi_2)$, now equipped with periodic
   boundary conditions on $[0,2\pi]$, has a compact inverse $D^{-1}: H_{\rm per}\to H_{\rm per}$ where $H_{\rm
   per}$ denotes the restriction of $2\pi$-periodic functions from $H^1(\R;\R^2)$ to the interval $(0,2\pi)$.
   Furthermore, in the Fourier series expansion \eqref{fourier} the terms $\tilde \alpha^k \sin(kx)$ with
   vectors $\tilde \alpha^k\in\R^2$ additionally occur. The vanishing of $\det(dk^2\Id-N(a,\zeta))$ then
   appears in the same way as a necessary condition for the nontrivial solvability of the linear equation
   $-d\phi'' = N(a,\zeta)\phi$ by a $2\pi$-periodic function $\phi$. However, with 
   $\alpha^k \cos(kx)$ belonging to the $\ker(G_a(a,\zeta,f))$ for some $k\in \N$ also  
   $\alpha^k \sin(kx)$ belongs to the kernel making it at least
   two-dimensional. This is one of the reasons why we chose to consider synchronized solutions rather than
   periodic solutions.

 \medskip

 \subsection{Determination of all possible bifurcation points} \label{subsec Det bifpoints} ~

 \medskip

 First let us mention that the solutions of \eqref{Gl bifurcation condition} for $k=0$ do not give rise to
 bifurcation from $\hat\Gamma_f,\bar\Gamma_\zeta$ regardless of whether the assumptions (S),(T) are satisfied.
 This is not in contradiction with the Crandall-Rabinowitz Theorem for the following reason. In our analysis
 we use the parameterizations of $\hat\Gamma_f,\bar\Gamma_\zeta$ from Lemma~\ref{trivial} having the property
 that $t\mapsto \hat\zeta(t),s\mapsto \bar f(s)$ may not be injective for some parameter samples. In our
 bifurcation analysis related to $k=0$ this inconvenience leads to a false prediction of bifurcation in the
 following way.
 In order to keep the explanations short we explain the situation only for the bifurcation analysis related to
 $\hat\Gamma_f$. Since we use $t$ (and not $\zeta$) as the parameter in the Crandall-Rabinowitz theorem we
 find a bifurcating branch w.r.t. $t$ whenever \eqref{Gl bifurcation condition} as well as (S) and (T) are
 satisfied for some $k\in\N_0$ and some $t_0\in (-1,1)$. One can check that in the special case $k=0$ this
 is equivalent to saying that the curve $\hat\zeta$ has a turning point at $t_0$, i.e. we have
 $\hat\zeta'(t_0)=0,\hat\zeta''(t_0)\neq 0$. As a consequence, for any given $\eps$ close enough to $0$ there
 is a value $\delta_\eps$ converging to $0$ as $\eps\to 0$ such that $\delta_\eps\cdot \eps <0$ and
 $\hat\zeta(t_0+\eps)=\hat\zeta(t_0+\delta_\eps)$. Hence the bifurcation theorem detects the branch
 $(t_0+\eps,\hat a(t_0+\delta_\eps))$ bifurcating from $(t_0+\epsilon,\hat a(t_0+\epsilon))$ at $\epsilon=0$.
 Clearly, $(\hat\zeta(t_0+\eps),\hat a(t_0+\delta_\eps))=(\hat\zeta(t_0+\delta_\eps),\hat
 a(t_0+\delta_\eps))$ still lies on $\hat\Gamma_f$ and so this branch bifurcates with respect to the variable
 $t$, but not with respect to $\zeta$.

 For that reason the case $k=0$ will be left aside when we determine the possible bifurcation points.
 Note that this phenomenon could be avoided if we locally parameterized the trivial
 solution families $\hat\Gamma_f$ by $\zeta$. However, since this parameterization is in general not global
 further technical complications would arise.

 \medskip
 
\noindent
 {\it In Theorem \ref{Thm 3 bifurcation zeta}:} For given $f\in\R$ we have to determine all $t\in (-1,1)$
 such that $\hat F_a(0,t)=G_a(\hat a(t),\hat\zeta(t),f)$ has a nontrivial kernel. According to Proposition
 \ref{Prop kernels} this is the case if and only if there is $k\in\N_0$ such that
 \begin{equation} \label{Gl bifpoints 1}
   (\hat\zeta(t)+dk^2-2|\hat a(t)|^2)^2 = |\hat a(t)|^4-1.
 \end{equation}
 In particular this implies $1\leq |\hat a(t)|^2 = f^2(1-t^2)$ (see Lemma~\ref{trivial})
 so that $|f|\geq 1$ is a necessary condition for bifurcation from $\hat\Gamma_f$. Furthermore, in case
 $|f|\geq 1$, we obtain from the formulas for
 $\hat a(t),\hat\zeta(t)$ (see Lemma \ref{trivial}) and \eqref{Gl bifpoints 1}
 \begin{equation*}
   |t|\leq 1-|f|^{-2}\quad\text{and}\quad
   dk^2 = f^2(1-t^2)-\frac{t}{\sqrt{1-t^2}}-\sigma\sqrt{f^4(1-t^2)^2-1}
 \end{equation*}
 for some $\sigma\in\{-1, 1\}$ so that part (i) of  Theorem \ref{Thm 3 bifurcation zeta} is proved.

 \medskip

\noindent
 {\it In Theorem \ref{Thm 4 bifurcation f}:} Now let $\zeta\in\R$ be fixed.  Proposition~\ref{Prop
 kernels} and $|\bar a(s)|^2=s^2$ imply that the operator $\bar F_a(0,s)=G_a(\bar a(s),\zeta,\bar f(s))$ has
 a nontrivial kernel if and only if
 \begin{equation*}
   (\zeta+dk^2 - 2s^2)^2 = s^4-1.
 \end{equation*}
 This implies $(\zeta+dk^2)^2\geq 3$ and
 $s^2=\frac{2}{3}(\zeta+dk^2)-\frac{{\sigma}}{3}((\zeta+dk^2)^2-3)^{1/2}$ for some {$\sigma\in\{-1, 1\}$}. From
 the nonnegativity of $s^2$ we infer $\zeta+dk^2\geq 0$ and thus $\zeta+dk^2\geq \sqrt 3$.
 Hence, $|s|$ is given by the formula~\eqref{Gl bifpoints Thm4} for {$\sigma\in\{-1, 1\}$} and part
 (i) of Theorem \ref{Thm 4 bifurcation f} is proved.

 \medskip

 \subsection{Number of bifurcation points} \label{subsec number bifpoints} ~

  \medskip

\noindent
  {\it In Theorem \ref{Thm 3 bifurcation zeta}:}\; We have to prove that for all $f\in\R$ the
  trivial solution family $\hat\Gamma_f$ contains at most $\hat k(f)$ bifurcation points where
  $\hat k(f)$ was defined in Theorem~\ref{Thm 3 bifurcation zeta}~(ii). By Proposition~\ref{Prop kernels}
  every bifurcation point $(a,\zeta)\in\hat\Gamma_f$ satisfies the quadratic equation \eqref{Gl bifurcation
  condition} for some $k\in\N$. Hence every $k$ gives rise to at most two bifurcation points and therefore it suffices to prove $2k\leq \hat k(f)$. Formula \eqref{Gl
  bifurcation condition} implies
  \begin{equation*}
    0\leq (\zeta+dk^2-2|a|^2)^2=|a|^4-1.
  \end{equation*}
  This shows that bifurcation can only occur if $|a|\geq 1$ and together with \eqref{das_simple} 
  we get 
  \begin{equation*}
    f^2 \geq |a|^2 \quad \mbox{and}\quad f^2 \geq 1+ (|a|^2-\zeta)^2.
  \end{equation*}
  Substituting $\zeta$ from \eqref{Gl bifurcation condition} we obtain $\zeta+dk^2-2|a|^2=
  \pm \sqrt{|a|^4-1}$ and thus
  \begin{align*}
    |d|k^2
    &\leq  \big| dk^2-|a|^2\mp\sqrt{|a|^4-1} \big| + |a|^2+\sqrt{|a|^4-1} \\
    &= \big| |a|^2-\zeta| \big| + |a|^2 + \sqrt{|a|^4-1} \\
    &\leq \sqrt{f^2-1}+f^2+\sqrt{f^4-1}.
  \end{align*}
  From this inequality we directly conclude $2k\leq \hat k(f)$.

  \medskip

\noindent
 {\it In Theorem \ref{Thm 4 bifurcation f}:}\; Let $\zeta\in\R$. Arguing as above we find that
 every bifurcation point $(a,\zeta)\in\bar\Gamma_f$ satisfies equation \eqref{Gl bifpoints Thm4} for some
 $k\in\N, {\sigma\in \{-1, 1\}}$ and hence gives rise to at most four bifurcation points. Therefore, in case $d<0$, we have to prove $4k\leq \bar k(\zeta)$.
 Indeed, in that case the inequality $\zeta+dk^2\geq \sqrt{3}$ from \eqref{Gl bifpoints Thm4} implies
 $k\leq (|d|^{-1}(\zeta-\sqrt{3})_+)^{1/2}$ which is all we had to show. 

 \medskip

 \subsection{Simplicity of the kernels} ~ \label{simple_kernels}

  \medskip

  \noindent
  {\it In Theorem \ref{Thm 3 bifurcation zeta}:}\; Let $f\in\R$ and
  let $(\hat a(t),\hat\zeta(t))$ be a possible bifurcation point with respect to $\hat\Gamma_f$, i.e., we
  assume that $t\in (-1,1)$ satisfies equation \eqref{Gl bifpoints Thm3} for some $k\in\N$ and some $\sigma\in\{-1, 1\}$. Then
  Proposition~\ref{Prop simple kernels} implies that $\ker(\hat F_a(0,t))$ is one-dimensional if and only if
  we have
  \begin{equation} \label{simple_f}
  -k^2+d^{-1}(4|\hat a(t)|^2-2\hat\zeta(t))\notin (\N_0\sm\{k\})^2.
  \end{equation}
  Since we know from Lemma~\ref{trivial}(a) that
  \begin{equation*}
    4|\hat a(t)|^2-2\hat\zeta(t)
    = 2f^2(1-t^2)-2t(1-t^2)^{-1/2}
  \end{equation*}
  \eqref{simple_f} is guaranteed by condition (S) of Theorem \ref{Thm 3 bifurcation zeta}~(iii) and we are done.

  \medskip

\noindent
  {\it In Theorem \ref{Thm 4 bifurcation f}:}\;  Let $\zeta\in\R$ and let $(\bar a(s),\bar f(s))$ be
  a possible bifurcation point with respect to $\bar\Gamma_\zeta$, i.e., we assume that
  $s$ is given by \eqref{Gl bifpoints Thm4} for some $k\in\N, {\sigma\in\{-1, 1\}}$. As above, Lemma~\ref{trivial}(b) implies the
  equation
  \begin{equation*}
    4|\bar a(s)|^2-2\zeta
    = 4s^2-2\zeta
    \stackrel{\eqref{Gl bifpoints Thm4}}{=}
    \frac{2}{3}\Big(\zeta+4dk^2-2{{\sigma}}\sqrt{(\zeta+dk^2)^2-3}\Big)
  \end{equation*}
  which shows that condition (S) from Theorem \ref{Thm 4 bifurcation f}~(iii) guarantees the
  simplicity of $\ker(\bar F_a(0,s))$.

  \medskip

 \subsection{Transversality condition} ~ \label{transversality}

  \medskip

  In the calculations related to the verification of the transversality condition we will use the following short-hand notations.
  In the context of Theorem \ref{Thm 3 bifurcation zeta} where $(t,k)$ is a solution of \eqref{Gl bifpoints
  Thm3} we write
  \begin{equation*}
    a=(\hat a_1(t),\hat a_2(t)),\quad\zeta=\hat\zeta(t), \quad \dot a = \left(\frac{d \hat a_1}{d
    t}(t),\frac{d \hat a_2}{d t}(t)\right),\quad \dot\zeta=\frac{d\hat\zeta}{dt}(t)
  \end{equation*}
  and in the context of Theorem \ref{Thm 4 bifurcation f} where $(s,k)$ is a solution of \eqref{Gl bifpoints Thm4} we write
  \begin{equation*}
    a=(\bar a_1(s),\bar a_2(s)),\quad f=\bar f(s), \quad \dot a = \left(\frac{d \bar a_1}{d s}(s),\frac{d
    \bar a_2}{d s}(s)\right),\quad \dot f=\frac{d\bar f}{ds}(s).
  \end{equation*}
  Furthermore we will use
  \begin{equation*}
    \phi(x):=\phi_k(a)(x)=\alpha
    \cos(kx),\quad \phi^*(x):=\phi_k^*(a)(x)=\beta\cos(kx),
  \end{equation*}
  where the vectors $\alpha=\alpha^k, \beta=\beta^k\in \R^2$ were defined in \eqref{Gl formula phik},
  \eqref{Gl formula phikstar}. We have to check the transversality condition in the possible bifurcation
  points that we determined in Section~\ref{subsec Det bifpoints}. Hence we may use \eqref{Gl bifurcation
  condition}, i.e.,
  \begin{equation*}
    (\zeta+dk^2)^2-4(\zeta+dk^2)|a|^2+1+3|a|^4=0.
  \end{equation*}
  In view of the formulas \eqref{Gl formula phik} and \eqref{Gl formula phikstar} we will have to
  investigate the following cases:
  \begin{itemize}
    \item[] {\it Case (1):} $|a_1a_2|\neq \frac{1}{2}$ or $3a_1^2+a_2^2\neq \zeta+dk^2$
    \item[] {\it Case (2):} $a_1a_2=\frac{1}{2}$ and $3a_1^2+a_2^2=\zeta+dk^2$
    \item[] {\it Case (3):} $a_1a_2=-\frac{1}{2}$ and $3a_1^2+a_2^2=\zeta+dk^2$
  \end{itemize}
  In all three cases Proposition \ref{prop derivatives} yields the following formula for every constant
  $\psi\in \R^2\subset H$:
  \begin{align*}
   T(\psi)
   := \langle G_{aa}(a,\zeta,f)[\phi,\psi], \phi^* \rangle_H
   =&  \langle DG_{aa}(a,\zeta,f)[\phi,\psi], \phi^* \rangle_{L^2} \\
   =&
     - \Big\langle
      \begin{pmatrix}
        \phi^TM_1(a)\psi\\
        \phi^TM_2(a)\psi
      \end{pmatrix}
      ,\phi^* \Big\rangle_{L^2} \\
   \stackrel{\eqref{Gl formula phik}}{=}&
     - \Big\langle
      \begin{pmatrix}
        \alpha^TM_1(a)\psi\\
        \alpha^T M_2(a)\psi
      \end{pmatrix}
      \cos(k\cdot)
      ,\beta \cos(k\cdot)\Big\rangle_{L^2} \\
   =&  -\begin{pmatrix}
        \alpha^TM_1(a)\psi\\
        \alpha^TM_2(a)\psi
      \end{pmatrix}^T  \beta \int_0^\pi \cos(kx)^2 \,dx\\
    =&
   - {\pi}\, \alpha^T
     \begin{pmatrix}
       3a_1\beta_1 + a_2\beta_2 & a_1\beta_2+a_2\beta_1  \\
       a_1\beta_2+a_2\beta_1 & a_1\beta_1 +  3a_2\beta_2
     \end{pmatrix}
      \psi
    \end{align*}
    Using  \eqref{Gl formula phik}, \eqref{Gl formula phikstar} and
    $(\zeta+dk^2)^2-4|a|^2(\zeta+dk^2)+1+3|a|^4=0$  we find in case (1)
    \begin{align*}
      &\; T(\psi) \\
      &= -{\pi}
     \begin{pmatrix}
       a_1(-3+6a_1^2a_2^2-3a_2^4+9a_1^4+ (2a_2^2-6a_1^2)(\zeta+dk^2)+(\zeta+dk^2)^2) \\
       a_2(-1+15a_1^4+3a_2^4+18a_1^2a_2^2- (14a_1^2+6a_2^2)(\zeta+dk^2)+3(\zeta+dk^2)^2)
    \end{pmatrix}^T\psi  \\
     &=  -{\pi} \begin{pmatrix}
       a_1(-4+6a_1^4-6a_2^4+(6a_2^2-2a_1^2)(\zeta+dk^2)) \\
       a_2(-4+6a_1^4-6a_2^4+(6a_2^2-2a_1^2)(\zeta+dk^2))
    \end{pmatrix}^T\psi  \\
    &= -{2\pi} (-2+3a_1^4-3a_2^4+(3a_2^2-a_1^2)(\zeta+dk^2)) a^T\psi.
  \end{align*}
  In case (2) one has $\alpha=(2(a_2^2-a_1^2),-2)^T,\beta=(-2,0)^T$ and using $a_1a_2=\frac{1}{2}$ one
  obtains
  \begin{equation*}
    T(\psi)
     = {2\pi} \alpha^T \begin{pmatrix}
     3a_1 & a_2 \\
     a_2 & a_1
     \end{pmatrix}\psi
     = {2\pi} \begin{pmatrix}
       2a_1(3a_2^2-3a_1^2-\frac{a_2}{a_1}) \\
       2a_2(a_2^2-a_1^2-\frac{a_1}{a_2})
     \end{pmatrix}^T\psi
     = -{2\pi}(6a_1^2-2a_2^2)a^T\psi
  \end{equation*}
  while in case (3) we may use $\alpha=(2,0)^T,\beta=(2(a_2^2-a_1^2),2)^T$ and $a_1a_2=-\frac{1}{2}$ to get
  \begin{equation*}
    T(\psi)
     = -{2\pi} \begin{pmatrix}
       3a_1\beta_1+a_2\beta_2 \\
       a_1\beta_2+a_2\beta_1
     \end{pmatrix} \psi
     = -{2\pi}(2a_2^2-6a_1^2)a^T\psi.
  \end{equation*}
  Summarizing these calculations we find
  \begin{align}
      T(\psi)
      = -{2\pi}\, a^T\psi \cdot \begin{cases}
        -2+3a_1^4-3a_2^4+(3a_2^2-a_1^2)(\zeta+dk^2) &\text{ in case }(1), \\
        6a_1^2- 2a_2^2 &\text{ in case }(2), \\
        2a_2^2-6a_1^2 &\text{ in case }(3).
      \end{cases}
      \label{Gl formula transversality0}
  \end{align}
  In a similar way we obtain in case (1)
  \begin{align*}
   \langle \phi,\phi^*\rangle_{L^2}
   &= \alpha^T \beta \int_0^\pi \cos(kx)^2\,dx  \\
   &= {\frac{\pi}{2}} \, \big(-1+4a_1^2a_2^2 + (3a_1^2+a_2^2-\zeta-dk^2)^2\big) \\
   &= {\frac{\pi}{2}}\, \Big( (\zeta+dk^2)^2-2(\zeta+dk^2)(3a_1^2+a_2^2) -1+4a_1^2a_2^2 +
   (3a_1^2+a_2^2)^2  \Big).
 \end{align*}
 Using again $(\zeta+dk^2)^2-4|a|^2(\zeta+dk^2)+1+3|a|^4=0$ in case (1) and performing the corresponding
 calculations for the cases (2) and (3) we arrive at
  \begin{align}
   \langle \phi,\phi^*\rangle_{L^2}
   &= {\pi} \begin{cases}
      (\zeta+dk^2)(-a_1^2+a_2^2) -1+3a_1^4 + 2a_1^2a_2^2 - a_2^4 &\text{ in case }(1), \\
      2a_1^2- 2a_2^2 & \text{ in case }(2), \\
      2a_2^2-2a_1^2 & \text{ in case } (3).
   \end{cases}
   \label{Gl formula transversality1}
 \end{align}
 Now we are going to use the formulas \eqref{Gl formula transversality0},\eqref{Gl formula
 transversality1} in the concrete settings of Theorem~\ref{Thm 3 bifurcation zeta} and Theorem~\ref{Thm 4
 bifurcation f}.

  \medskip

 \noindent
 {\it In Theorem \ref{Thm 3 bifurcation zeta}:}\;  Let $f\in\R$ and let $t\in (-1,1)$ satisfy equation
 \eqref{Gl bifpoints Thm3} for some $k\in\N$ and some $\sigma\in \{-1,1\}$. Since $\dot a\in \R^2\subset H$
 is a constant vector the formulas \eqref{Gl formula transversality0},\eqref{Gl formula transversality1} and
 Proposition~\ref{prop derivatives} yield in case (1)
 \begin{align*}
   &\;\langle \hat F_{at}(0,t)[\phi],\phi^*\rangle_H \\
   &\stackrel{\eqref{Gl defn hatF, barF}}{=} \langle G_{aa}(a,\zeta,f)[\phi,\dot a],\phi^*\rangle_H +
   \dot\zeta \langle G_{a\zeta}(a,\zeta,f)[\phi],\phi^*\rangle_H  \\
   &\stackrel{\eqref{Gl formula transversality0}}{=}
   - {2\pi} \, a^T \dot a\, (-2+3a_1^4-3a_2^4+(3a_2^2-a_1^2)(\zeta+dk^2))
    + \dot\zeta \langle D^{-1}\phi , \phi^* \rangle_H \\
   &= -{2\pi}\, a^T \dot a\, (-2+3a_1^4-3a_2^4+(3a_2^2-a_1^2)(\zeta+dk^2))
    + \dot\zeta \langle \phi , \phi^* \rangle_{L^2} \\
   &\stackrel{\eqref{Gl formula transversality1}}{=}
   {\pi} \Big( -(-4+6a_1^4-6a_2^4+(6a_2^2-2a_1^2)(\zeta+dk^2))a^T \dot a  \\
   &\qquad  +  \dot \zeta \big((\zeta+dk^2)(-a_1^2+a_2^2) -1+3a_1^4 + 2a_1^2a_2^2 - a_2^4 \big)\Big).
 \end{align*}
 We will now insert the trivial solution $(a_1,a_2,\zeta)=(\hat a_1(t),\hat a_2(t),\hat\zeta(t))$ from Lemma~\ref{trivial} and use the identities
  $$
  a^T \dot a
  = \frac{d}{dt} \frac{|\hat a(t)|^2}{2}
   = \frac{d}{dt} \frac{f^2(1-t^2)}{2}
   = -f^2t, \qquad \dot \zeta = -2f^2t+ (1-t^2)^{-3/2}.
 $$
Notice also that the necessary condition \eqref{Gl bifurcation condition} for
bifurcation becomes
$$
  \zeta + dk^2 = 2f^2(1-t^2)-\sigma \sqrt{f^4(1-t^2)^2-1}
$$
so that \eqref{Gl bifpoints Thm3} is proved. After a lengthy computation we obtain
 \begin{align*}
  &\; \langle \hat F_{at}(0,t)[\phi],\phi^*\rangle_H \\
  &= {\pi} \Big( 4 f^6
 t^3(1-t^2)^2+f^4(1-t^2)^{1/2}-2tf^2-(1-t^2)^{-3/2}\\
 & -\sigma \sqrt{f^4(1-t^2)^2-1}\Bigl(4f^4 t^3(1-t^2)+f^2(2t^2-1)(1-t^2)^{-1/2}\Bigr) \Big)
 \end{align*}
 and hence the transversality condition (H2) from Theorem~\ref{Thm Crandall-Rabinowitz} is satisfied
 whenever the right-hand side is nonzero, i.e. when condition (T) holds. In case (2) or (3) the transversality
 condition is always satisfied. Indeed, proceeding as above and using the explicit formulas
 \begin{equation*}
   a_1^2=\frac{\sqrt{1-t^2}}{2|t|},\qquad a_2^2=\frac{|t|}{2\sqrt{1-t^2}},\qquad f^2 =
   \frac{1}{2|t|(1-t^2)^{3/2}}
 \end{equation*}
 from Lemma~\ref{trivial}~(a) we find in case (2), where $t<0$, that
 \begin{align*}
    \langle \hat F_{at}(0,t)[\phi],\phi^*\rangle_H
   &= T(\dot a) + \dot\zeta \langle \phi , \phi^* \rangle_{L^2} \\
   &\stackrel{ \eqref{Gl formula transversality0}}{=} -{2\pi}\, a^T\dot a\,(6a_1^2-2a_2^2) +
   \dot\zeta\cdot {\pi}(2a_1^2-2a_2^2) \\
   &= {2\pi}\, \big( f^2t(6a_1^2-2a_2^2)+(-2f^2t+(1-t^2)^{-3/2})(a_1^2-a_2^2)  \big) \\
   &= {2\pi}\,\big( f^2t\cdot 4a_1^2 + (1-t^2)^{-3/2}(a_1^2-a_2^2)\big) \\
   &= - \frac{{\pi}}{|t|(1-t^2)^2}
    \neq 0
 \end{align*}
 and similarly in case (3), where $t>0$, that
 $$
   \langle \hat F_{at}(0,t)[\phi],\phi^*\rangle_H
   = + \frac{{\pi}}{|t|(1-t^2)^2}
   \neq 0.
 $$

 \medskip

  \noindent
  {\it In Theorem \ref{Thm 4 bifurcation f}:}\;  Now let $\zeta\in\R$ and let $s$ be given by \eqref{Gl
  bifpoints Thm4}. Since $\dot a\in \R^2\subset H$ is a constant vector we get from \eqref{Gl formula
  transversality0} and Proposition~\ref{prop derivatives} in case (1)
  \begin{align*}
   \langle \bar F_{as}(0,s)[\phi],\phi^*\rangle_H
   &\stackrel{\eqref{Gl defn hatF, barF}}{=} \langle G_{aa}(a,\zeta,f)[\phi,\dot a],\phi^*\rangle_H + \dot f \langle
   \underbrace{G_{af}(a,\zeta,f)[\phi]}_{=0} ,\phi^*\rangle_H
   \\
   &\stackrel{ \eqref{Gl formula transversality0}}{=}
     - {2\pi}\, a^T\dot a\,(-2+3a_1^4-3a_2^4+(3a_2^2-a_1^2)(\zeta+dk^2)).
 \end{align*}
 Recalling
  $$
  a^T \dot a
   = \frac{d}{ds} \frac{|\bar a(s)|^2}{2}
   = \frac{d}{ds} \frac{s^2}{2}= s
 $$
 we obtain
 $$
   \langle \bar F_{as}(0,s)[\phi],\phi^*\rangle_H =  -
   {2\pi}s(-2+3a_1^4-3a_2^4+(3a_2^2-a_1^2)(\zeta+dk^2)).
$$
Now we substitute $(a_1, a_2)=(\bar a_1(s), \bar a_2(s))$ from Lemma~\ref{trivial} and the value $s$
from \eqref{Gl bifpoints Thm4} of Theorem~\ref{Thm 4 bifurcation f} to obtain after a lengthy calculation
  \begin{align*}
    &\; \langle \bar F_{as}(0,s)[\phi],\phi^*\rangle_H \\
    =& \frac{{2\pi\sigma} s}{27(1+(s^2-\zeta)^2)} \sqrt{(\zeta+dk^2)^2-3}
 \Big(2\zeta+ 5dk^2-4{\sigma}\sqrt{(\zeta+dk^2)^2-3}\Big) \\
& \cdot \Big(4\zeta+dk^2-2{\sigma}\sqrt{(\zeta+dk^2)^2-3}\Big)
 \underbrace{\Big(\zeta+dk^2+{\sigma} \sqrt{(\zeta+dk^2)^2-3}\Big)}_{\not =0}.
  \end{align*}
  Hence, using $s\neq 0$ (from \eqref{Gl bifurcation condition} we get $|s|=|a(s)|=|a|\geq 1$) we find that
  the transversality condition holds if and only if
  \begin{align*}
     \zeta+dk^2\neq\sqrt{3}\quad
     &\text{and}\quad
     4\zeta+dk^2- 2{\sigma}\sqrt{(\zeta+dk^2)^2-3}\neq 0 \\
     &\text{and}\quad 2\zeta+5dk^2-4{\sigma}\sqrt{(\zeta+dk^2)^2-3}\neq 0
  \end{align*}
  which is precisely assumption (T) in Theorem \ref{Thm 4 bifurcation f}.

  \medskip

  In case (2) and (3) we have $|T(\dot a)| = {2\pi}|s||3a_1^2-a_2^2|$. So let us identify
  those values of $a_1, a_2, s, \zeta, d$ and $k\in \N_0$ where
  \begin{equation*}
    a_1^2a_2^2=\frac{1}{4}, \quad 3a_1^2=a_2^2 \mbox{ and } 3a_1^2+a_2^2=\zeta+dk^2 \geq \sqrt{3},
  \end{equation*}
  i.e., the situations in case (2), (3) where transversality fails. Since in this case $a_1^2= 1/\sqrt{12}$, $a_2^2= 3/\sqrt{12}$ we see that $s^2=a_1^2+a_2^2=2/\sqrt{3}$, i.e.,
  \begin{equation*}
  s = \pm \sqrt{ \frac{2}{\sqrt{3}}}\; \mbox{ and either }\;\zeta = -\frac{1}{\sqrt{3}}\; \mbox{ or }\;
    \zeta = \frac{5}{\sqrt{3}}.
  \end{equation*}
  In all of these cases $3a_1^2+a_2^2 = \sqrt{3}$, i.e, $d$ and $k$ must be such that $\zeta^2+dk^2=
  \sqrt{3}$ so that the necessary conditions of bifurcation is satisfied but transversality fails. However,
  this is already covered by condition (T) which excludes $\zeta+dk^2=\sqrt{3}$. This finishes part (iii) of
  Theorem~\ref{Thm 4 bifurcation f}.

  
  \section{Illustrations} \label{illustrations}

  
  Below we highlight in five subsections different interesting features of our bifurcation results. All
  bifurcation diagrams are plotted with respect to the $L^2$-norm of the solutions over the interval
  $(0,\pi)$. The $L^2$-norm of a given solution represents the electrical power of the correpsonding state inside the ring resonator.
  In order to illustrate our main results from Theorem~\ref{Thm 3 bifurcation zeta} and Theorem~\ref{Thm 4
  bifurcation f} we numerically followed the primary bifurcation branches which emanate from the trivial curve.
  So-called points of secondary bifurcations, i.e., bifurcation points on these primary branches were very often detected by AUTO. However, for the sake of clarity, most of them were not included into our
  diagrams. In Sections~\ref{zeta=0,d=0.1} and \ref{zeta=10,d=-0.2} we illustrate the results of Theorem~\ref{Thm 4
  bifurcation f}  for fixed $\zeta$ and for positive resp. negative dispersion coefficient $d$. In Sections~\ref{f=1.6,d=0.1} and
  \ref{f=2,d=-0.1} we show corresponding bifurcation diagrams based on Theorem~\ref{Thm 3 bifurcation zeta}. In Section~\ref{dynamic} we demonstrate that dynamic detuning is a method to generate soliton combs as in \cite{Herr2013}.
  
  \subsection{The case $\zeta=0,d=0.1$} \label{zeta=0,d=0.1}
  
  
  Here we have infinitely many bifurcation points on the trivial branch. Four pairs of those bifurcation
  points are shown. The trivial branch has no turning point and we only observe periodic Turing patterns. This means
  that the bifurcating solutions look as if they were multiples of the functions lying in the kernel of the
  linearized operator at the bifurcation point. In particular, no solitons were found.
  
\begin{figure}[H]\label{branches}
\centering
\begin{minipage}[c]{0.41\textwidth}
  The table provides some bifurcation points on the trivial branch. The bifurcation
  points at $s$ given by \eqref{Gl bifpoints Thm4}  
  are predicted by Theorem~\ref{Thm 4 bifurcation f} since the conditions (S) and (T) are satisfied.
  The $f-$values found by AUTO are listed as well. 
\end{minipage} \quad
\begin{minipage}[c]{0.55\textwidth}
\tiny
\centering 
\begin{tabular}{|c|c|c|c|c|c|c|}
\hline
$k$ & $\sigma$ & $s$ & curve & $f$ AUTO & $f$ Thm~\ref{Thm 4 bifurcation f} &  label \\
\hline
5 & ~1 & 1.03235 & red & ~1.50873 & ~1.50871 & 1\\
5 & -1 & 1.50582 & red & ~3.73196 & ~3.73195 & 5 \\
6 & ~1 & 1.16104 & green & ~1.94874 & ~1.94874 & 2\\
6 & -1 & 1.85795 & green & ~6.67731 & ~6.67731 & 6 \\
7 & ~1 & 1.31863 & blue & ~2.64489 &  ~2.64494 & 3\\
7 & -1 & 2.18965 & blue & 10.72430 & 10.72430 & 7\\
8 & ~1 & 1.48760 & black & ~3.61248 & ~3.61248 & 4\\
8 & -1 & 2.51404 & black & 16.08740 & 16.08736 & 8\\
\hline
\end{tabular}
\caption{Bifurcation points on trivial branch.}
\label{list_branches_1}
\end{minipage}
\end{figure} 
\vspace{-\baselineskip}
\begin{figure}[H]
\begin{minipage}[c]{0.4\textwidth}
\includegraphics[width=\textwidth]{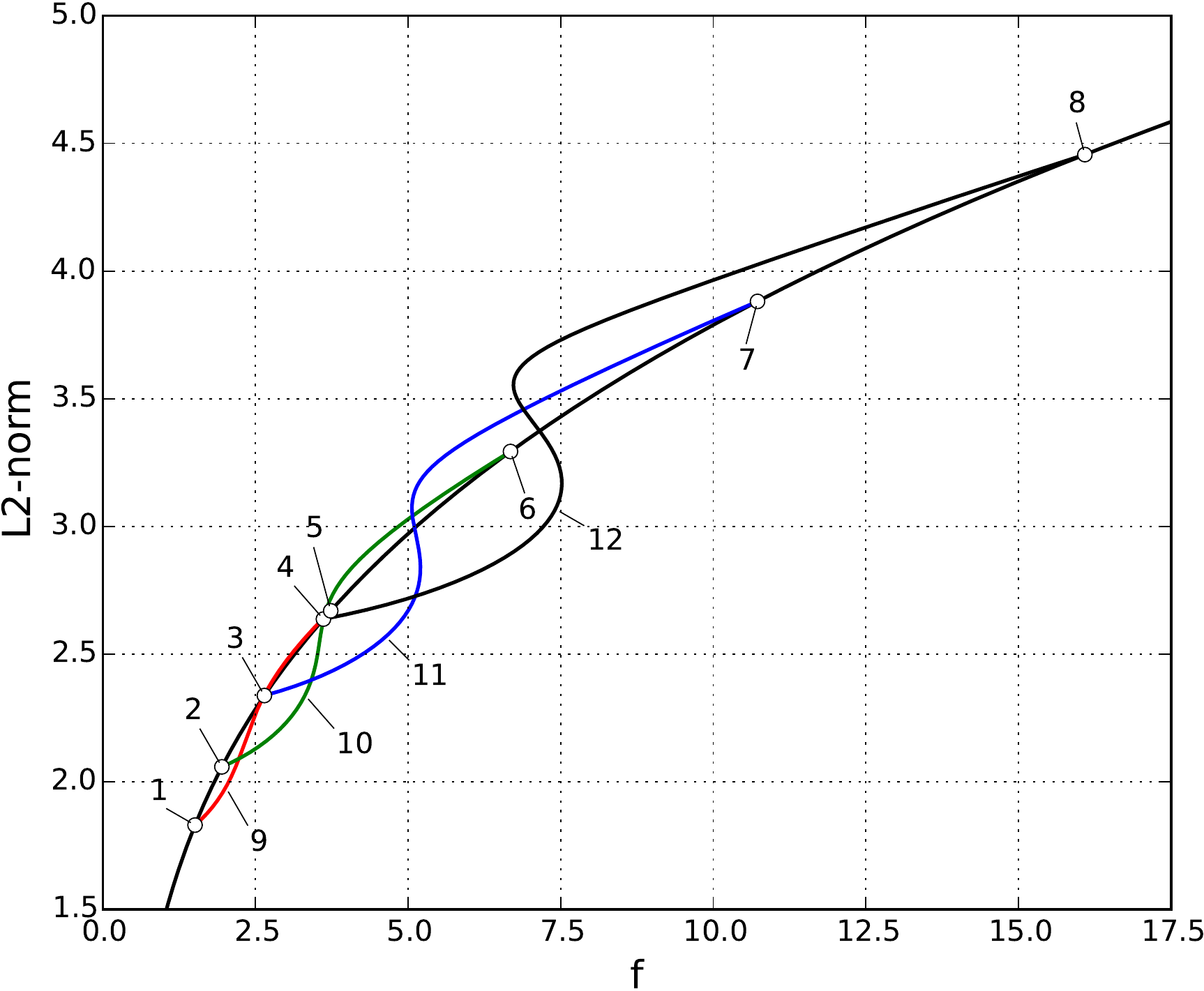}
 \caption{Bifurcation diagram}
\end{minipage} \qquad
\begin{minipage}[c]{0.54\textwidth}
\vspace{-2\baselineskip}
  Four bifurcating branches are depicted in the bifurcation diagram. All branches return to the trivial one.
  Their end points correspond to the same~$k$. 
  The labels 9-12 indicate solutions with $k$ maxima ($k\in\{5,6,7,8\}$) as shown below. They can be found on
  the branches emanating from bifurcation points associated to $k$ via Theorem~\ref{Thm 4 bifurcation f} (or
  Figure~\ref{list_branches_1}).
\end{minipage}
\end{figure}

\vspace{-1.1cm}

\begin{figure}[H]
\begin{subfigure}{0.49\textwidth}
\includegraphics[width=0.502\textwidth]{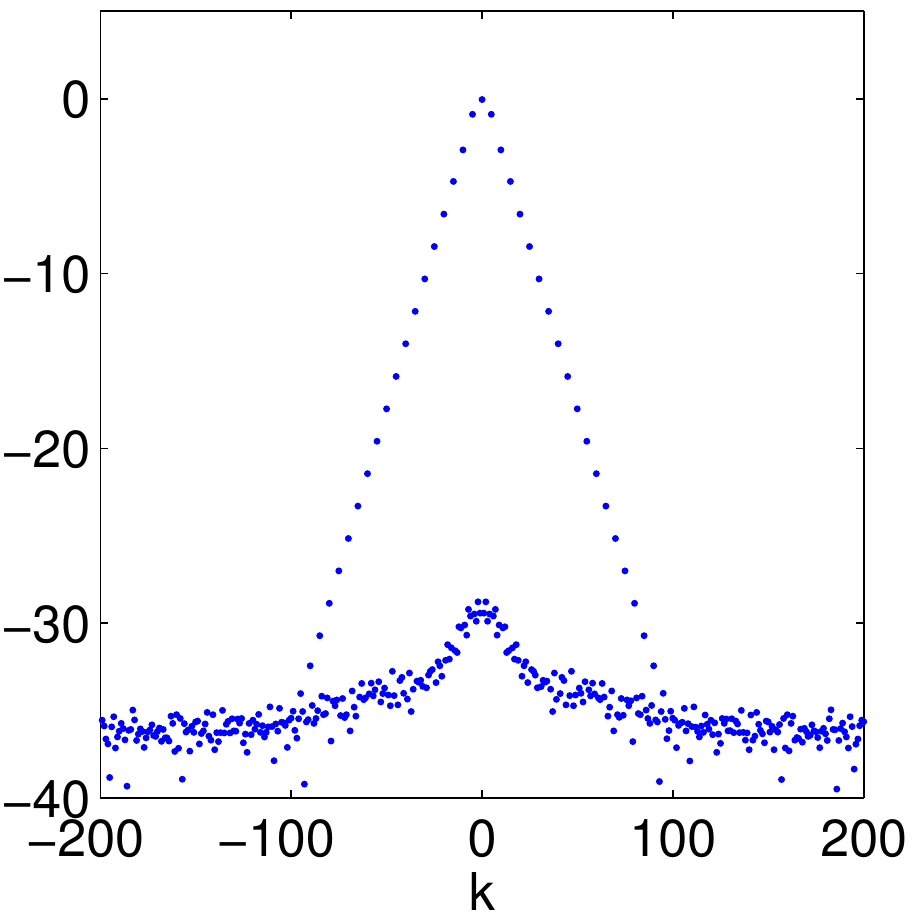}
\includegraphics[width=0.474\textwidth]{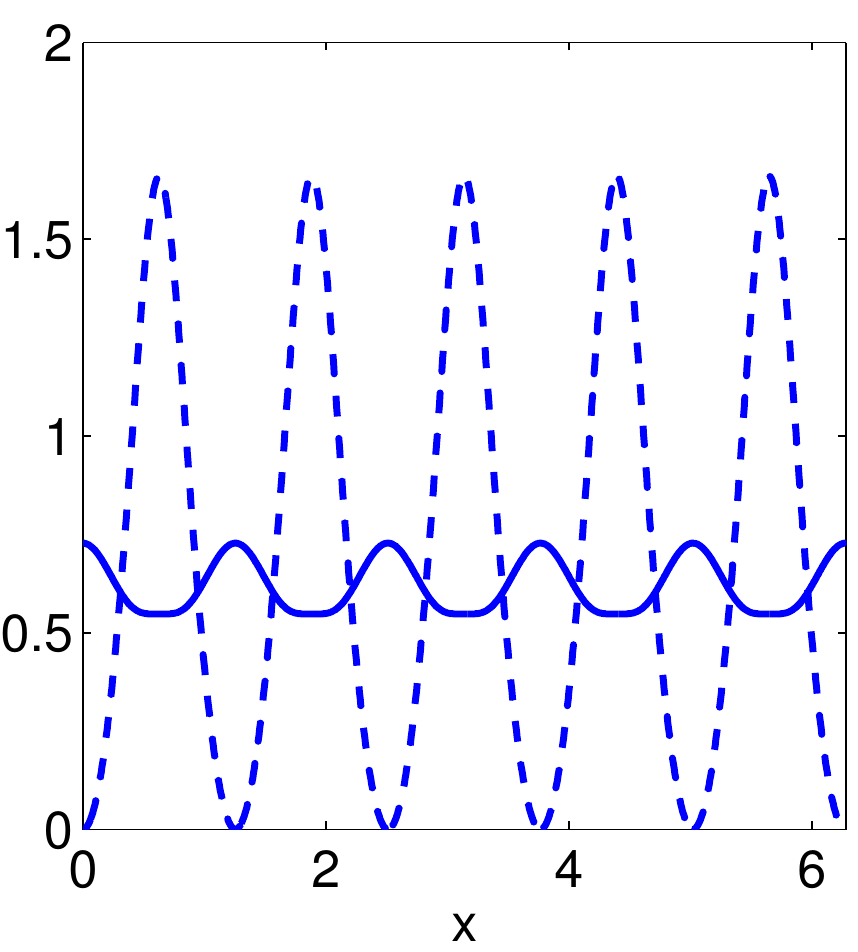}
\subcaption{Label 9: $f=2.01207$}
\end{subfigure}
\begin{subfigure}{0.49\textwidth}
\includegraphics[width=0.502\textwidth]{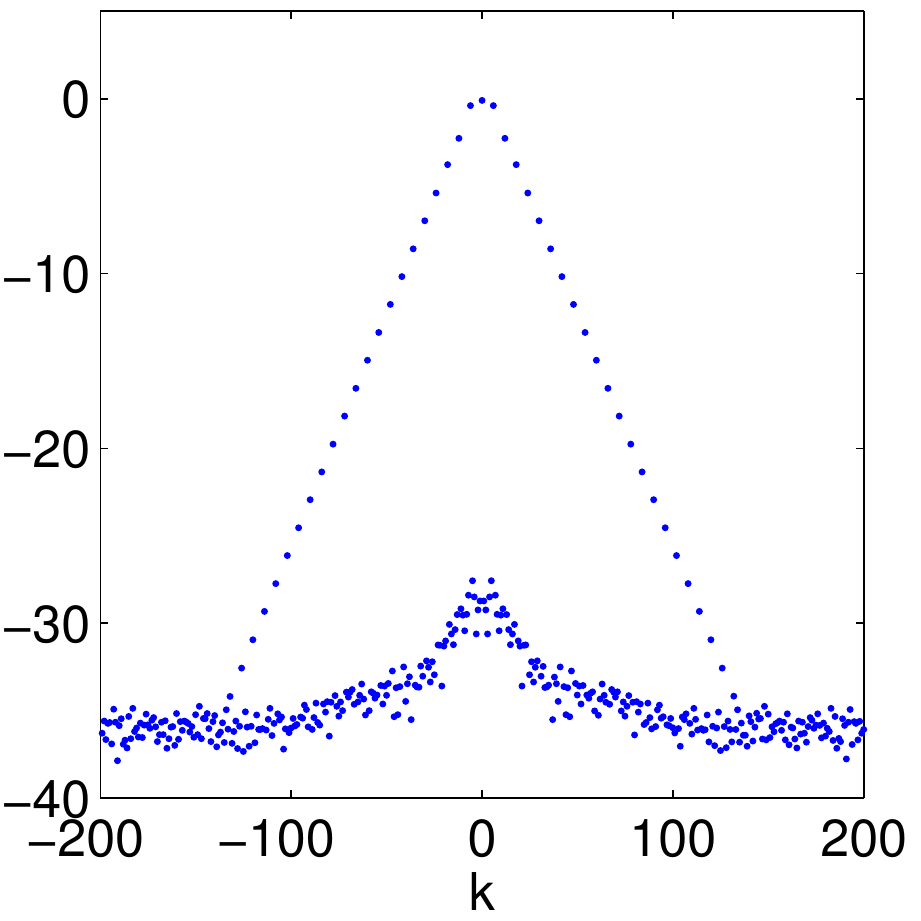}
\includegraphics[width=0.482\textwidth]{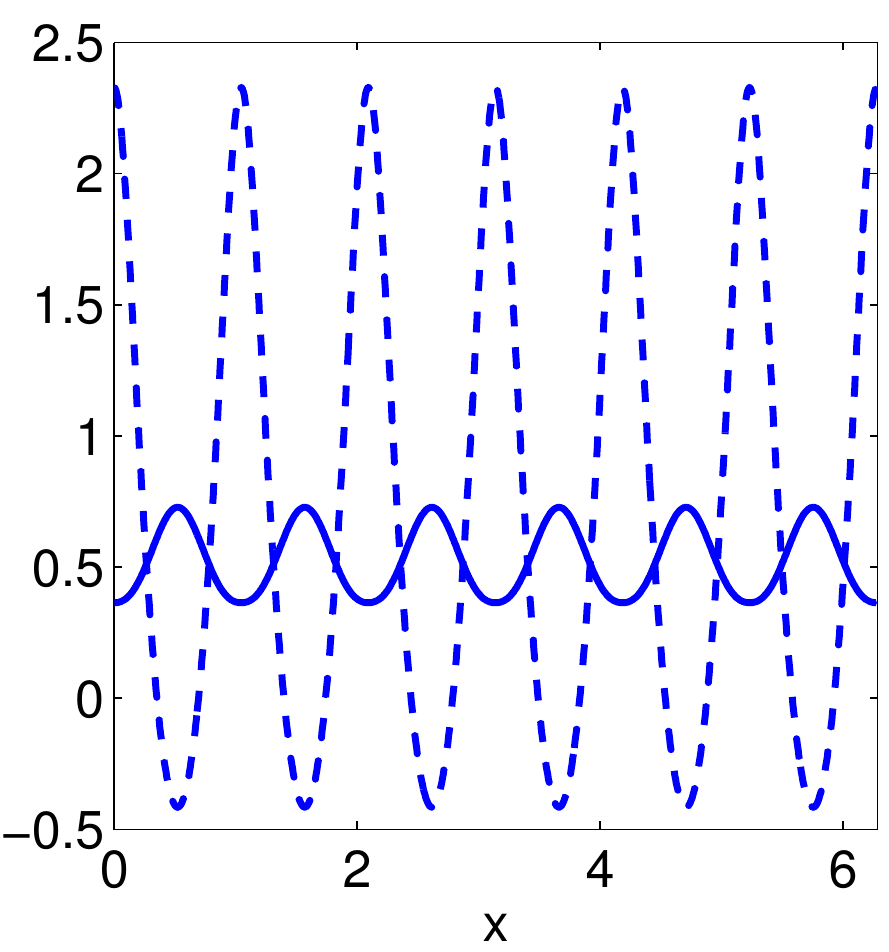}
\subcaption{Label 10: $f=3.31010$}
\end{subfigure}
\begin{subfigure}{.49\textwidth}
\includegraphics[width=0.505\textwidth]{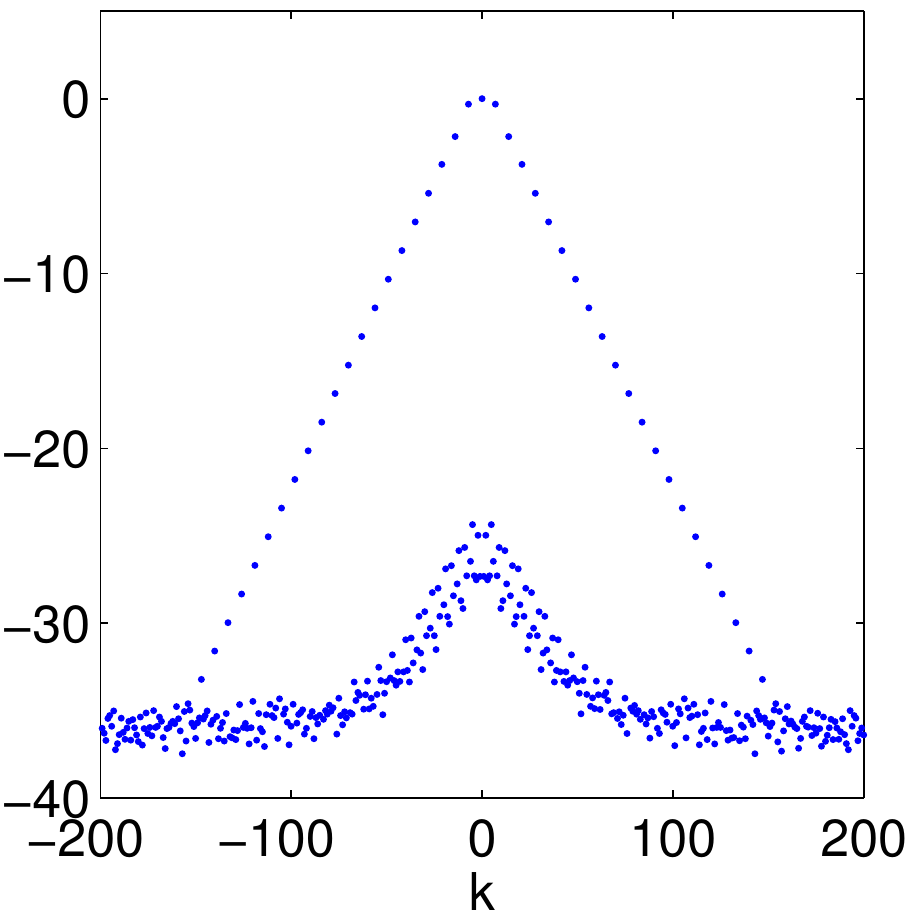}
\includegraphics[width=0.481\textwidth]{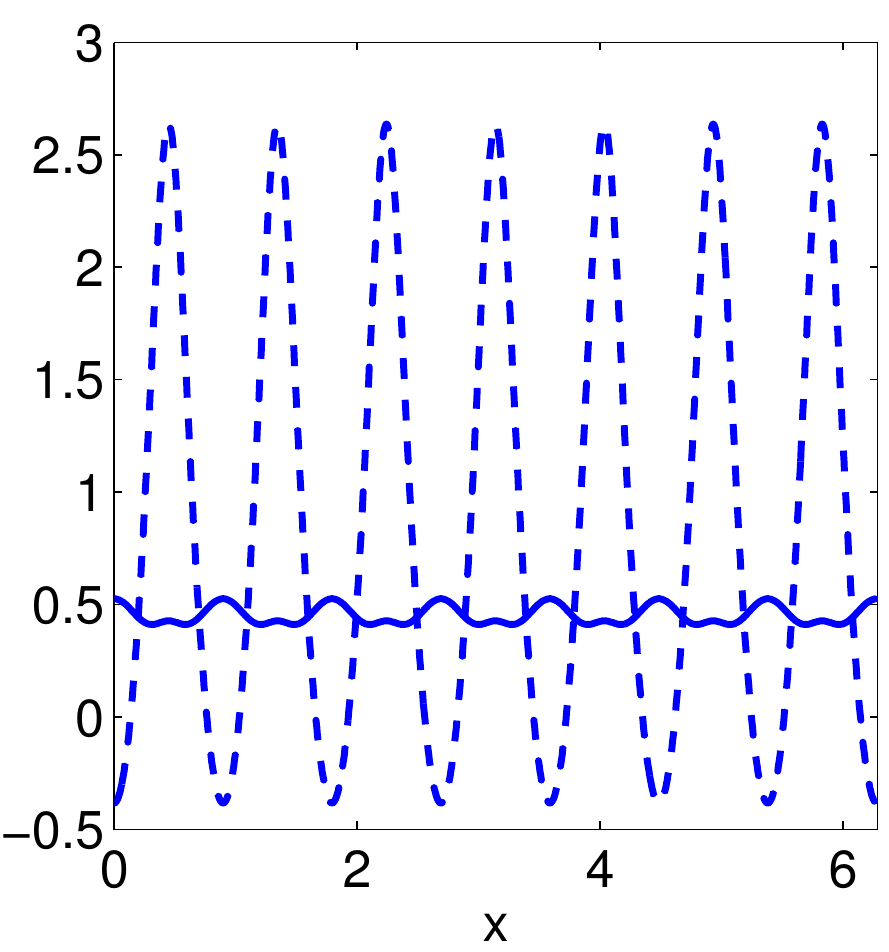}
\subcaption{Label 11: $f=4.63260$}
\end{subfigure}
\begin{subfigure}{.49\textwidth}
\includegraphics[width=0.49\textwidth]{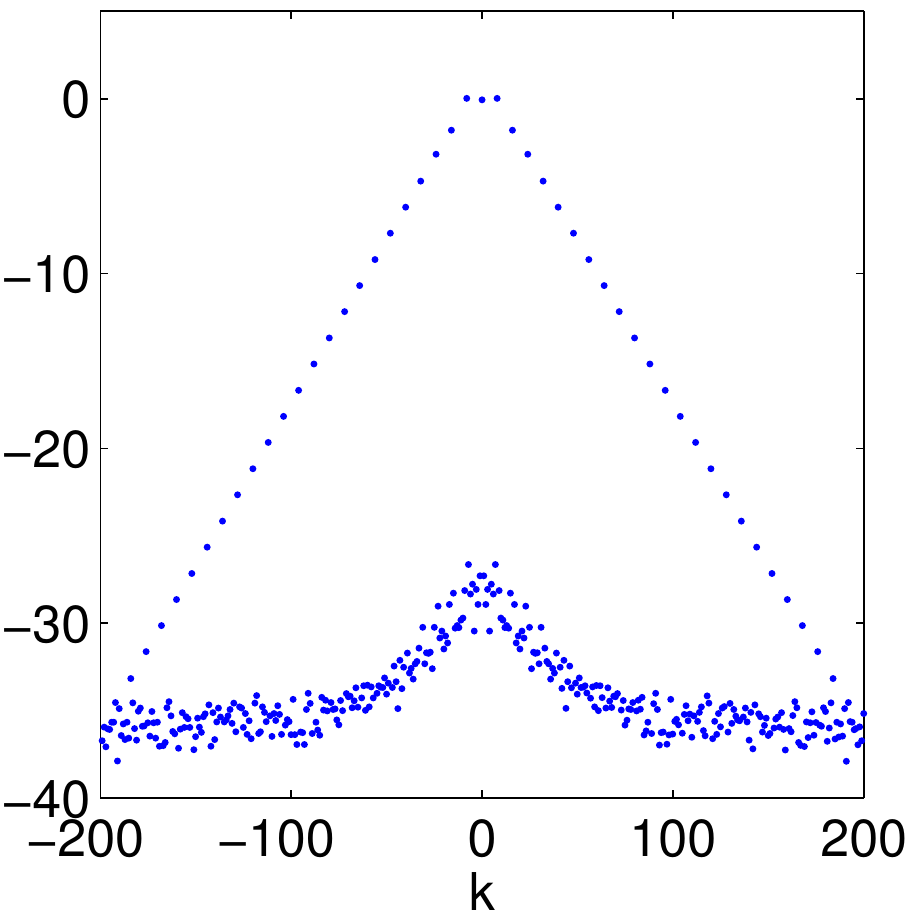}
\includegraphics[width=0.452\textwidth]{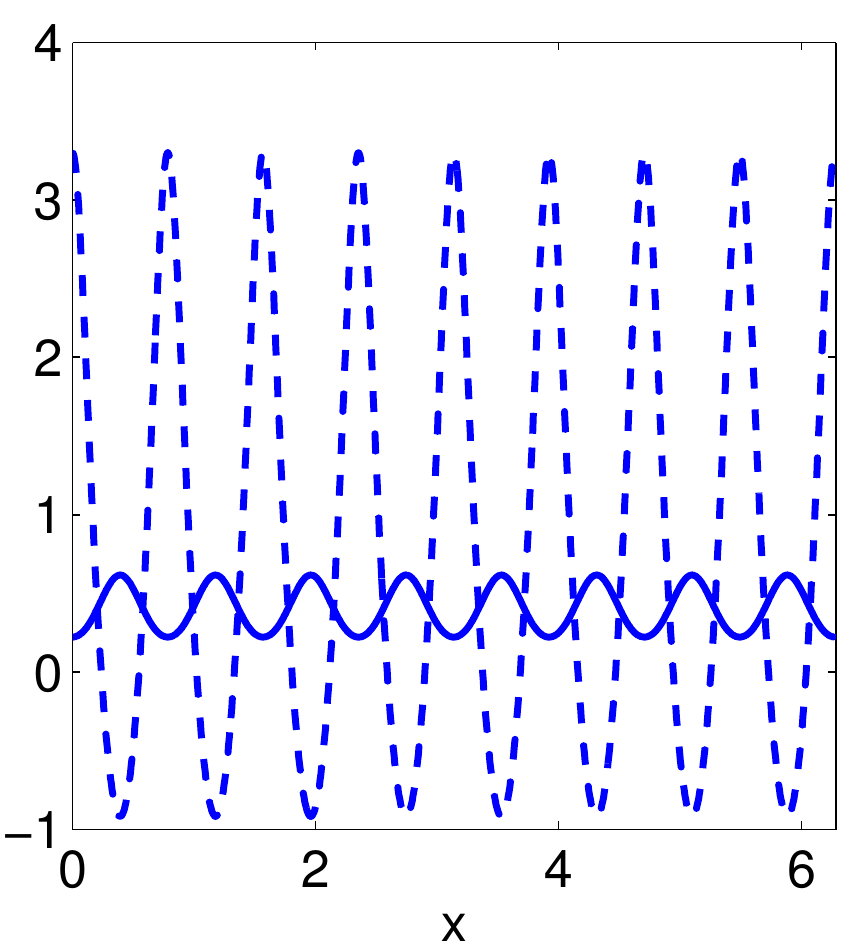}
\subcaption{Label 12: $f=7.41317$}
\end{subfigure}
\caption{Selected solutions. Left: $\log(|\hat a(k)|)$; right: solid $a_1(x)$, dashed $a_2(x)$.}
\label{turing_f_2}
\end{figure} 

\vspace{-0.5cm}




\subsection{The case $\zeta=10,d=-0.2$} \label{zeta=10,d=-0.2}


In this example there are only finitely many (namely 12) bifurcation points on the trivial branch. The
trivial branch has turning points and we find dark solitons. They occur in most pronounced form at the
turning points of two nontrivial branches.

\begin{figure}[H]\label{branches_2}
\centering
\begin{minipage}[c]{0.38\textwidth}
  \vspace{-\baselineskip}
  The table provides all bifurcation points on the trivial branch. The bifurcation
  points at $s$ given by \eqref{Gl bifpoints Thm4} are predicted by Theorem~\ref{Thm 4 bifurcation f} since the conditions (S) and (T) are satisfied.
  The $f-$values found by AUTO are listed as well.  
\end{minipage} \quad
\begin{minipage}[c]{0.58\textwidth}
\tiny
\centering
\begin{tabular}{|r|r|r|c|r|r|c|}
\hline
$k$ & $\sigma$ & $s$ & curve & $f$ AUTO & $f$ Thm~\ref{Thm 4 bifurcation f} &  label \\
\hline
1 & 1 & 1.82156 & magenta & 12.30707 & 12.30707 & ~7\\
1 & -1 & 3.12227 & magenta & 3.21945 & 3.21945 & 12\\
2 & 1 & 1.76678 & brown & 12.28057 & 12.28053 & ~6\\
2 & -1 & 3.02410 & brown & 3.97844 & 3.97844 & 11\\
3 & 1 & 1.67183 & orange & 12.16098  & 12.16097 & ~5 \\
3 & -1 & 2.85277 & orange & 6.02862  & 6.02862 & 10\\
4 & 1 & 1.53017 & blue & 11.81841 & 11.81841 & ~3\\
4 & -1 & 2.59331 & blue & 8.87959  & 8.87959 & ~9\\
5 & 1 & 1.33036 & green & 11.029645 & 11.02958 & ~2\\ 
5 & -1 & 2.21287 & green & 11.50750  & 11.50749 & ~8\\
6 & 1 & 1.06458 & red & 9.499203 & 9.49913 & ~1 \\
6 & -1 & 1.61245 & red & 12.04054  & 12.04060 & ~4 \\
\hline
\end{tabular}
 \label{fig tabular 2nd case}
\caption{Bifurcation points on trivial branch.}
\end{minipage}
\end{figure}
  In the right diagram of Figure~\ref{bif_diagram2_zeta=10} one sees brown and magenta branches which
  bifurcate from the trivial curve at label 11 and 12 and connect to the blue and red branches at points of
  secondary bifurcation.  Solitons were found at turning points, cf. labels 13 and 14. Their shapes
 are shown in Figure~\ref{solitons_f}. 
 
 \vspace{-1em}
 
\begin{figure}[H]
\centering
\includegraphics[width=0.47\textwidth]{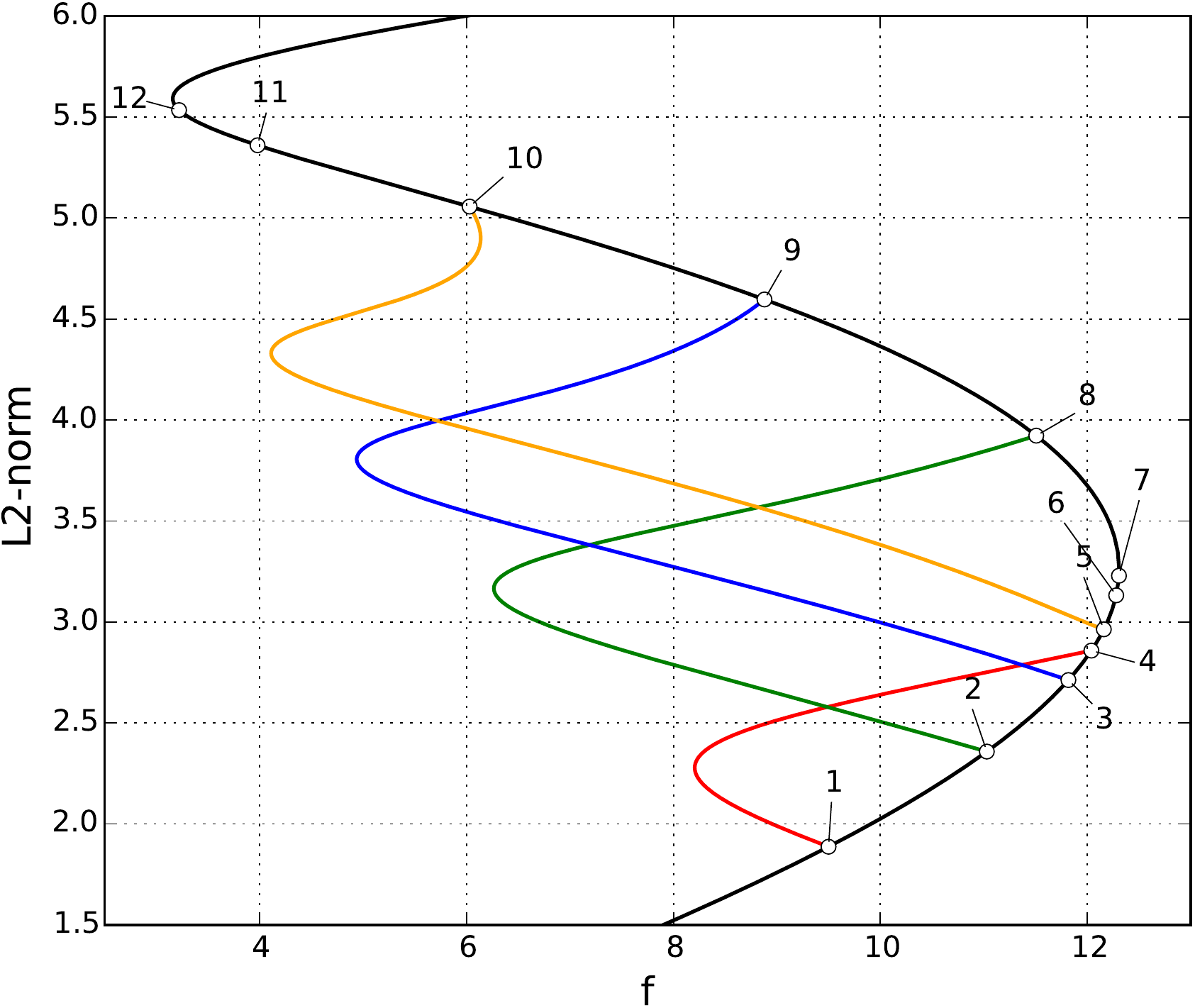}\qquad
 \includegraphics[width=0.47\textwidth]{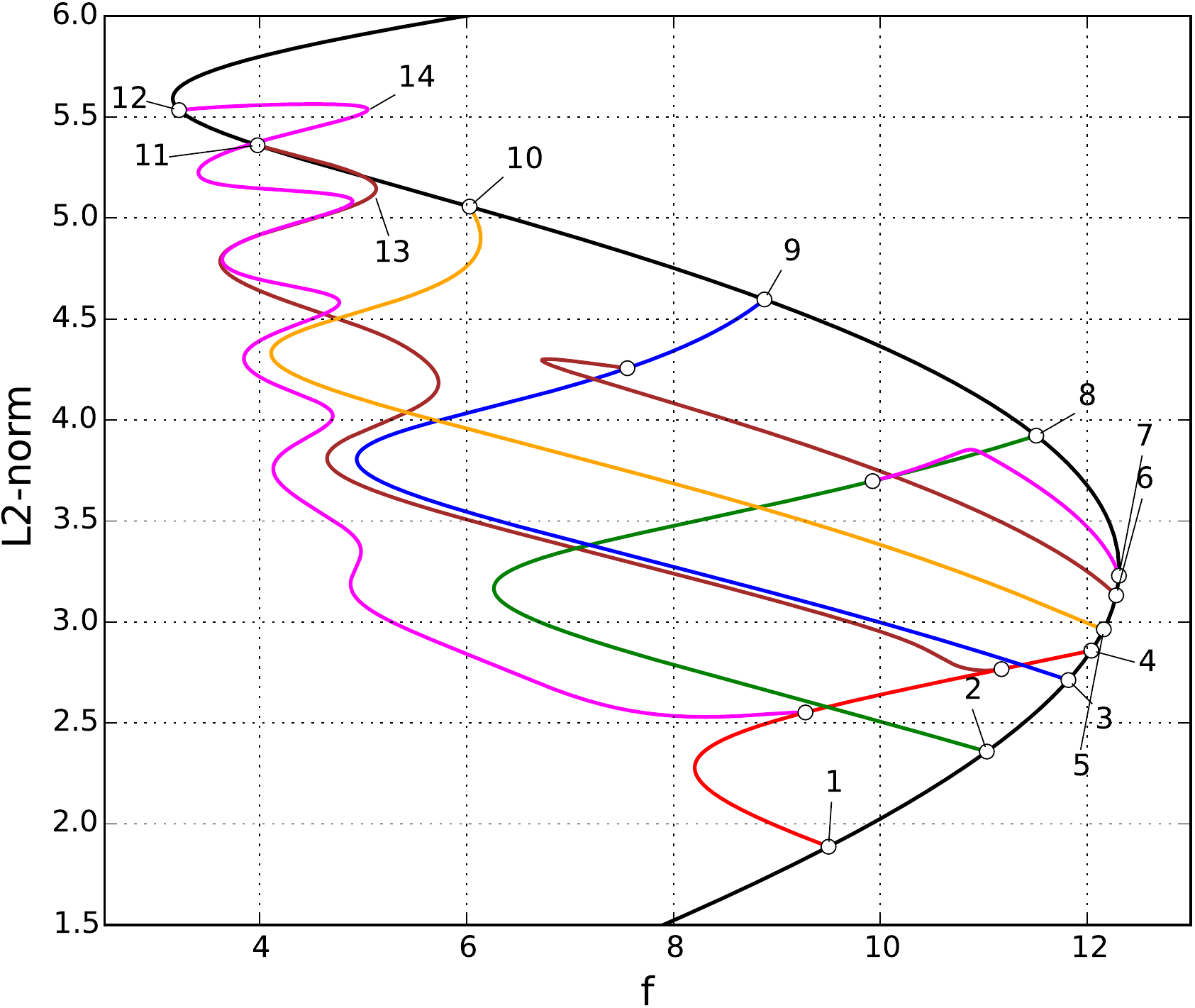}
\caption{Bifurcation diagrams
}
\label{bif_diagram2_zeta=10}
\end{figure}

\vspace{-1.5em}

\begin{figure}[H]
\begin{subfigure}{.49\textwidth}
\includegraphics[width=0.508\textwidth]{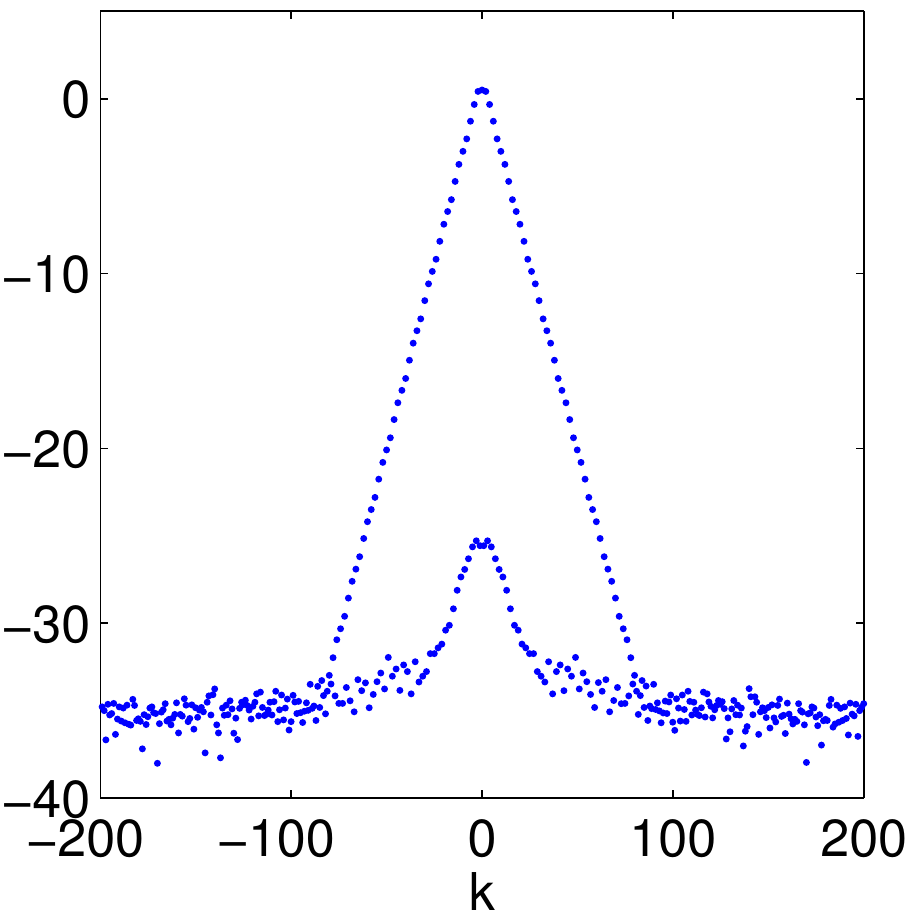}
\includegraphics[width=0.472\textwidth]{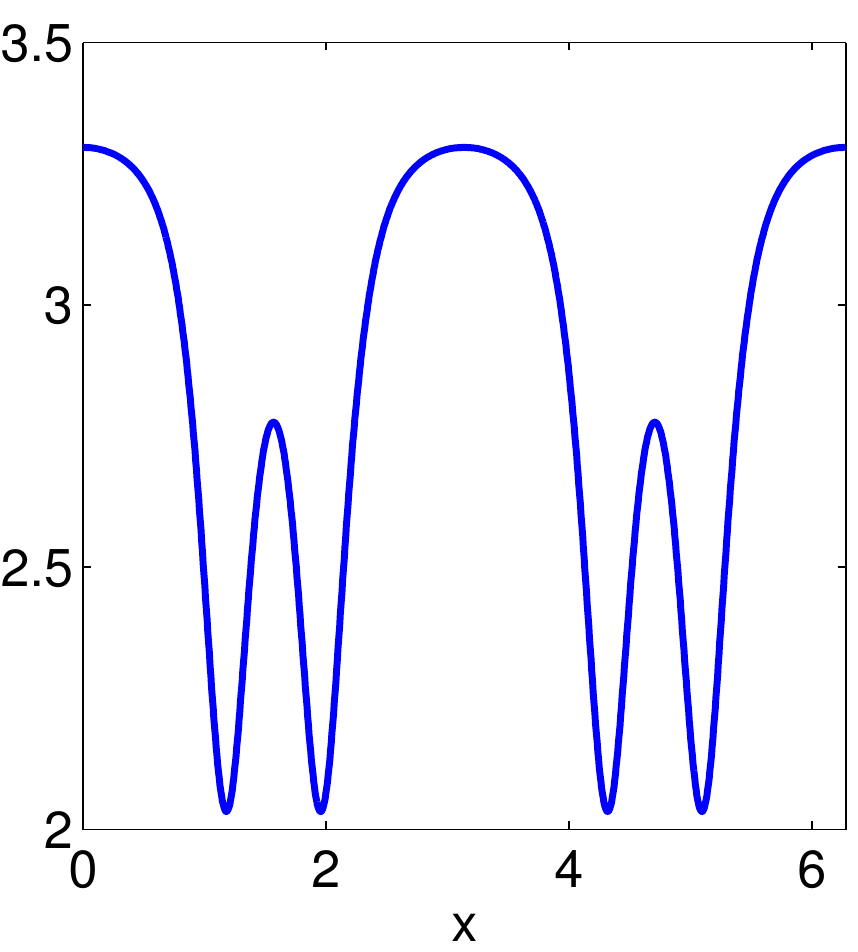}
\subcaption{Label 13: dark 4-soliton at $f=5.10776$}
\end{subfigure}
\begin{subfigure}{.49\textwidth}
\includegraphics[width=0.508\textwidth]{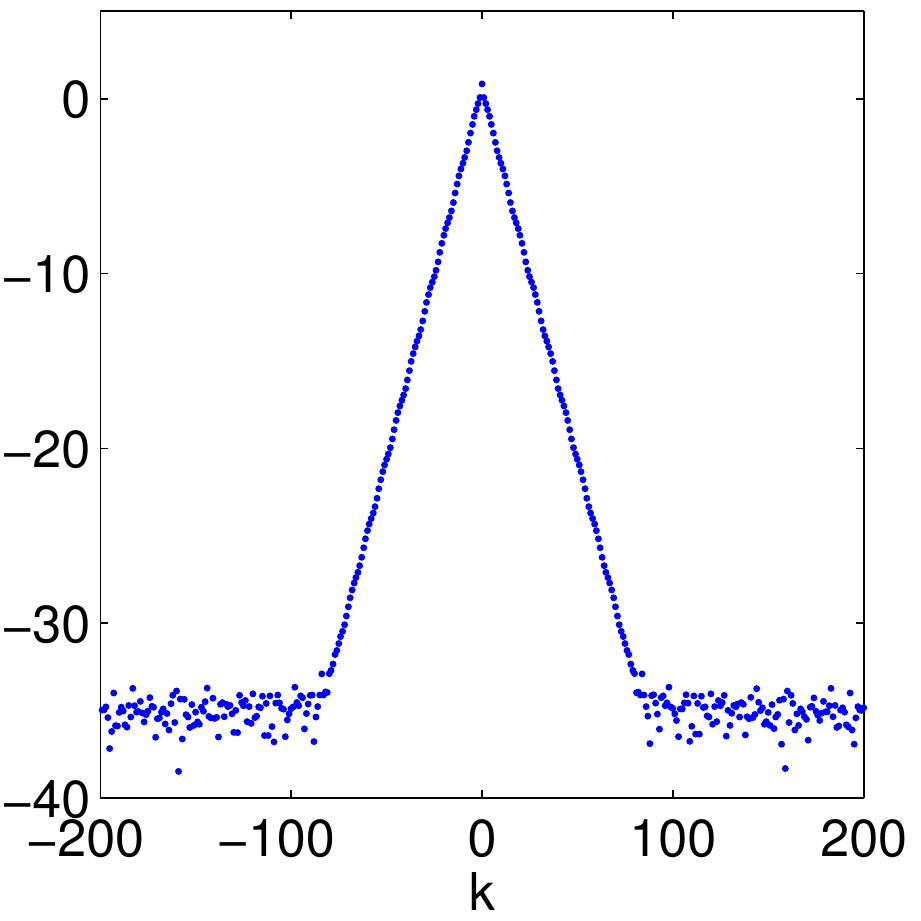}
\includegraphics[width=0.472\textwidth]{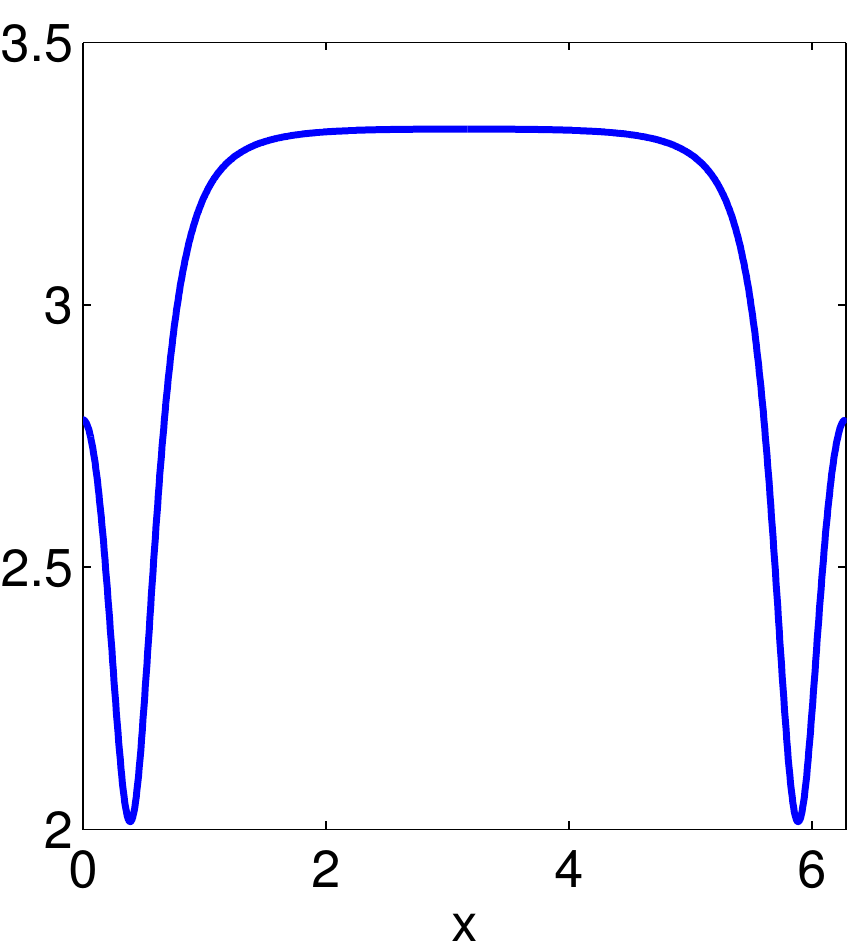}
\subcaption{Label 14: dark 2-soliton at $f=5.02592$}
\end{subfigure}
\caption{Selected solutions. Left: $\log(|\hat a(k)|)$; right: $|a(x)|$.}
\label{solitons_f}
\end{figure}

\vspace{-1.5em}




\subsection{The case $f=1.6,d=0.1$} \label{f=1.6,d=0.1}


This is an example with 14 bifurcations points on the trivial curve and many interesting features. Solitons are found at labels 15, 16, 17, 18 cf. Figures~\ref{many_solitons} and \ref{many_solitons2}.

\vspace{-1.5em}

\begin{figure}[H]\label{branches_3}
\centering
\begin{minipage}[c]{0.33\textwidth}
 \vspace{-\baselineskip}
 The table provides all bifurcation points on the trivial branch. The bifurcation
 points at $\zeta=\hat\zeta(t)$ are predicted by Theorem~\ref{Thm 3 bifurcation zeta} since
 $k,\sigma,t$ solve \eqref{Gl bifpoints Thm3} and the conditions (S) and (T) are satisfied.
 The $\zeta-$values found by AUTO are listed as well.
\end{minipage} \qquad
\begin{minipage}[c]{0.6\textwidth}
\tiny
\centering
\begin{tabular}{|c|c|c|c|c|c|c|}
\hline
$k$ & $\sigma$ & $t$ & curve & $\zeta$ AUTO & $\zeta$ Thm~\ref{Thm 3 bifurcation zeta} &  label \\
\hline
1 & ~1 & ~0.10528 & magenta & ~2.63750 & ~2.63750 & ~8 \\
1 & -1 & ~0.77130 & magenta & ~2.24888 & ~2.24888 & 14\\
2 & ~1 & -0.18543 & orange & ~2.28327 & ~2.28327 & ~6 \\
2 & -1 & ~0.75556 & orange &~ 2.25196 & ~2.25196 & 13 \\
3 & ~1 & -0.52046 & brown & ~1.25701  & ~1.25702 & ~4\\
3 & -1 & ~0.72127 & brown & ~2.26952 &  ~2.26952 &   12 \\
4 & ~1 & -0.72866 & green & ~0.13681 & ~0.13682 & ~2\\ 
4 & -1 & ~0.66089 & green & ~2.32248 & ~2.32248 & 11\\
5 & -1 & -0.77281 & red & -0.18665 & -0.18666 & ~1\\
5 & -1 & ~0.56321 & red & ~2.42958 & ~2.42954 & 10\\
6 & -1 & -0.61695 & blue & ~0.80168 & ~0.80166 & ~3\\
6 & -1 & ~0.40312 & blue & ~2.58451 & ~2.58449 & ~9 \\
7 & -1 & -0.20600 & pink & ~2.24085 & ~2.24085 & ~5\\
7 & -1 & ~0.01535 & pink & ~2.57474 & ~2.57475 & ~7\\
\hline
\end{tabular}
\caption{Bifurcation points on trivial branch.}
\label{list_branches_3}
\end{minipage}
\end{figure}

\vspace{-1.5em}

\begin{figure}[H]
\begin{minipage}[c]{0.55\textwidth}
\includegraphics[width=\textwidth]{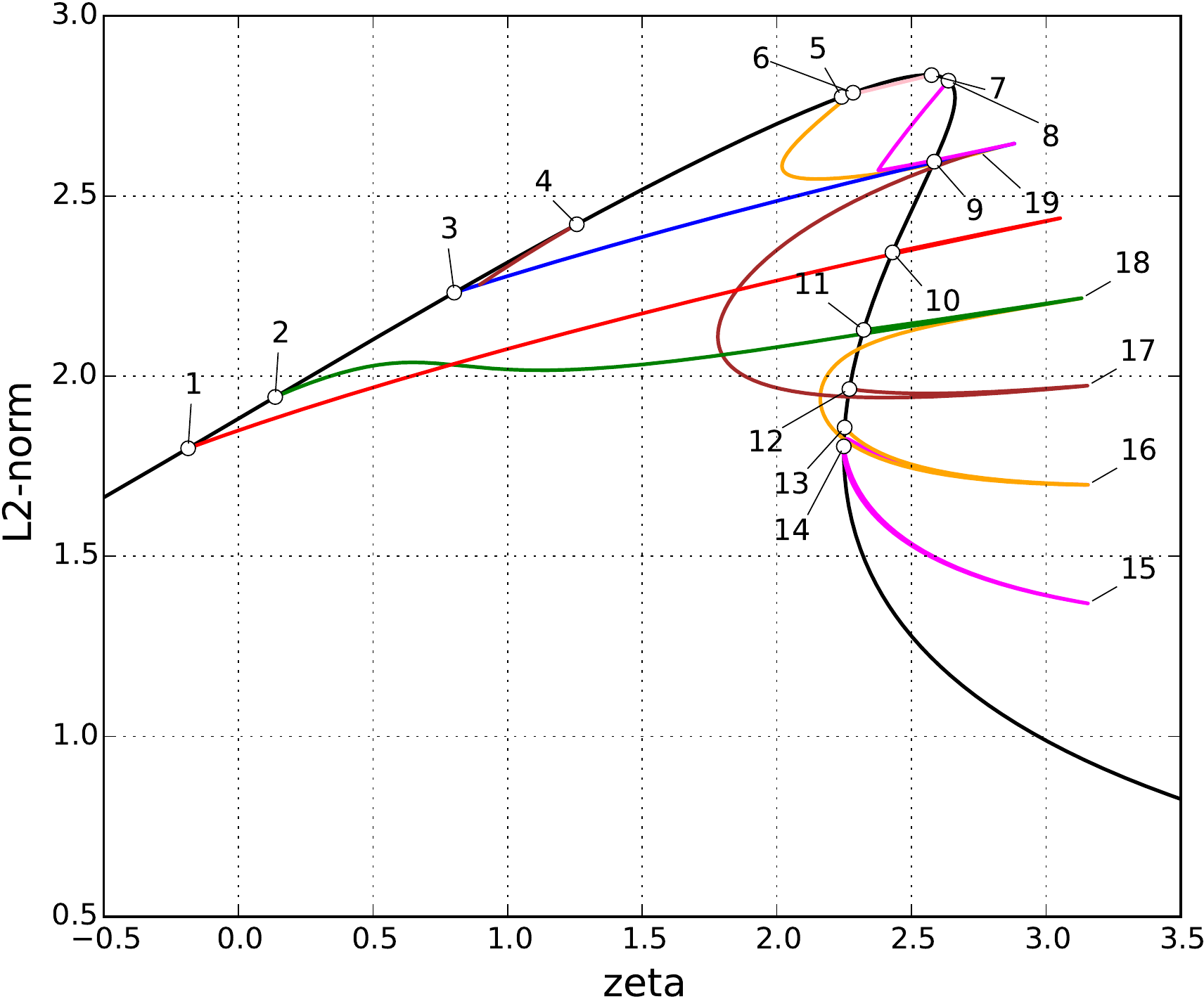}
\caption{Bifurcation diagram}
\label{bif_diagram_f=1.6}
\end{minipage} \qquad
\begin{minipage}[c]{0.3\textwidth}
  \vspace{-2\baselineskip} 
  The brown branches, the magenta branch starting at label 8 and the orange branch starting at label 6 all
  enter the blue one. The orange branch going off label 12 enters the green one. A period-doubling bifurcation occurs when the brown branch going off label 4 meets the blue branch. 
\end{minipage}
\end{figure}

\vspace{-1.5em}

\begin{figure}[H]
\begin{minipage}[c]{0.45\textwidth}
\hspace{4em}\includegraphics[width=.5\textwidth]{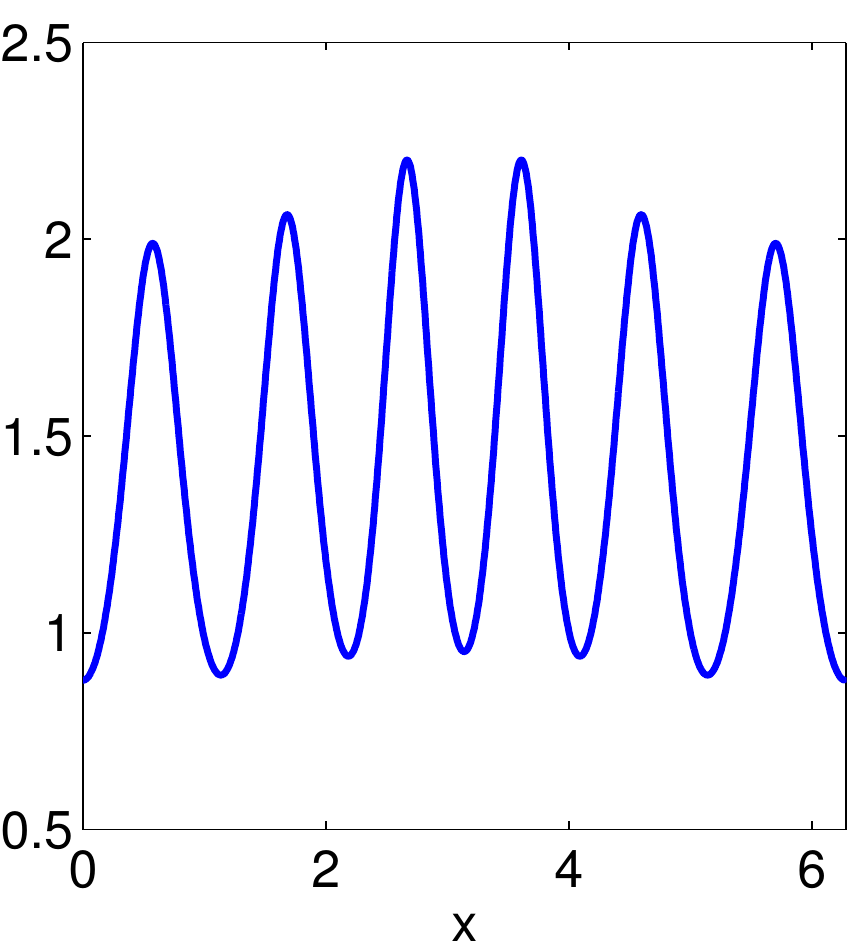}
\caption{Label 19 at $\zeta=2.75156$}
\end{minipage} 
\begin{minipage}[c]{0.5\textwidth}
  \vspace{-2\baselineskip} 
  At label 19 on the magenta branch we find a $2\pi$-periodic solution with 6 maxima which is not
  $2\pi/6$-periodic. In contrast, all solutions on the blue branch are $2\pi/6$-periodic. The explanation is,
  that the $2\pi/6$-periodicity is lost along the magenta branch which joins the blue one in a secondary
  bifurcation.
\end{minipage}
\end{figure}

At the labels~15,~16,~17,~18 bright 1-, 2-, 3-, 4-solitons are found, respectively. 
They lie on branches emanating from bifurcation points associated to $k=1, 2, 3, 4$, respectively.   
\begin{figure}[H]
\begin{subfigure}[b]{.49\textwidth}
\includegraphics[width=.505\textwidth]{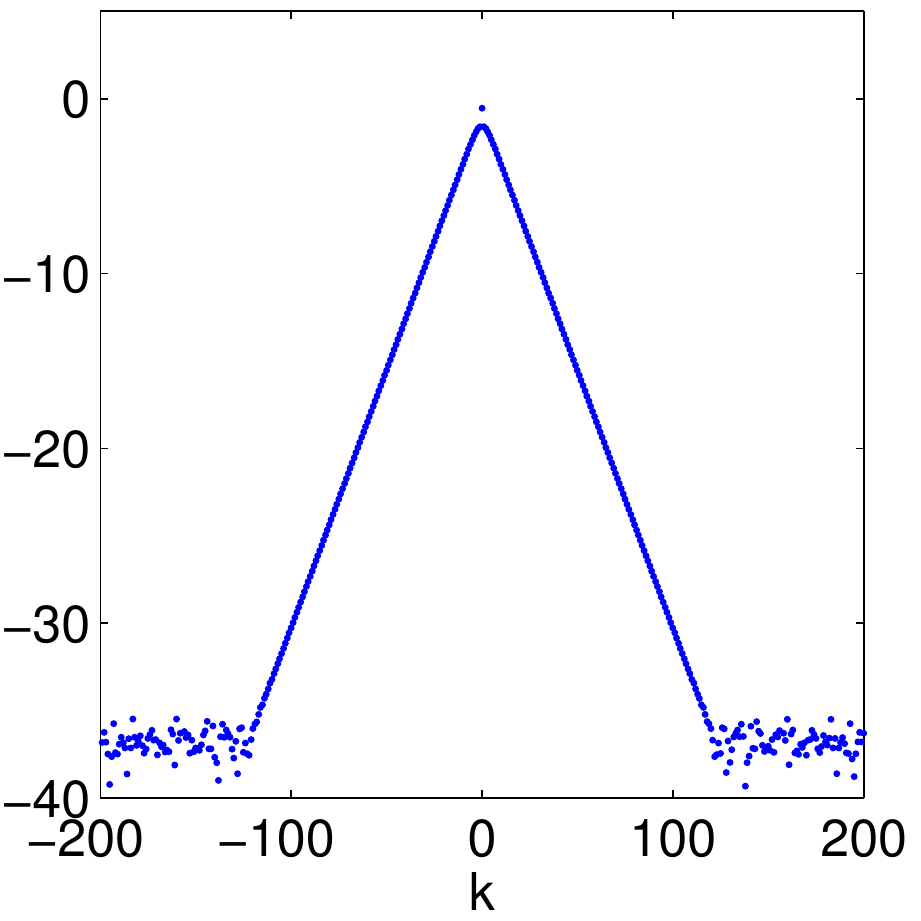}
\includegraphics[width=.475\textwidth]{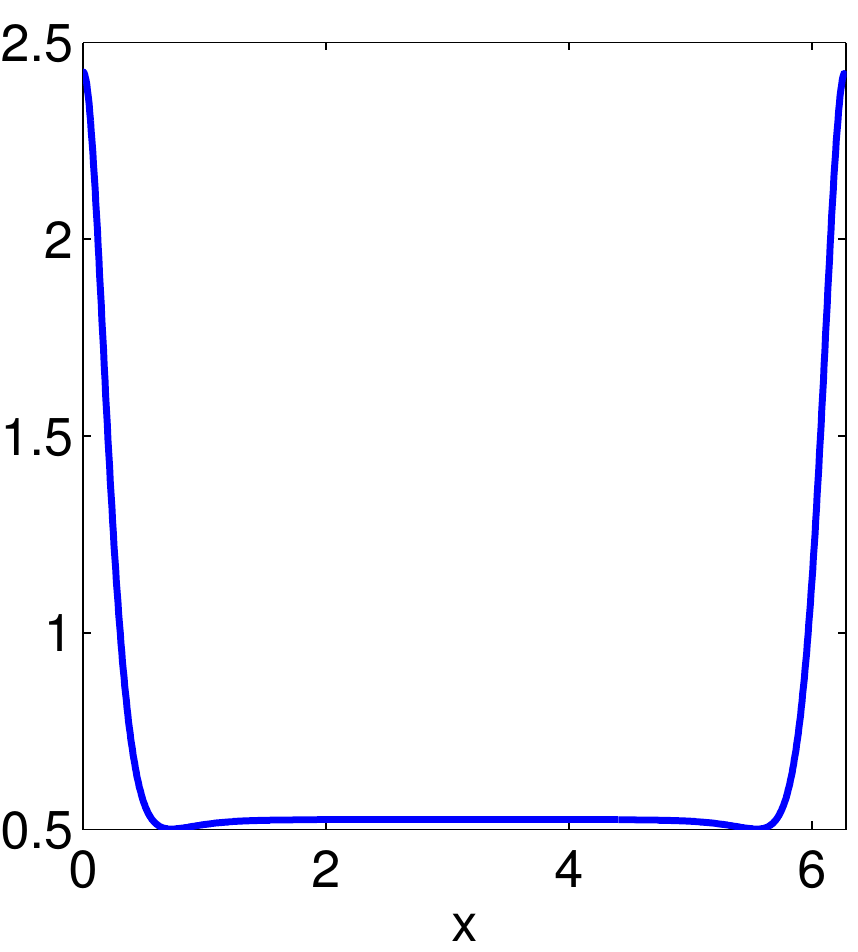}
\subcaption{Label 15: $\zeta=3.15568$}
\end{subfigure}
\begin{subfigure}[b]{.49\textwidth}
\includegraphics[width=.505\textwidth]{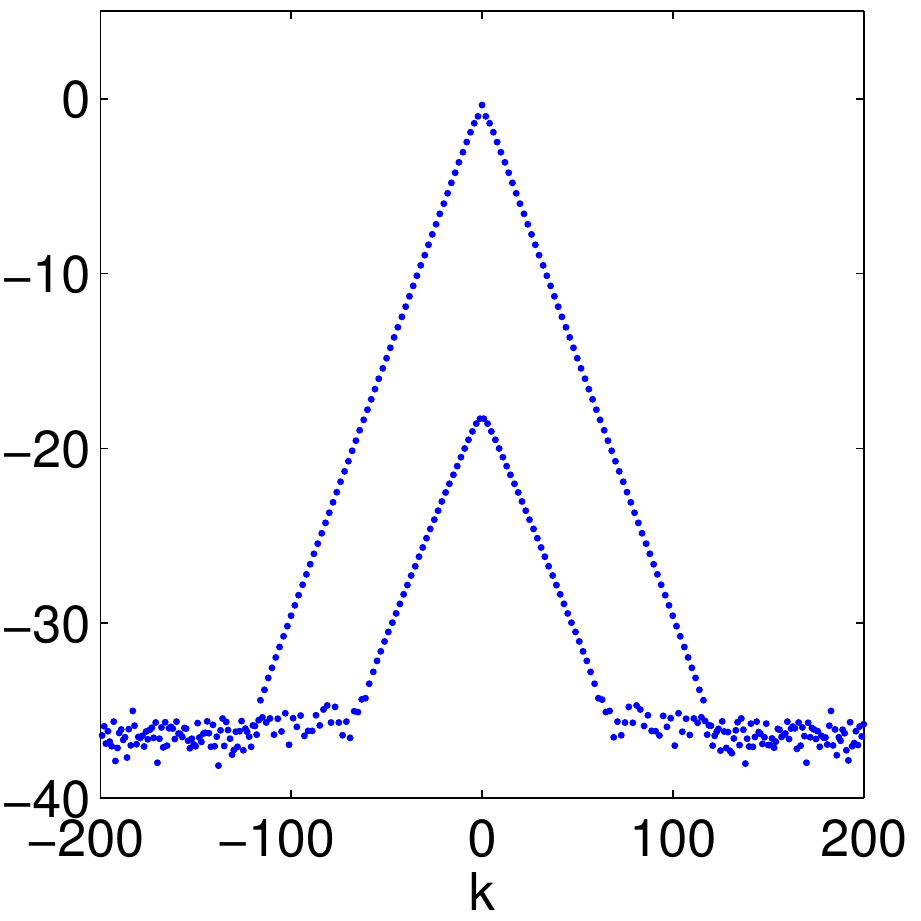}
\includegraphics[width=.475\textwidth]{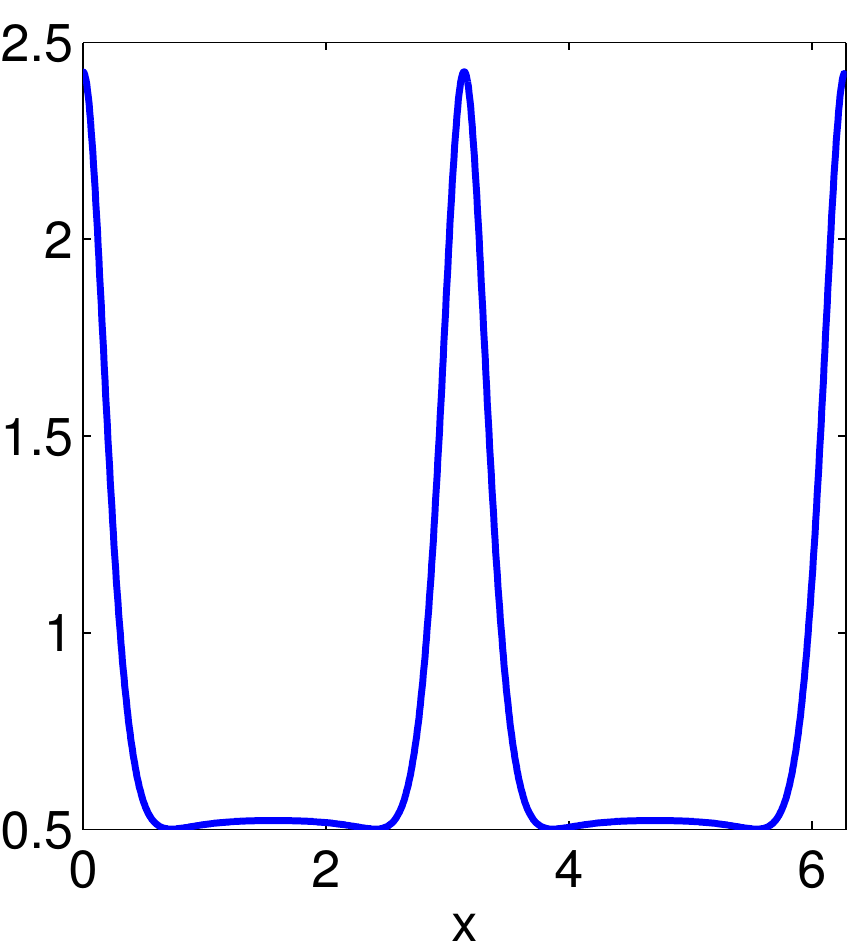}
\subcaption{Label 16: $\zeta=3.15568$}
\end{subfigure}
\caption{Selected solutions. Left: $\log(|\hat a(k)|)$; right: solid $|a(x)|$.}
\label{many_solitons}
\end{figure}

\begin{figure}[H]
\begin{subfigure}[b]{0.49\textwidth}
\includegraphics[width=.505\textwidth]{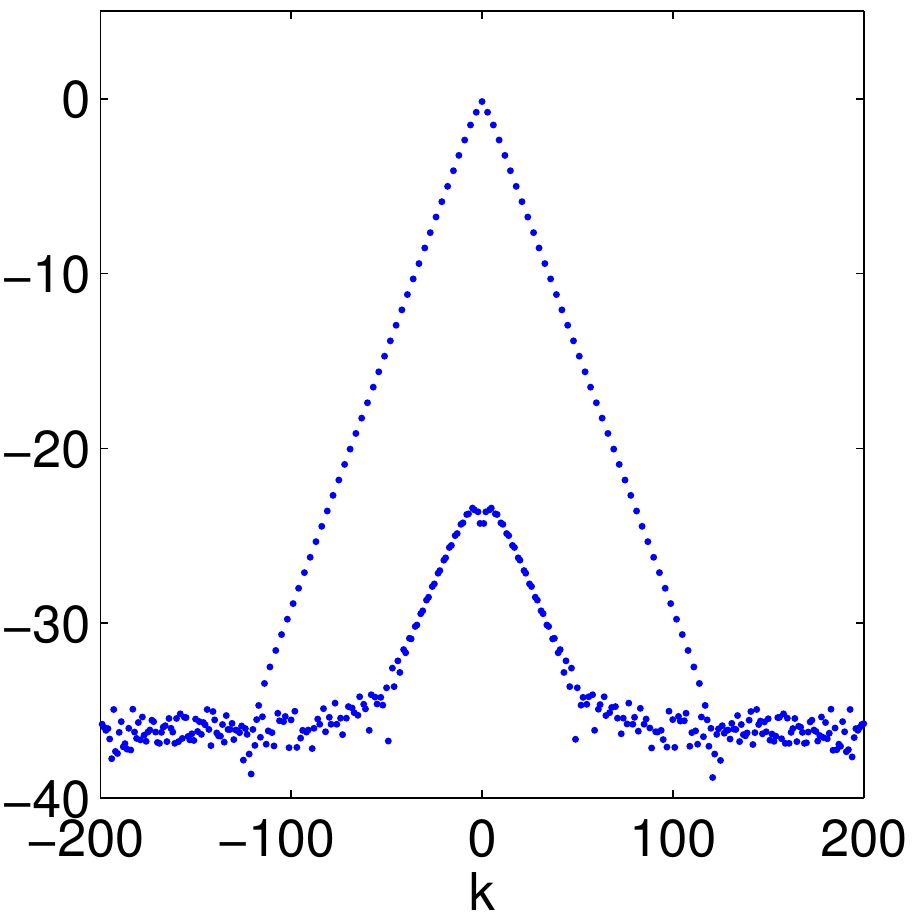}
\includegraphics[width=.475\textwidth]{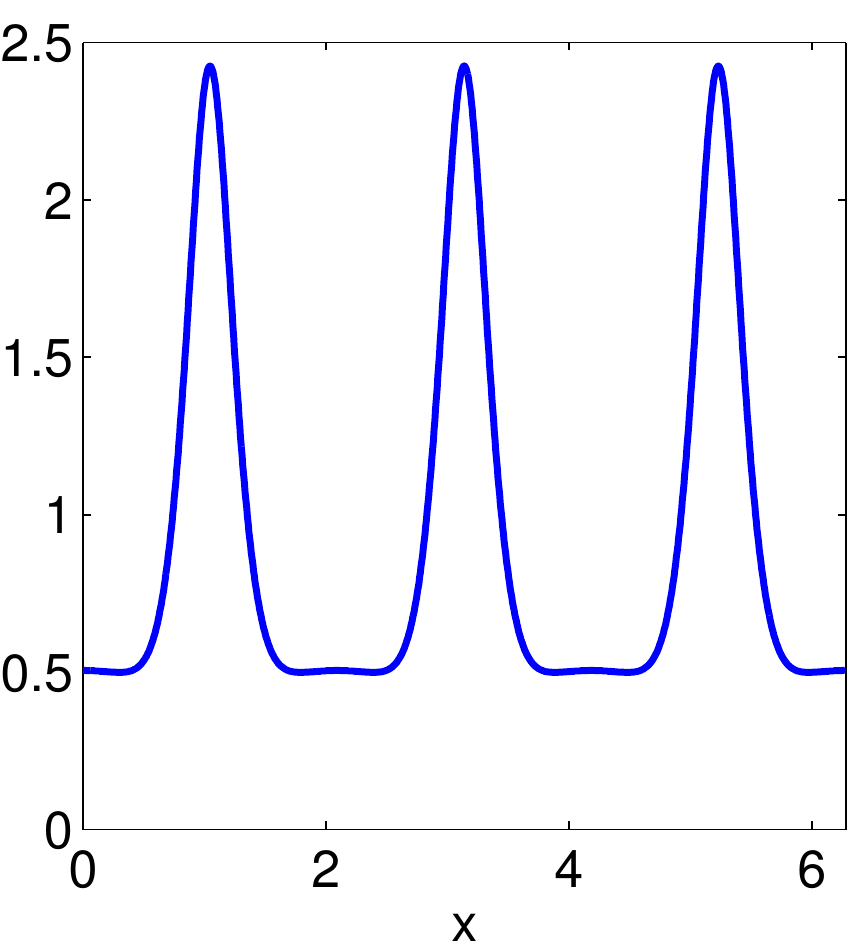}
\subcaption{Label 17: $\zeta=3.15392$}
\end{subfigure}
\begin{subfigure}[b]{0.49\textwidth}
\includegraphics[width=.505\textwidth]{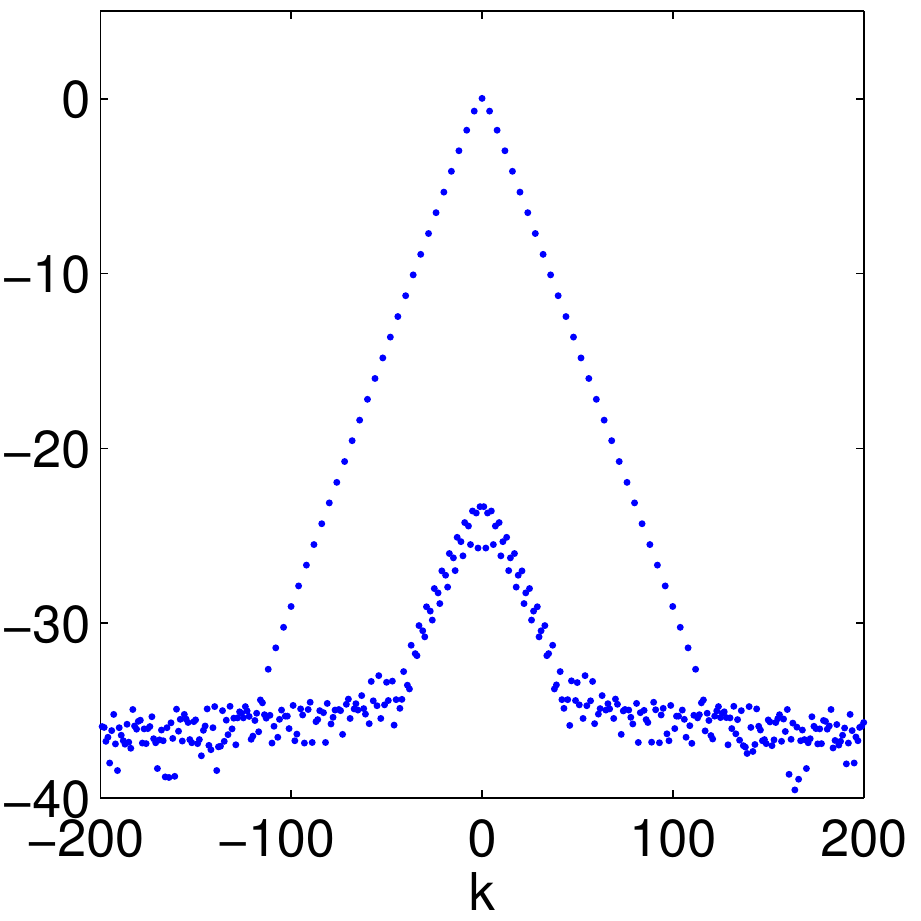}
\includegraphics[width=.475\textwidth]{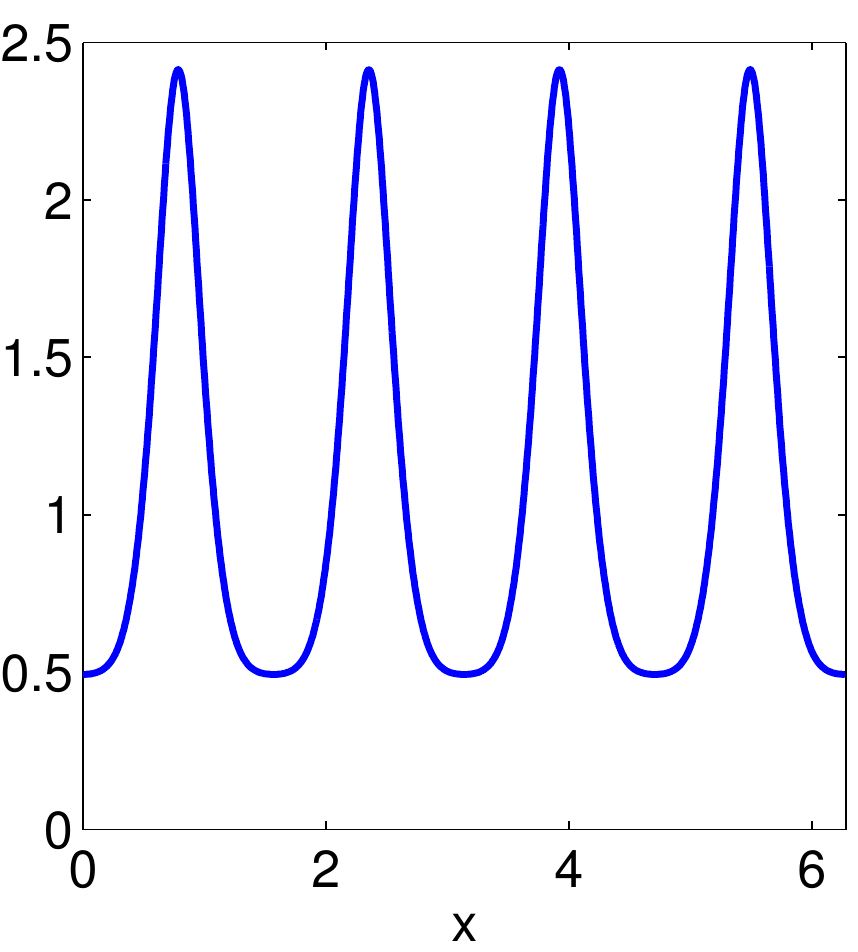}
\subcaption{Label 18: $\zeta=3.13217$}
\end{subfigure}
\caption{Selected solutions. Left: $\log(|\hat a(k)|)$; right: solid $|a(x)|$.}
\label{many_solitons2}
\end{figure}





\subsection{The case $f=2,d=-0.1$} \label{f=2,d=-0.1}


 This is again an example with finitely many bifurcation points and dark solitions like in Setion~\ref{zeta=10,d=-0.2}. Also here the dark solitons occur in most pronounced form at turning points of nontrivial branches.

\begin{figure}[H]\label{branches_4}
\begin{minipage}[c]{0.41\textwidth}
 The table provides all bifurcation points on the trivial branch. The bifurcation
 points at $\zeta=\hat\zeta(t)$ are predicted by Theorem~\ref{Thm 3 bifurcation zeta} since
 $k,\sigma,t$ solve \eqref{Gl bifpoints Thm3} and the conditions (S) and (T) are satisfied.
 The $\zeta-$values found by AUTO are listed as well.  
\end{minipage} \quad
\begin{minipage}[c]{0.55\textwidth}
\tiny
\centering 
\begin{tabular}{|r|r|c|c|r|r|c|}
\hline
$k$ & $\sigma$ & $t$ & curve & $\zeta$ AUTO & $\zeta$ Thm~\ref{Thm 3 bifurcation zeta} &  label \\
\hline
1 & -1 & 0.85260 & red & 2.72386 & 2.72386 & 4 \\
1 & 1 & 0.22806 & red & 4.02617 & 4.02619 & 1\\
2 & -1 & 0.86118 & blue & 2.72771 & 2.72771 & 5\\
2 & 1 & 0.49553 & blue & 3.58829 & 3.58830 & 2\\
3 & 1 & 0.86262 & green & 2.72883 &  2.72883 & 6\\
3 & 1 & 0.78647 & green  & 2.79924 &  2.79924 & 3 \\
\hline
\end{tabular}
\caption{Bifurcation points on trivial branch.}
\end{minipage} 
\end{figure}

\begin{figure}[H]
\begin{subfigure}{0.49\textwidth}
\includegraphics[width=\textwidth]{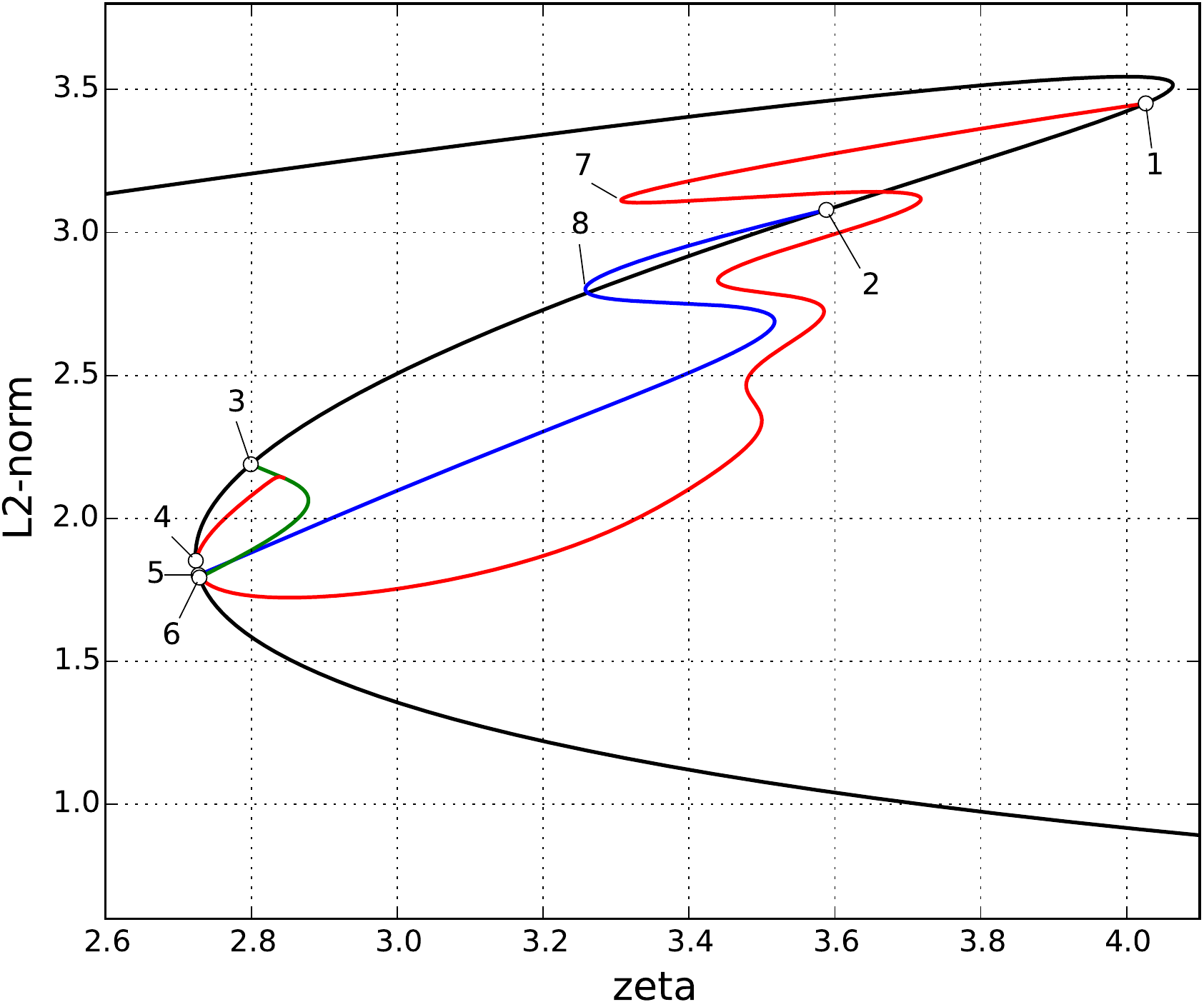}
\caption{primary bifurcations}
\end{subfigure}
\begin{subfigure}{0.49\textwidth}
\includegraphics[width=\textwidth]{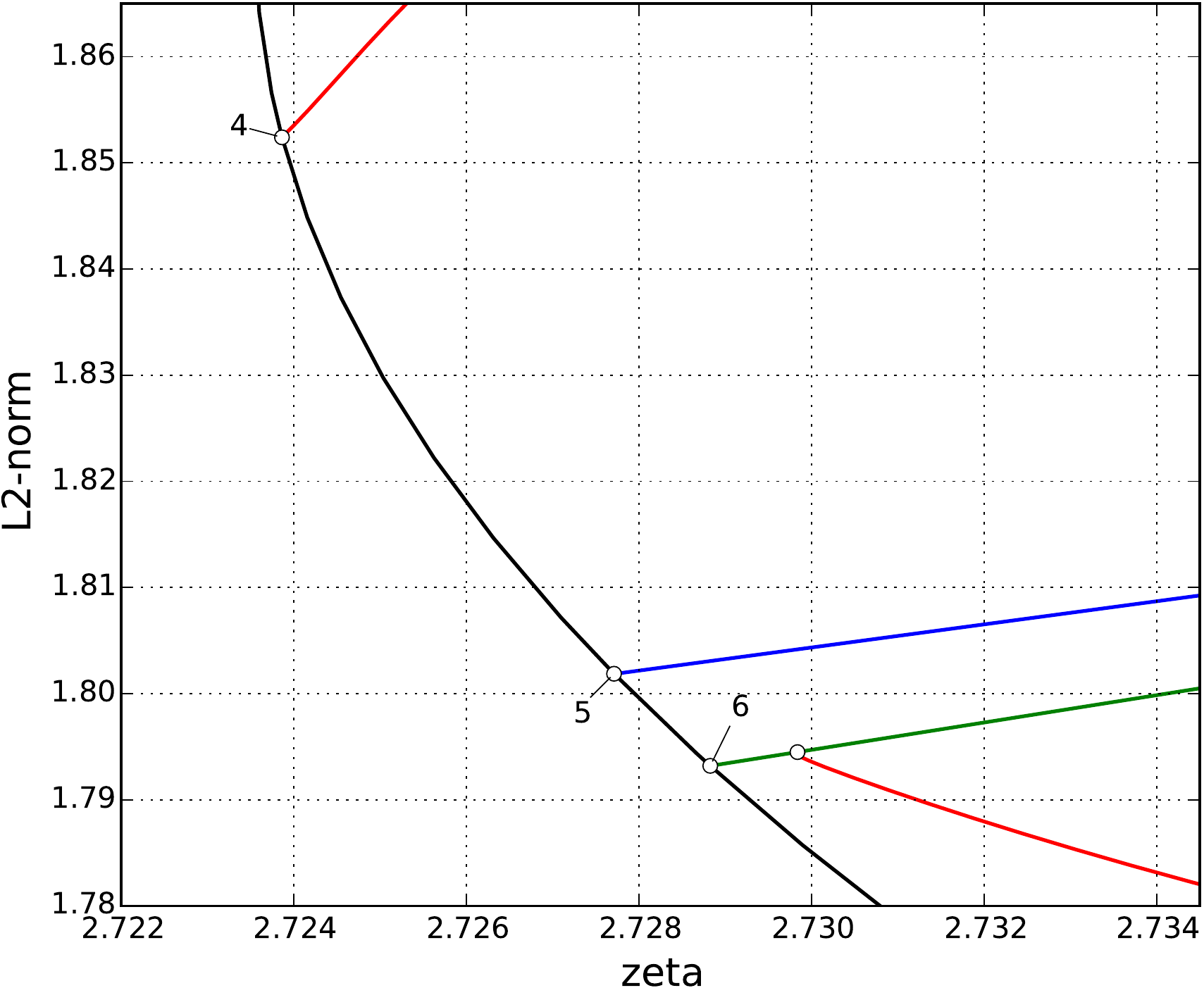} 
\caption{zoom near bifurcation points 5 and 6}
\end{subfigure}
 \caption{Bifurcation diagram}
 \label{bif_diagram_f=2}
\end{figure}


\begin{minipage}[c]{0.9\textwidth}
  Both red branches $(k=1)$ enter the green one $(k=3$) so that the trivial solutions
  at labels 1, 3, 4, 6 are connected to the dark 1-soliton at label 7 through nontrivial solutions.
  Similarly, the blue branch ($k=2$) connects the trivial solutions at labels 2, 5 with the dark 2-soliton at
  label 8.
\end{minipage}

\begin{figure}[H]
\begin{subfigure}{0.48\textwidth}
\includegraphics[width=0.508\textwidth]{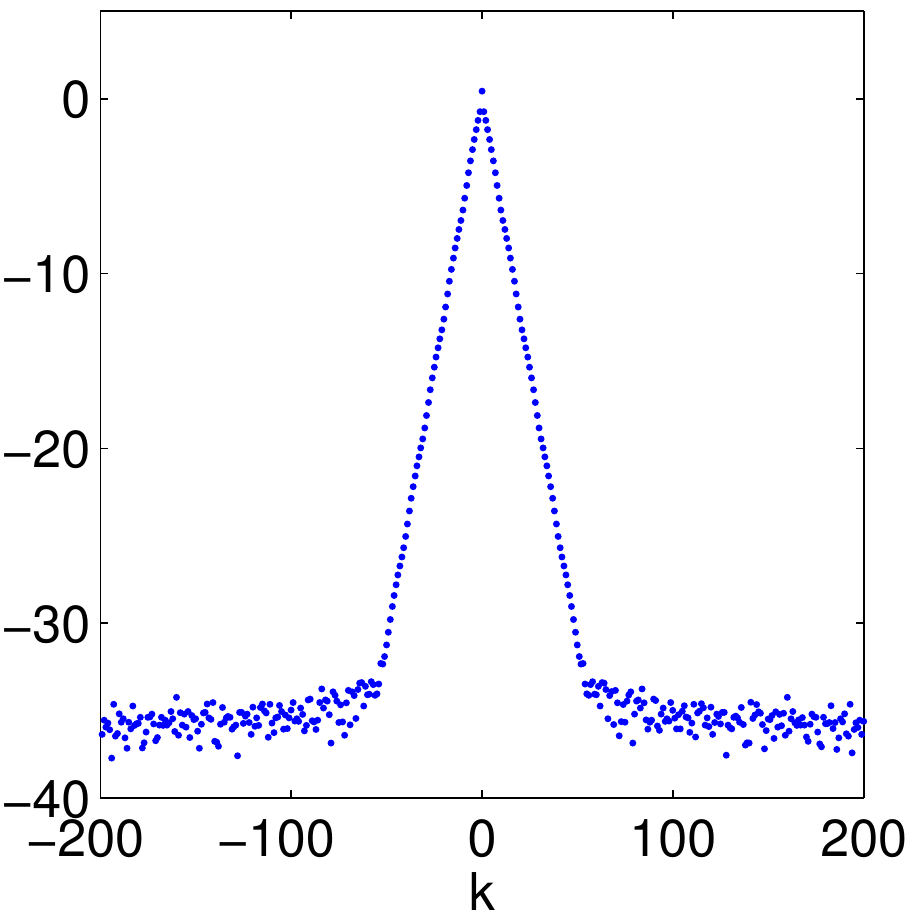}
\includegraphics[width=0.472\textwidth]{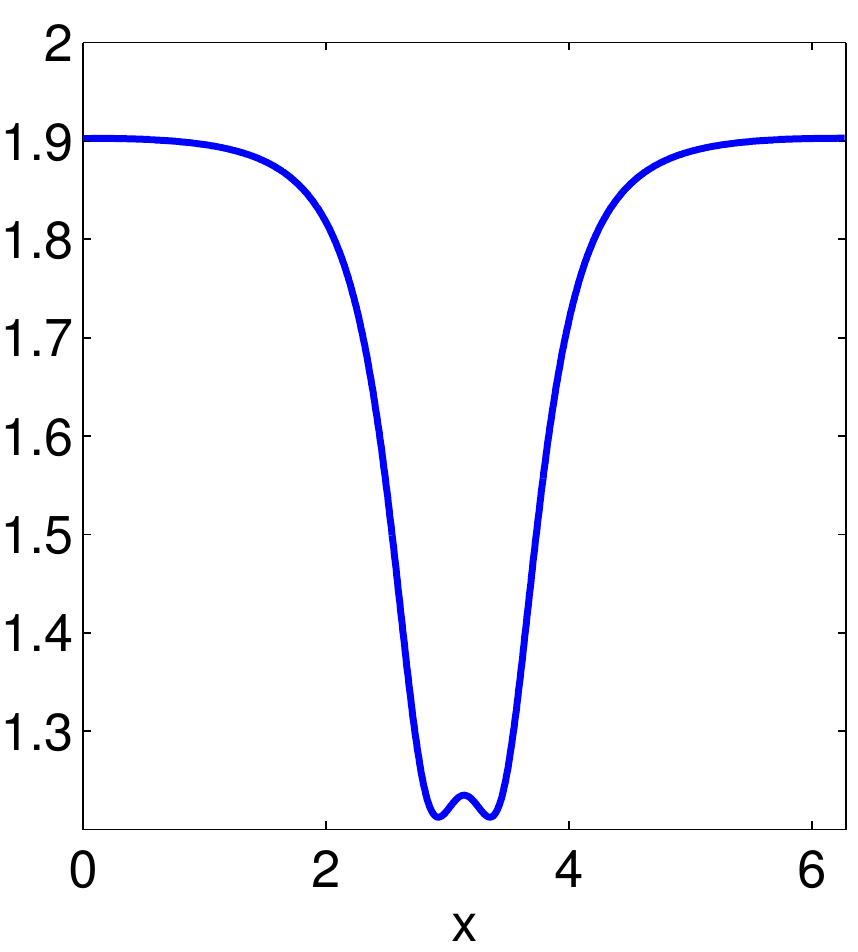}
\caption{Label 7: dark 1-soliton at $\zeta=3.30685$}
\end{subfigure}
\begin{subfigure}{0.48\textwidth}
\includegraphics[width=0.508\textwidth]{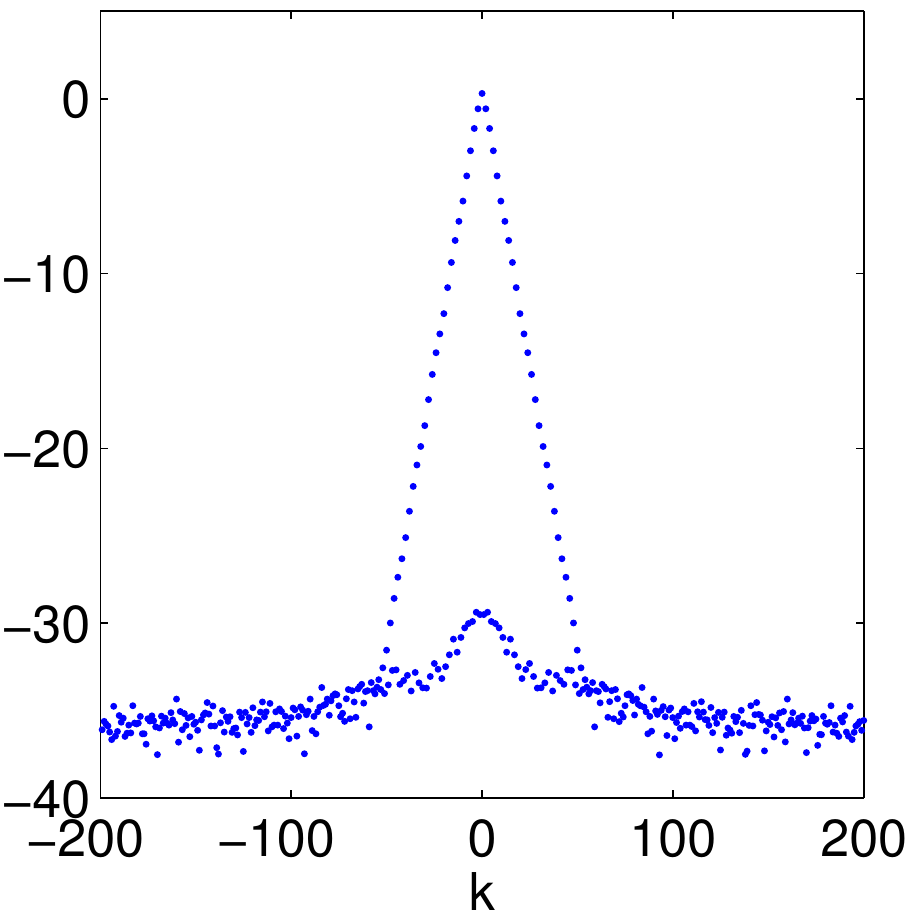}
\includegraphics[width=0.472\textwidth]{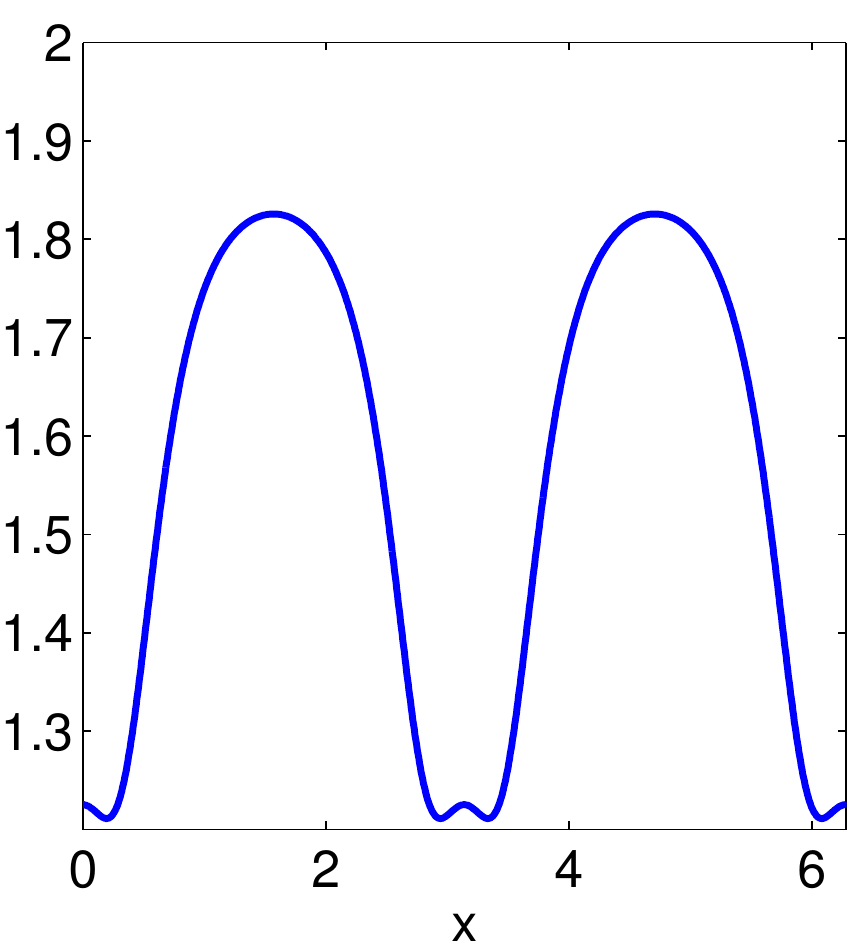}
\caption{Label 8: dark 2-soliton at $\zeta=3.25783$}
\end{subfigure}
\caption{Selected solutions. Left: $\log(|\hat a(k)|)$; right: $|a(x)|$.}
\label{solitons_zeta}
\label{solutions_branches_4}
\end{figure}

\subsection{Time-dependent detuning and bifurcation diagrams} \label{dynamic}


A typical problem in applications is the following: how can one drive the coupled laser/ring resonator into the 1-soliton state? 
A commonly used and quite practical idea is the use of time-dependent detuning
\cite{Karpovetal_Universal_dynamics}. Here the detuning $\zeta= \zeta(t)$ varies in time until it reaches its
final value at which it stays, cf. Figure~\ref{dynamic_detuning_soliton_pic1}(a).

\medskip

The mathematical model is given by the time-dependent Lugiato-Lefever equation
\begin{equation}
     \ri\partial_t a(x,t) = (- \ri+\zeta(t))a(x,t)-d \partial_x^2 a(x,t) -|a(x,t)|^2 a(x,t)+ \ri f,
    \; x\in(0,\pi), \; t \in \R
\label{t_ll}
\end{equation}
with homogenoeus Neumann boundary conditions at $x=0$ and $x=\pi$. In our example we considered the low-power
scenario of Section~\ref{f=1.6,d=0.1} with $f=1.6$, $d=0.1$ where 1-, 2- and 3-solitons exist for
$\zeta=2.67$, cf. Figure~\ref{bif_diagram_f=1.6}. A convenient choice for $\zeta$ is then given by the piecewise linear function 
$$ 
\zeta(t) = \left\{
\begin{array}{ll}
-5 & 0 \leq t \leq T/30, \vspace{\jot}\\
\frac{2.67+5}{T/3-T/30}(t-T/30) -5 & T/30 \leq t \leq T/3, \vspace{\jot}\\
2.67 & T/3\leq t \leq T
\end{array}
\right.
$$
with $T=1000$. With these choices we numerically integrated equation \eqref{t_ll} starting from initial data
given by the spatially constant steady-state solution at $\zeta=-5$ perturbed by a random 
function of size $10^{-14}$. The numerical scheme is a Strang-splitting\footnote{ In the Strang-splitting the
linear inhomogeneous part (including the space derivatives and the forcing/damping terms) is propagated by an
exponential integrator whereas the nonlinear part is solved exactly as for the standard NLS. It is proved in
\cite{jami:14} that under certain regularity assumptions on the initial data (see Section~\ref{further}) and
the forcing/detuning the numerical method converges to the true solution of \eqref{t_ll} on bounded
time-intervals with order $2$ in $L^2(0,\pi)$ and with order $1$ in $H^1(0,\pi)$.} in time as suggested in
\cite{jami:14}, and a pseudo-spectral method in space.

\smallskip

Depending on the (random) initial data we observe different scenarios for the evolution of the $L^2$-norm. We
show three of these scenarios in the following figures. The underlying picture is the bifurcation diagram
(with respect to $\zeta$) of the stationary equation from Figure~\ref{bif_diagram_f=1.6}. On top of this bifucation diagram we plot in grey the time evolution of the $L^2$-norm of the solution of \eqref{t_ll}.

\begin{figure}[H] 
\begin{subfigure}{0.47\textwidth}
\includegraphics[width=\textwidth]{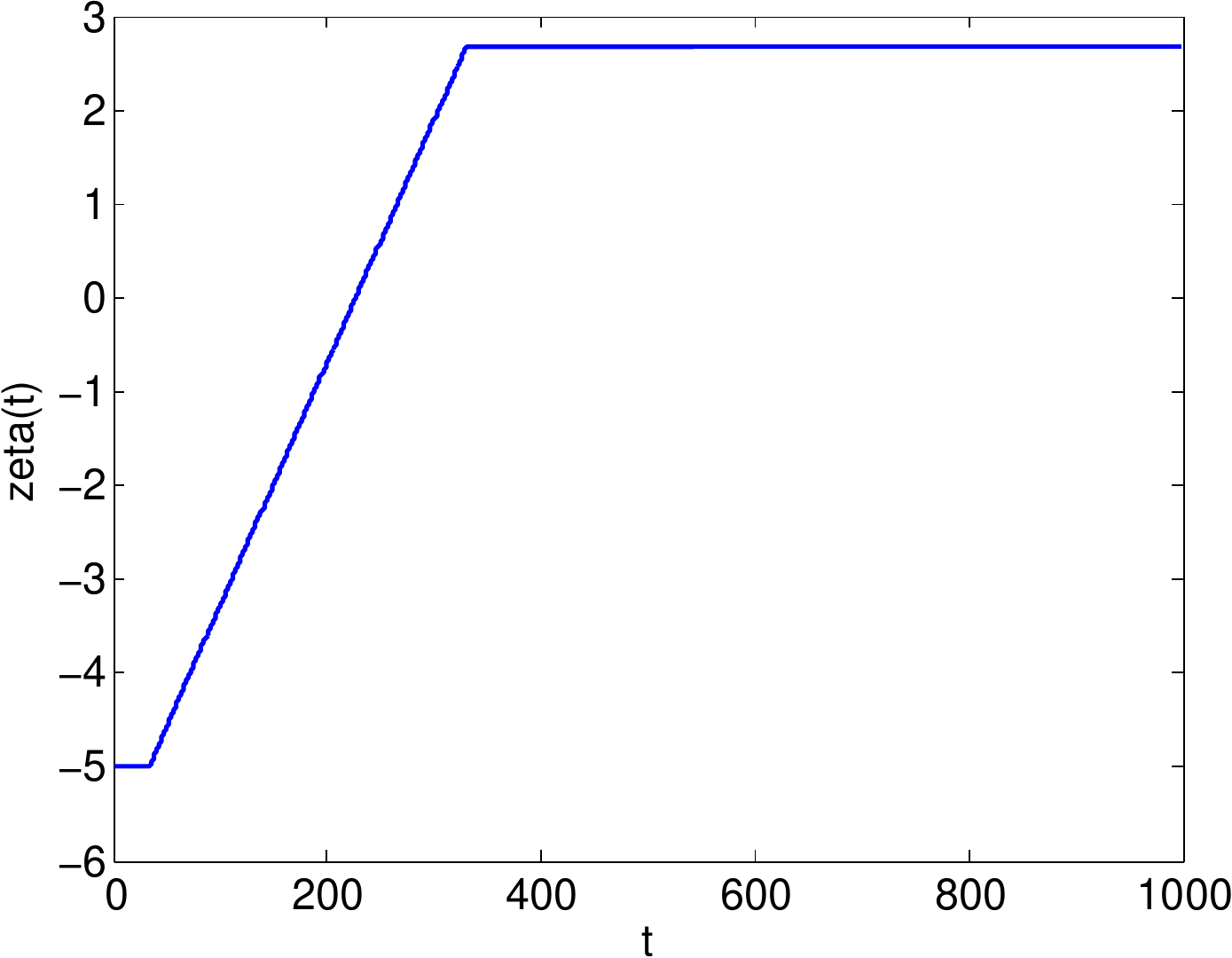}
\caption{(a) time dependent detuning}
\end{subfigure} \quad 
\begin{subfigure}{0.47\textwidth}
\includegraphics[width=\textwidth]{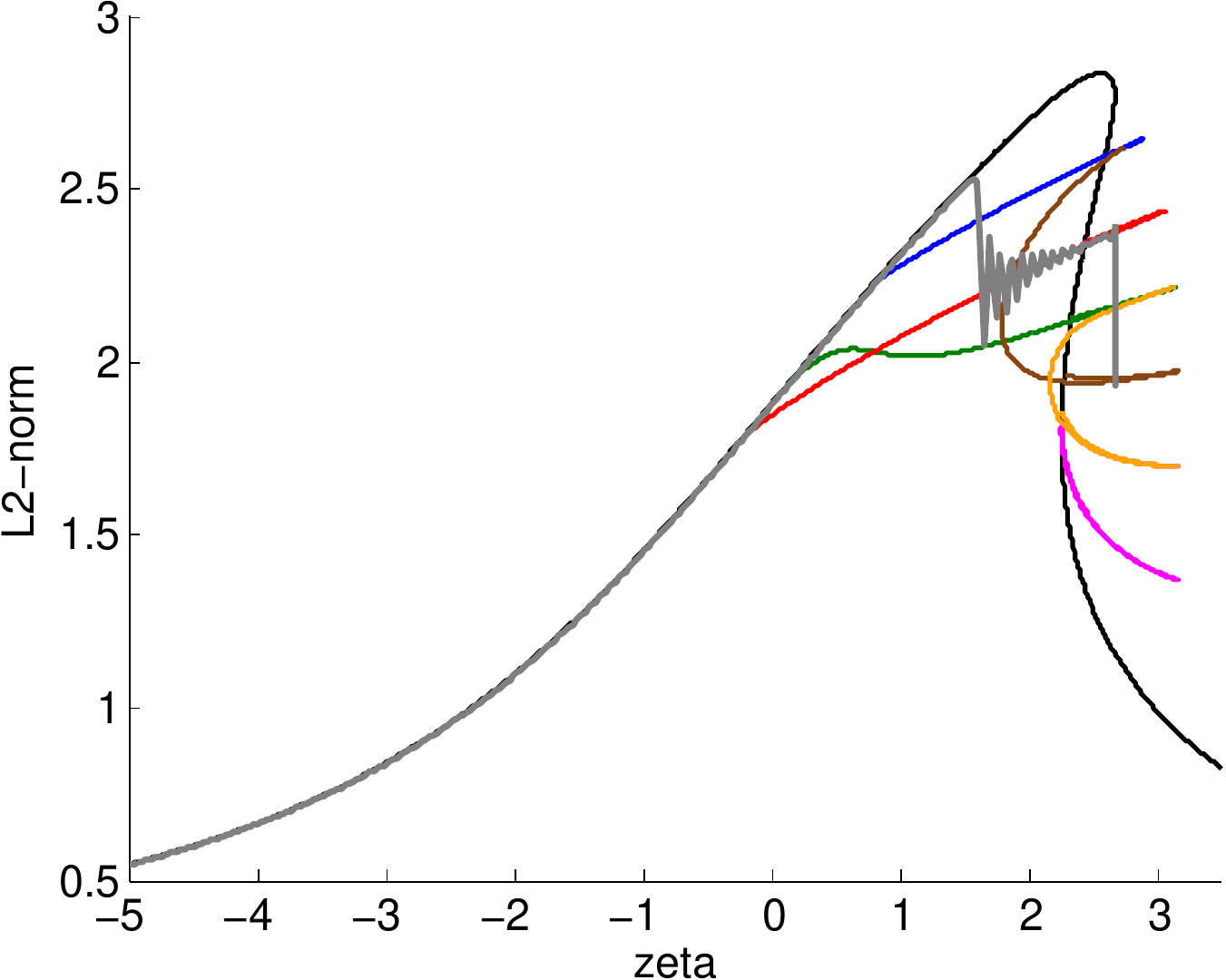}  
\caption{(b) final state: 3 soliton}
\end{subfigure}
\caption{Time evolution of $L^2$-norm in dynamic detuning}
\label{dynamic_detuning_soliton_pic1}
\end{figure}

\begin{figure}[H]
\begin{subfigure}{0.47\textwidth}
\includegraphics[width=\textwidth]{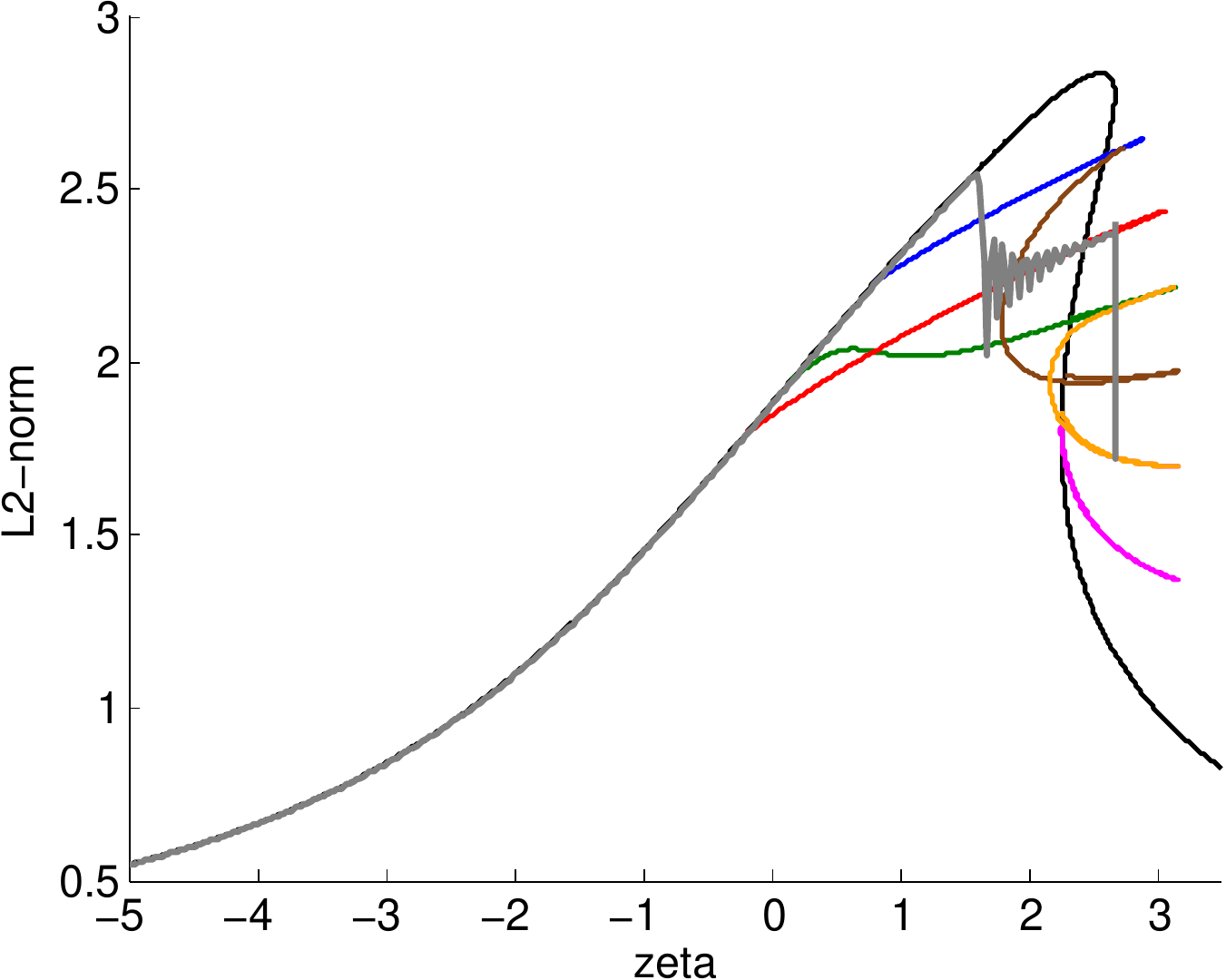} 
\caption{(c) final state: 2 soliton}
\end{subfigure}
\begin{subfigure}{0.47\textwidth}
\includegraphics[width=\textwidth]{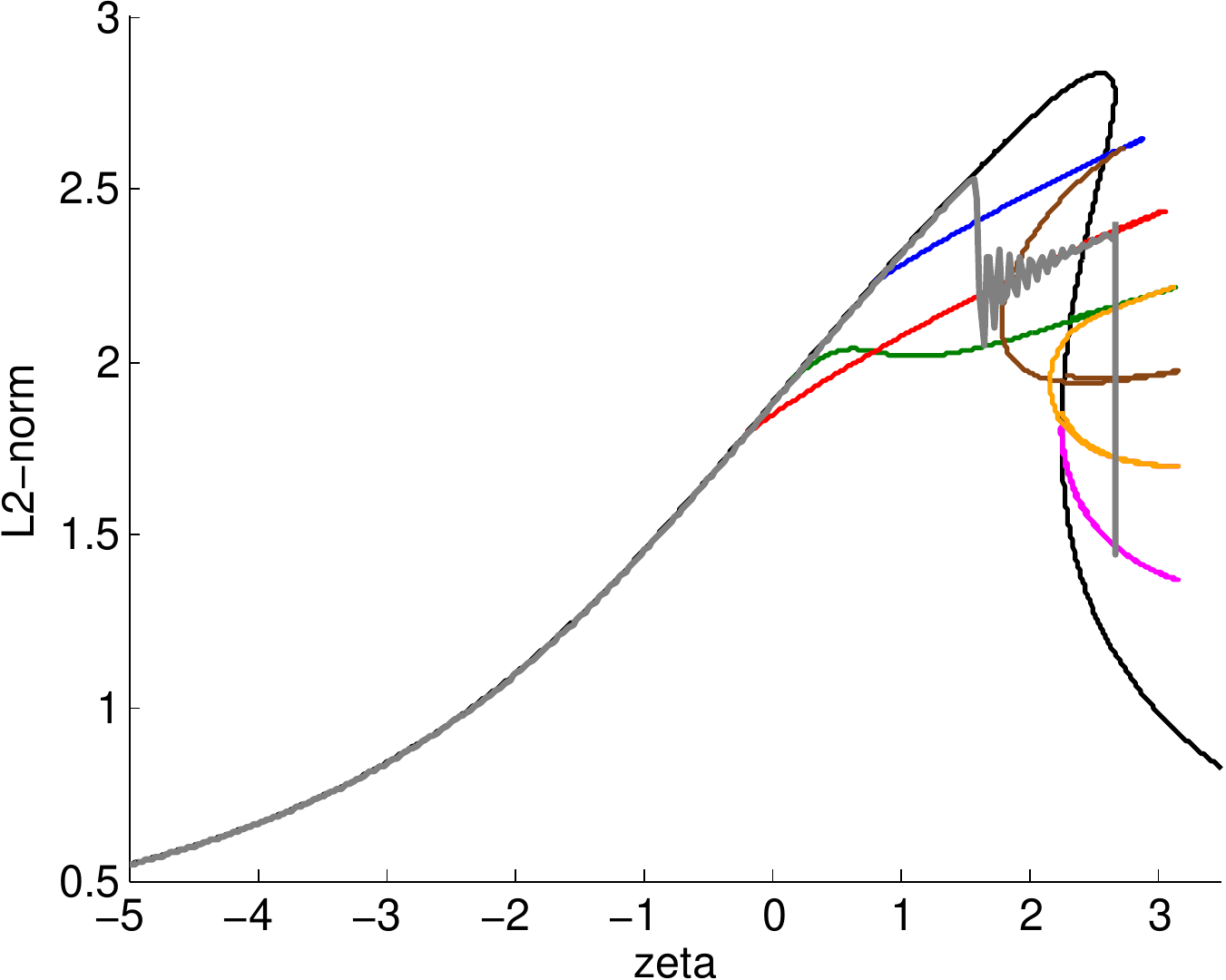} 
\caption{(d) final state: 1 soliton}
\end{subfigure}
\caption{Time evolution of $L^2$-norm in dynamic detuning}
\label{dynamic_detuning_soliton_pic2}
\end{figure} 

At the beginning of the time evolution the solution remains close to the constant one; the grey curve and the black curve practically overlap. When $\zeta(t)$ reaches a certain value between $1$ and $2$ the solution develops a spatially
$\frac{2\pi}{5}$-periodic pattern close to the red curve before passing to a $k$-soliton ($k\in\{1,2,3\}$) where the time evolution becomes numerically stationary. 
Qualitatively very similar time evolutions were observed by Herr et al. \cite{Herr2013},
p.148 for a different set of parameter values.\footnote{As explained in the supplementary material to
\cite{Herr2013} the numerical integration method differs from ours since it is based on a time-dependent
version of the coupled mode equations \eqref{basic.F} instead of the Lugiato-Lefever equation \eqref{t_ll}.}
The strong similarity of time-dependent detuning simulations with our bifurcation diagrams suggests that the
bifurcation diagrams may give important clues on the question how to drive the laser/ring resonator system
into a soliton state. In particular, we conclude that the final soliton states of a time-dependent detuning
approach typically lie on bifurcation curves that we have analytically characterized in Theorem~\ref{Thm 3
bifurcation zeta} and Theorem~\ref{Thm 4 bifurcation f}.


\section{Conclusions} \label{conclusions}
  
  Let us finally summarize our results from the point of view of their applicability. A first outcome of our
  analysis is that the search for frequency combs can from now on be reduced to specific parameter regimes.
  Theorem~\ref{Thm 1 A priori estimates} shows for large $|f|$ and small $|d|$ that, roughly speaking,
  frequency combs satisfy the estimate $\|a\|_\infty = O(|f|^3 |d|^{-1})$ uniformly with respect to $\zeta$.
  In particular this result gives an upper bound for the maximal amplitude that one can expect for a highly
  localized soliton. In the case of normal dispersion $d<0$ Theorem~\ref{Thm 2 nonexistence} shows that only
  trivial combs exists outside the explicit interval $[\zeta_*, \zeta^\ast]$ of detuning parameters. Here
  $\zeta_*<0<\zeta^*$, and $\zeta^*, -\zeta_*= O( f^6 |d|^{-2})$ for large $f$ and small $|d|$. A similar
  result holds for abnomalous dispersion $d>0$.

  \medskip
  
  Theorem~\ref{Thm 3 bifurcation zeta} and~\ref{Thm 4 bifurcation f} prove the existence of bifurcating
  branches of frequency combs and provide explicit formulae for the parameter values at which these branches
  emanate from the trivial ones. This is the analytical justification for what AUTO is successfully
  doing, i.e., finding bifurcation points and following bifurcation branches. The resulting bifurcation
  diagrams are quite complicated (cf. Section~\ref{illustrations}) and so far only few theoretical information
  about their structures has become apparent, see the remarks following Theorem~\ref{Thm 3 bifurcation zeta}.
  It remains open to investigate both analytically and numerically points of secondary bifurcation. The strong
  similarity of time-dependent detuning simulations with our bifurcation diagrams shows that the final
  soliton states of a laser/ring resonator system cannot be understood without the bifurcation phenomena that
  are the core of this paper.
  
  \medskip
  
  The use of AUTO opens new possibilities to numerically compute the shape of specific frequency
  combs, and it allows to make observations and conjectures. Let us list two of them.
  
  \begin{itemize}
  \item We observed that in the bifurcation diagrams with respect to $f$ the two solitons of
  Figure~\ref{solitons_f} appeared at turning points of the bifurcation curves, cf.
  Figure \ref{bif_diagram2_zeta=10}. The same happens in bifurcation diagrams with respect
  of $\zeta$ from Figure~\ref{bif_diagram_f=1.6} and Figure~\ref{bif_diagram_f=2}. Also here the solutions
  with most pronounced soliton character, cf.  Figures~\ref{many_solitons}, \ref{many_solitons2},
  \ref{solitons_zeta}, lie on turning points of the branches. It will be worthwhile to investigate further
  if this observation is true in more general cases.
  \item Soliton combs are particularly important in applications and one would like to drive the pump-resonator
  system into such a state. Since all solitons were
  found on bifurcation curves connected to the trivial states time-dependent detuning
   seems to be a feasible way to eventually reach a soliton as outlined in Section~\ref{dynamic}. 
  \end{itemize}
  
  We are aware that the last remark immediately leads to the important question about stability of solutions
  along branches. To the best of our knowledge AUTO does not offer a straightforward option that provides
  stability information for non-parabolic equations like \eqref{basic}. Locally near the bifurcation points
  the principle of exchange of stability (cf. \cite{Kielh_bifurcation_theory,CrRab_exchangeofstab}) allows to analytically
  predict the stability or instability of solutions. Although this is an interesting piece of information, we
  refrained from elaborating it because of two reasons: it would have substantially enlarged the exposition
  and its validity is restricted only to small neighborhoods of the bifurcation points.  
  Instead, it would be much more interesting to characterize the stability or instabilty globally along
  the branches. Analytically, this is very challenging and currently out of reach. It will be one of our
  future goals to attack this question numerically.
  
\section*{Acknowledgements}  
 
 Both authors thank J.~G\"artner, T.~Jahnke, Ch.~Koos, P.~Palomo-Marin, J.~Pfeifle, and Ph.~Trocha (all from
 KIT) for fruitful discussions. We are grateful to T.~Jahnke for letting us use two of his MATLAB codes: one that
 postprocesses AUTO data and generates solution plots and one that performs the time integration of \eqref{t_ll} via Strang-splitting. Additionally, both authors gratefully
 acknowledge financial support by the Deutsche Forschungs\-gemeinschaft (DFG) through 
 the research grant MA 6290/2-1 (first author) and CRC 1173 (both authors).
 

\bibliographystyle{plain}
\bibliography{bibliography}

\end{document}